\documentclass[11pt]{amsart}
\usepackage{amsmath}
\usepackage{amssymb}
\usepackage{amscd}

\def\NZQ{\mathbb}               
\def\NN{{\NZQ N}}
\def\QQ{{\NZQ Q}}
\def\ZZ{{\NZQ Z}}

\def\AA{{\NZQ A}}

\newtheorem{Theorem}{Theorem}[section]
\newtheorem{Lemma}[Theorem]{Lemma}

\newtheorem{Proposition}[Theorem]{Proposition}
\newtheorem{Remark}[Theorem]{Remark}

\newtheorem{Definition}[Theorem]{Definition}

\let\epsilon\varepsilon
\let\phi=\varphi
\let\kappa=\varkappa

\def \s {\sigma}
\def \d {\delta}

\begin{document}

\title{A simpler proof of toroidalization of morphisms from 3-folds to surfaces}
\author{Steven Dale Cutkosky}
\thanks{Partially supported by NSF}

\address{Steven Dale Cutkosky, Department of Mathematics,
University of Missouri, Columbia, MO 65211, USA}
\email{cutkoskys@missouri.edu}

\begin{abstract} We give a simpler and more conceptual proof of toroidalization
of morphisms of 3-folds to surfaces, over an algebraically closed field of
characteristic zero. A toroidalization is obtained by performing sequences
of blow ups of nonsingular subvarieties above the domain and range, to make
a morphism toroidal. The original proof of toroidalization of morphisms of
3-folds to surfaces, which appeared in Springer Lecture Notes in Math.
 in 2002 \cite{C3},
is much more complicated. 
\end{abstract}

\maketitle

\section{Introduction}
 Let $\mathfrak k$ be  an algebraically closed field of characteristic zero.
 If $X$ is a nonsingular variety, then the choice
of a simple normal crossings  divisor (SNC divisor) on $X$ makes $X$ into a toroidal variety.

Suppose that $\Phi:X\rightarrow Y$ is a dominant morphism  of nonsingular $\mathfrak k$-varieties, and there is a SNC divisor $D_Y$ on $Y$ such that $D_X=\Phi^{-1}(D_Y)$ is a SNC divisor on $X$. Then $\Phi$ is torodial (with respect to $D_Y$ and $D_X$) if and only if $\Phi^*(\Omega_Y^1(\mbox{log }D_Y))$ is a subbundle of $\Omega_X^1(\mbox{ log }D_X)$ (Lemma 1.5 \cite{C3}). A toroidal morphism can be expressed locally 
by monomials. All of the cases are written down for toroidal morphisms from
a 3-fold to a surface in Lemma 19.3 \cite{C3}.

The toroidalization problem is to determine, given a dominant morphism
$f:X\rightarrow Y$ of  $\mathfrak k$-varieties, if there exists a commutative diagram
$$
\begin{array}{rcl}
X_1&\stackrel{f_1}{\rightarrow}& Y_1\\
\Phi\downarrow&&\downarrow \Psi\\
X&\stackrel{f}{\rightarrow}&Y
\end{array}
$$
such that $\Phi$ and $\Psi$ are products of blow ups of nonsingular subvarieties, $X_1$ and $Y_1$ are nonsingular, and there exist SNC divisors $D_{Y_1}$ on $Y_1$ and $D_{X_1}=f^{*}(D_{Y_1})$ on $X_1$ such that $f_1$ is toroidal (with respect to $D_{X_1}$ and $D_{Y_1}$). 

The toroidalization problem does not have a positive answer in positive characteristic $p$, even for maps of curves; $t=x^p+x^{p+1}$ gives a simple example. 

In characteristic zero, the toroidalization problem has an affirmative answer
if $Y$ is a curve and $X$ has arbitrary dimension; this is really embedded resolution of hypersurface singularities, so follows from resolution of singularities 
\cite{H1} (some  of the simplified proofs are  \cite{BEV}, \cite{BrM} \cite{C6}, \cite{EHa} and  \cite{Hau1}). 
Toroidalization is proven for morphisms from a 3-fold to a surface in \cite{C3} and for the case of a 3-fold to a 3-fold in \cite{C4}.
Detailed  history and  references on the toroidalization problem are given in the introductions to \cite{C3} and \cite{C4}.

We consider the problem of toroidalization as a resolution of singularities type problem.
When the dimension of the base is larger than one,  the problem shares many of
the complexities of resolution of vector fields (\cite{Seid}, \cite{Ca}, \cite{Pan}) and of
resolution of singularities in positive characteristic (some references are \cite{Ab1}, \cite{Ab}, \cite{H2}, \cite{Cos1}, \cite{CP1}, \cite{CP2}, \cite{BV}, \cite{C5}, \cite{DJ}, \cite{Hau}, \cite{Hau2}, \cite{Hau3},  \cite{H9},  \cite{KK}, \cite{T}). In particular, 
natural invariants  do not have a ``hypersurface of maximal contact'' and
are sometimes  not upper semicontinuous.

Toroidalization, locally along a fixed valuation, is proven in all dimensions
and relative dimensions in \cite{C1} and \cite{C2}. 

The proof of toroidalization of a dominant morphism  from a 3-fold to a surface given in  \cite{C3} consists of 2  steps.

The first step is to prove ``strong preparation''. 
Suppose that $X$ is a nonsingular variety, $S$ is a nonsingular surface with
a SNC divisor $D_S$, and $f:X\rightarrow S$ is a dominant morphism
such that $D_X=f^{-1}(D_S)$ is a SNC divisor on $X$ which contains the locus
where $f$ is not smooth. $f$ is strongly prepared if 
$f^*(\Omega_S^2(\mbox{log }D_S))=\mathcal I\mathcal M$ where $\mathcal I\subset \mathcal O_X$ is an ideal sheaf, and $\mathcal M$ is a subbundle of $\Omega_X^2(\mbox{ log }D_X)$ (Lemma 1.7 \cite{C3}). A strongly prepared morphism has nice local forms which are close to being toroidal (page 7 of \cite{C3}).

Strong preparation is the construction of a commutative diagram
$$
\begin{array}{cll}
X_1&&\\
\downarrow&\searrow\\
X&\stackrel{f}{\rightarrow}&S
\end{array}
$$
where $S$ is a nonsingular surface with a SNC divisor $D_S$ such that
$D_X=f^{*}(D_S)$ is a SNC divisor on the nonsingular variety $X$ which contains the locus where $f$ is not smooth, the vertical arrow is a product of blow ups
of nonsingular subvarieties so that $X_1\rightarrow S$ is strongly prepared.
Strong preparation of morphisms from 3-folds to surfaces is proven in Theorem 17.3 of \cite{C3}.

The second step is to prove that  a strongly prepared morphism from a 3-fold to a surface can be toroidalized. This is proven in Sections 18 and 19 of \cite{C3}.

This second step is generalized in \cite{CK} to prove that a strongly prepared morphism from
an $n$-fold to a surface can be toroidalized. Thus to prove toroidalization of
a morphism from an $n$-fold to a surface, it suffices to proof strong preparation. 

The proof of strong preparation in \cite{C3} is extremely complicated, and does
not readily generalize to higher dimensions. The proof of this result occupies 170 pages of \cite{C3}. We mention that that the main invariant considered in this paper, $\nu$, can be interpreted as the adopted order of 
Section 1.2 of \cite{Ca} of the 2-form $du\wedge dv$. 

In this paper, we give a significantly simpler and more conceptual proof of strong preparation of morphisms of 3-folds to surfaces.
It is our hope that this proof can be extended to prove strong preparation 
for morphisms of $n$-folds to surfaces, for $n>3$. The proof is built around a new 
upper semicontinuous invariant
$\sigma_D$, whose value is a natural number or $\infty$. if $\sigma_D(p)=0$ 
for all $p\in X$, then $X\rightarrow S$ is prepared (which is slightly stronger than being strongly prepared). A first step towards obtaining a reduction in $\sigma_D$ is to 
make $X$ 3-prepared, which is achieved  in Section \ref{Section3}. This is a nicer local form, which is proved by making a local reduction to lower dimension.
The proof proceeds  by performing a toroidal morphism above
$X$ to obtain that $X$ is 3-prepared at all points except for a finite number of 1-points. Then general curves through these points lying on $D_X$ are blown up
to achieve 3-preparation everywhere on $X$. if $X$ is 3-prepared at a point $p$,
then there exists an \'etale cover $U_p$ of an affine neighborhood of $p$ and a local toroidal structure $\overline D_p$ at $p$ (which contains $D_X$) such that there exists 
a projective toroidal morphism $\Psi:U'\rightarrow U_p$ such that $\sigma_D$
has dropped everywhere above $p$ (Section \ref{Section4}).  The final step of the proof is to make these
local constructions  algebraic, and  to patch them. This is accomplished in 
Section \ref{Section5}. In Section \ref{Section6} we state and prove strong preparation for morphisms of 3-folds to surfaces (Theorem \ref{TheoremA}) and
toroidalization of morphisms from 3-folds to surfaces (Theorem \ref{TheoremB}).
Important definitions along the way are: 

prepared, Definition \ref{Prep},

1-prepared, Definition \ref{1-prepdef},

2-prepared, after the proof of Proposition \ref{Step2},

3-prepared, Definition \ref{3-prep}.

The author thanks the referee for  their  helpful suggestions for improving the readability of the article.

\section{The invariant $\sigma_D$, 1-preparation and 2-preparation.}\label{Section2}

For the duration of the paper, $\mathfrak k$ will be an algebraically closed field of characteristic zero. A $\mathfrak k$-variety is an integral quasi projective  $\mathfrak k$-scheme.
We will write  curve (over $\mathfrak k$) to mean a 1-dimensional $\mathfrak k$-variety, and similarly for surfaces and 3-folds. We will assume that varieties are
quasi-projective. This is not really a restriction, by the fact that after a sequence of blow ups of nonsingular subvarieties, all varieties satisfy this condition. By a general point of a $\mathfrak k$-variety $Z$,
we will mean a member of a nontrivial  open subset of $Z$ on which some specified good condition holds. When we say that ``$p$ is a point of $X$'' or ``$p\in X$'' we will mean that $p$ is a closed point, unless we indicate otherwise (for instance, by saying that ``$p$ is a generic point of a subvariety $Y$ of $X$'').

A reduced divisor $D$ on a nonsingular variety $Z$ of dimension $n$ is a simple normal crossings divisor (SNC divisor) if all irreducible components of $D$ are nonsingular, and if $p\in Z$, then there exists a regular system of parameters 
$x_1,\ldots, x_n$ in $\mathcal O_{Z,p}$ such that $x_1x_2\cdots x_r=0$ is a local equation of $D$ at $p$, where $r\le n$ is the number of irreducible components of $D$ containing $p$. Two nonsingular subvarieties $X$ and $Y$ intersect transversally at $p\in X\cap Y$ if there exists a regular system of parameters 
$x_1,\ldots,x_n$ in $\mathcal O_{Z,p}$ and subsets $I,J\subset \{1,\ldots,n\}$
such that $\mathcal I_{X,p}= (x_i\mid i\in I)$ and $\mathcal I_{Y,p}= (x_j\mid j\in J)$.

\begin{Definition}\label{1-prepdef} Let $S$ be a nonsingular surface over $\mathfrak k$ with a reduced SNC divisor $D_S$. Suppose that $X$ is a nonsingular 3-fold, and $f:X\rightarrow S$ is a dominant morphism.
$X$ is 1-prepared (with respect to $f$) if $D_X=f^{-1}(D_S)_{red}$ is a SNC divisor on $X$ which contains the locus where $f$ is not smooth, and if  $C_1$, $C_2$ are the two components of $D_S$ whose intersection is nonempty, $T_1$ is a component of $X$ dominating $C_1$ and  $T_2$ is a component of $D_X$ which dominates $C_2$, then $T_1$ and $T_2$ are disjoint.
\end{Definition}

The following lemma is an easy consequence of the main theorem on resolution of singularities.

\begin{Lemma}\label{1-prep} Suppose that $g:Y\rightarrow T$ is a dominant morphism of a 3-fold over $\mathfrak k$ to a surface over $\mathfrak k$
and $D_T$ is a 1-cycle on $T$ such that $g^{-1}(D_R)$ contains the locus where $g$ is not smooth.
Then there exists a commutative diagram of morphisms
$$
\begin{array}{rcl}
Y_1&\stackrel{g_1}{\rightarrow}& T_1\\
\pi_1\downarrow&&\downarrow\pi_2\\
Y&\stackrel{g}{\rightarrow}& T
\end{array}
$$
such that the vertical arrows are products of blow ups of nonsingular subvarieties contained in the preimage of $D_T$, 
$Y_1$ and $T_1$ are nonsingular and  $D_{T_1}=\pi_1^{-1}(D_T)$ is a SNC divisor  on $T_1$ such that $Y_1$ is 1-prepared with respect to $g_1$.
\end{Lemma}

For the duration of this paper,  $S$ will be a fixed nonsingular surface over $\mathfrak k$, with a (reduced) SNC divisor $D_S$.
To simplify notation, we will often write $D$ to denote $D_X$, if $f:X\rightarrow S$ is 1-prepared.

Suppose that $X$ is 1-prepared with respect to $f:X\rightarrow S$. 
A nonsingular curve $C$ of $X$  which is contained in $D_X$ makes SNCs with $D_X$ if either $C$ is a 2-curve, or  $C$ contains no 3-points, and if $q\in C$ is a 2-point, then there are regular parameters $x,y,z$ in the local ring $\mathcal O_{X,q}$ such that $xy=0$ is a local equation of $D_X$ at $q$, and $x=z=0$ are local equations of $C$ at $q$.

A permissible blow up of $X$
is the blow up $\pi_1:X_1\rightarrow X$ of a point of $D_X$ or a nonsingular curve contained in $D_X$ which
makes SNCs with $D_X$. Then $D_{X_1}=\pi_1^{-1}(D_X)_{red}=(f\circ\pi_1)^{-1}(X_S)_{red}$ is a SNC divisor on $X_1$ and $X_1$  is
1-prepared with respect to $f\circ\pi_1$.  A permissible curve is a curve  which satisfies these conditions (so its blow up is permissible).

Assume that $X$ is 1-prepared with respect to $D$. We will say that $p\in X$ is a $n$-point (for $D$) if $p$ is on exactly $n$ components of $D$. Suppose $q\in D_S$ and $u,v$ are regular parameters in $\mathcal O_{S,q}$ 
such that either $u=0$ is a local equation of $D_S$ at $q$ or $uv=0$ is a local equation of $D_S$ at $q$.
$u,v$ are called permissible parameters at $q$.

For $p\in f^{-1}(q)$, we have regular parameters $x,y,z$ in $\hat{\mathcal O}_{X,p}$ such that
\begin{enumerate}
\item[1)]
If $p$ is a 1-point,
\begin{equation}\label{1-point}
u=x^a, v=P(x)+x^bF
\end{equation}
where $x=0$ is a local equation of $D$, $x\not\,\mid F$ and $x^bF$ has no terms which are a power of $x$.
\item[2)] If $p$ is a 2-point, after possibly interchanging $u$ and $v$,
\begin{equation}\label{2-point}
u=(x^ay^b)^l, v=P(x^ay^b)+x^cy^dF
\end{equation}
where $xy=0$ is a local equation of $D$, $a,b>0$, $\mbox{gcd}(a,b)=1$, $x,y\not\,\mid F$ and $x^cy^dF$ has no terms which are a power of $x^ay^b$.
\item[3)] If $p$ is a 3-point, after possibly interchanging $u$ and $v$,
\begin{equation}\label{3-point}
u=(x^ay^bz^c)^l, v=P(x^ay^bz^c)+x^dy^ez^fF
\end{equation}
where $xyz=0$ is a local equation of $D$, $a,b,c>0$, $\mbox{gcd}(a,b,c)=1$, $x,y,z\not\,\mid F$ and $x^dy^ez^fF$ has no terms which are a power of $x^ay^bz^c$.
\end{enumerate}

regular parameters $x,y,z$ in $\hat{\mathcal O}_{X,p}$ giving forms (\ref{1-point}), (\ref{2-point}) or (\ref{3-point})
are called permissible parameters at $p$ for $u,v$.

Suppose that $X$ is 1-prepared. We define an ideal sheaf 
$$
\mathcal I = \mbox{ fitting ideal sheaf of the image of $f^*:\Omega^2_S\rightarrow \Omega^2_X(\mbox{log}(D))$}
$$
in ${\mathcal O}_X$. $\mathcal I=\mathcal O_X(-G)\overline{\mathcal I}$ where $G$ is an effective divisor supported on $D$ and 
$\overline {\mathcal I}$ has height $\ge 2$.

Suppose that $E_1,\ldots, E_n$ are the irreducible components of $D$. For $p\in X$, define
$$
\sigma_D(p)=\mbox{order}_{\mathcal O_{X,p}/(\sum_{p\in E_i}\mathcal I_{E_i,p})}
\overline{\mathcal I_p}
\left(\mathcal O_{X,p}/\sum_{p\in E_i}\mathcal I_{E_i,p}\right)\in\NN\cup \{\infty\}.
$$

\begin{Lemma}\label{upsemi} $\sigma_D$ is upper semicontinuous in the Zariski topology of the scheme  $X$.
\end{Lemma}

\begin{proof} For a fixed subset $J\subset \{1,2,\ldots,n\}$, we have that the function
$$
\mbox{order}_{\mathcal O_{X,p}/(\sum_{i\in J}\mathcal I_{E_i,p})}
\overline{\mathcal I_p}
\left(\mathcal O_{X,p}/\sum_{i\in J}\mathcal I_{E_i,p}\right)
$$
is upper semicontinuous, and if $J\subset J'\subset \{1,2,\ldots, n\}$. we have that
$$
\mbox{order}_{\mathcal O_{X,p}/(\sum_{i\in J}\mathcal I_{E_i,p})}
\overline{\mathcal I_p}
\left(\mathcal O_{X,p}/\sum_{i\in J}\mathcal I_{E_i,p}\right)
\le \mbox{order}_{\mathcal O_{X,p}/(\sum_{i\in J'}\mathcal I_{E_i,p})}
\overline{\mathcal I_p}
\left(\mathcal O_{X,p}/\sum_{i\in J'}\mathcal I_{E_i,p}\right).
$$ 
\end{proof}

Thus for $r\in \NN\cup \{\infty\}$,
$$
\mbox{Sing}_r(X)=\{p\in X\mid \sigma_D(p)\ge r\}
$$
is a closed subset of $X$, which is supported on $D$ and has dimension $\le 1$ if $r>0$.

\begin{Definition}\label{Prep}
A point $p\in X$ is prepared  if $\sigma_D(p)=0$.
\end{Definition}

We have that $\sigma_D(p)=0$ if and only if $\overline {\mathcal I}_p=\mathcal O_{X,p}$. Further,
$$
\mbox{Sing}_1(X)=\{p\in X\mid \overline{\mathcal I}_p\ne\mathcal O_{X,p}\}.
$$
If $p\in X$ is a 1-point with an expression (\ref{1-point}) we have
\begin{equation}\label{li1}
(\overline{\mathcal I}_p+(x))\hat{\mathcal O}_{X,p}=(x,\frac{\partial F}{\partial y},\frac{\partial F}{\partial z}).
\end{equation}
If $p\in X$ is a 2-point with an expression (\ref{2-point}) we have
\begin{equation}\label{li2}
(\overline{\mathcal I}_p+(x,y))\hat{\mathcal O}_{X,p}=(x,y,(ad-bc)F,\frac{\partial F}{\partial z}).
\end{equation}
If $p\in X$ is a 3-point with an expression (\ref{3-point}) we have
\begin{equation}\label{li3}
(\overline{\mathcal I}_p+(x,y,z))\hat{\mathcal O}_{X,p}=(x,y,z,(ae-bd)F,(af-cd)F,(bf-ce)F).
\end{equation}

If $p\in X$ is a 1-point with an expression (\ref{1-point}), then $\sigma_D(p)=\mbox{ord }F(0,y,z)-1$.
We  have $0\le \sigma_D(p)<\infty$ if $p$ is a 1-point.
If $p\in X$ is a 2-point, we have 
$$
\sigma_D(p)=
\left\{\begin{array}{ll}
0&\mbox{ if }\mbox{ord }F(0,0,z)=0 \mbox{ (in this case, $ad-bc\ne 0$)}\\
\mbox{ord }F(0,0,z)-1&\mbox{ if }1\le \mbox{ord }F(0,0,z)<\infty\\
\infty&\mbox{ if }\mbox{ord }F(0,0,z)=\infty.
\end{array}
\right.
$$
If $p\in X$ is a 3-point, let 
$$
A=\left(\begin{array}{lll}
a&b&c\\
d&e&f
\end{array}\right).
$$
we have 
$$
\sigma_D(p)=\left\{\begin{array}{ll}
0&\mbox{ if }\mbox{ord }F(0,0,0)=0\mbox{ (in this case, $\mbox{rank}(A)=2$)}\\
\infty & \mbox{ if }\mbox{ord }F(0,0,0)=\infty.
\end{array}\right.
$$

\begin{Lemma}\label{Torgood} Suppose that $X$ is 1-prepared and $\pi_1:X_1\rightarrow X$ is a toroidal morphism with respect to $D$. Then $X_1$ is 1-prepared and $\sigma_D(p_1)\le \sigma_D(p)$ for all $p\in X$ and $p_1\in\pi_1^{-1}(p)$.
\end{Lemma}

\begin{proof}
Suppose that $p\in X$ is a 2-point and $p_1\in\pi_1^{-1}(p)$. Then there exist permissible parameters $x,y,z$ at $p$ giving an expression (\ref{2-point}). In $\hat{\mathcal O}_{X_1,p_1}$, there are regular parameters $x_1,y_1,z$ where
\begin{equation}\label{eq93}
x=x_1^{a_{11}}(y_1+\alpha)^{a_{12}},\,\,
y=x_1^{a_{21}}(y_1+\alpha)^{a_{22}}
\end{equation}
with $\alpha\in \mathfrak k$ and $a_{11}a_{22}-a_{12}a_{22}=\pm1$.
If $\alpha=0$, so that $p_1$ is a 2-point, then $x_1,y_1,z$ are permissible parameters at $p_1$ and  substitution of (\ref{eq93}) into (\ref{2-point}) gives an expression of the form (\ref{2-point}) at $p_1$, showing that $\sigma_D(p_1)\le \sigma_D(p)$. If $\alpha\ne 0\in \mathfrak k$, so that $p_1$ is a 1-point, set
$\lambda=\frac{aa_{12}+ba_{22}}{aa_{11}+ba_{21}}$ and $\overline x_1=x_1(y_1+\alpha)^{\lambda}$. Then $\overline x_1, y_1, z$ are permissible parameters at $p_1$. Substitution into (\ref{2-point}) leads to a form (\ref{1-point}) with $\sigma_D(p_1)\le \sigma_D(p)$.

If $p\in X$ is a 3-point and $\sigma_D(p)\ne \infty$, then $\sigma_D(p)=0$ so that $p$ is prepared.
Thus there exist permissible  parameters $x,y,z$ at $p$ giving an expression (\ref{3-point}) with $F=1$.
Suppose that $p_1\in \pi_1^{-1}(p)$. In $\hat{\mathcal O}_{X_1,p_1}$ there are regular parameters $x_1,y_1,z_1$ such that
\begin{equation}\label{eq94}
\begin{array}{lll}
x&=&(x_1+\alpha)^{a_{11}}(y_1+\beta)^{a_{12}}(z_1+\gamma)^{a_{13}}\\
y&=&(x_1+\alpha)^{a_{21}}(y_1+\beta)^{a_{22}}(z_1+\gamma)^{a_{23}}\\
z&=&(x_1+\alpha)^{a_{31}}(y_1+\beta)^{a_{32}}(z_1+\gamma)^{a_{33}}
\end{array}
\end{equation}
where at least one of $\alpha,\beta,\gamma\in \mathfrak k$ is  zero. Substituting into (\ref{3-point}), we find permissible parameters at $p_1$ giving a prepared form.
\end{proof}

Suppose that $X$ is 1-prepared with respect to $f:X\rightarrow S$. Define 
$$
\Gamma_D(X)=\mbox{max}\{\sigma_D(p)\mid p\in X\}.
$$

\begin{Lemma}\label{Specialize} Suppose that $X$ is 1-prepared and $C$ is a 2-curve of $D$ and there exists $p\in C$ such that $\sigma_D(p)<\infty$. Then $\sigma_D(q)=0$ at the generic point $q$ of $C$.
\end{Lemma}

\begin{proof} If $p$ is a 3-point then $\sigma_D(p)=0$ and the lemma follows from upper semicontinuity of $\sigma_D$.

Suppose that $p$ is a 2-point. If $\sigma_D(p)=0$ then the lemma follows from upper semicontinuity of
$\sigma_D$, so suppose that $0<\sigma_D(p)<\infty$. There exist permissible parameters $x,y,z$ at $p$ giving a form (\ref{2-point}), such that $x,y,z$ are uniformizing parameters on an \'etale cover $U$ of an affine neighborhood of $p$. Thus for $\alpha$ in a Zariski open subset of $\mathfrak k$, $x,y,\overline z=z-\alpha$ are permissible parameters at a 2-point $\overline p$ of $C$. After possibly replacing $U$ with a smaller neighborhood of $p$, we have 
$$
\frac{\partial F}{\partial z}=\frac{1}{x^cy^d}\frac{\partial v}{\partial z}\in \Gamma(U,\mathcal O_X)
$$
and $\frac{\partial F}{\partial z}(0,0,z)\ne 0$. Thus there exists a 2-point $\overline p\in C$ with permissible parameters $x,y,\overline z=z-\alpha$ such that $\frac{\partial F}{\partial z}(0,0,\alpha)\ne 0$,
and thus there is an expression (\ref{2-point}) at $\overline p$
$$
\begin{array}{lll}
u&=&(x^ay^b)^{l}\\
v&=& P_1(x^ay^b)+x^cy^dF_1(x,y,\overline z)
\end{array}
$$
with $\mbox{ord }F_1(0,0,\overline z)=0\mbox{ or }1$, so that $\sigma_D(\overline p)=0$. By upper semicontinuity of $\sigma_D$, $\sigma_D(q)=0$.
\end{proof}

\begin{Proposition}\label{Step2}  Suppose that $X$ is 1-prepared with respect to $f:X\rightarrow S$. Then there exists a toroidal morphism $\pi_1:X_1\rightarrow X$ with respect to $D$, such that $\pi_1$ is a sequence of blow ups of 2-curves and 3-points,
and 
\begin{enumerate}
\item[1)] $\sigma_{D}(p)<\infty$ for all $p\in D_{X_1}$.
\item[2)] $X_1$ is  prepared (with respect to $f_1=f\circ \pi_1:X_1\rightarrow S$) at all 3-points and the generic point of all 2-curves of $D_{X_1}$.
\end{enumerate}
\end{Proposition}

\begin{proof} By upper semicontinuity of $\sigma_D$,  Lemma \ref{Specialize} and Lemma \ref{Torgood}, we must show that if $p\in X$ is a 3-point with $\sigma_D(p)=\infty$ then there exists a toroidal morphism $\pi_1:X_1\rightarrow X$ such that $\sigma_D(p_1)=0$ for all 3-points $p_1\in \pi_1^{-1}(p)$ and if $p\in X$ is a 2-point with $\sigma_D(p)=\infty$ then there exists a toroidal morphism $\pi_1:X_1\rightarrow X$ such that $\sigma_D(p_1)<\infty$ for all 2-points $p_1\in \pi_1^{-1}(p)$.

First suppose that $p$ is a 3-point with $\sigma_D(p)=\infty$. Let $x,y,z$ be permissible parameters at $p$ giving a form (\ref{3-point}). There exist regular parameters $\tilde x,\tilde y, \tilde z$ in $\mathcal O_{X,p}$ and unit series $\alpha,\beta,\gamma\in \hat{\mathcal O}_{X,p}$ such that $x=\alpha\tilde x$, $y=\beta\tilde y$, $z=\gamma\tilde z$. Write $F=\sum b_{ijk}x^iy^jz^k$ with $b_{ijk}\in \mathfrak k$. Let
$I=(\tilde x^i\tilde y^j\tilde z^k\mid b_{ijk}\ne 0)$, an ideal in $\mathcal O_{X,p}$. Since $\tilde x\tilde y\tilde z=0$ is a local equation of $D$ at $p$, there exists a toroidal  morphism $\pi_1:X_1\rightarrow X$
with respect to $D$ such that $I\mathcal O_{X_1,p_1}$ is principal for all $p_1\in \pi_1^{-1}(p)$. At a 3-point $p_1\in \pi_1^{-1}(p)$, there exist permissible parameters $x_1,y_1,z_1$ such that 
$$
\begin{array}{lll}
x&=& x_1^{a_{11}}y_1^{a_{12}}z_1^{a_{13}}\\
y&=& x_1^{a_{21}}y_1^{a_{22}}z_1^{a_{23}}\\
z&=& x_1^{a_{31}}y_1^{a_{32}}z_1^{a_{33}}
\end{array}
$$
with $\mbox{Det}(a_{ij})=\pm 1$. Substituting into (\ref{3-point}), we obtain an expression (\ref{3-point}) at $p_1$, where
$$
\begin{array}{lll}
u&=& (x_1^{a_1}y_1^{b_1}z_1^{c_1})^l\\
v&=& P_1(x_1^{a_1}y_1^{b_1}z_1^{c_1})+x_1^{d_1}y_1^{e_1}z_1^{f_1}F_1
\end{array}
$$
where  $P_1(x_1^{a_1}y_1^{b_1}z_1^{c_1})=P(x^ay^bz^c)$ and 
$$
F(x,y,z)=x_1^{\overline a}y_1^{\overline b}z_1^{\overline c}F_1(x_1,y_1,z_1).
$$
with $x_1^{\overline a}y_1^{\overline b}z_1^{\overline c}$ a generator of $I\hat{\mathcal O}_{X_1,p_1}$
and $F_1(0,0,0)\ne 0$. Thus $\sigma_D(p_1)=0$.

Now suppose that $p$ is a 2-point and $\sigma_D(p)=\infty$. There exist permissible parameters $x,y,z$ at $p$ giving a form (\ref{2-point}). Write $F=\sum a_i(x,y)z^i$, with $a_i(x,y)\in \mathfrak k[[x,y]]$ for all $i$. We necessarily have that no $a_i(x,y)$ is a unit series.

Let $I$ be the ideal $I=(a_i(x,y)\mid i\in \NN)$ in $\mathfrak k[[x,y]]$.  There exists a sequence of blow ups of 2-curves $\pi_1:X_1\rightarrow X$ such that $\hat{\mathcal O}_{X_1,p_1}$ is principal at all 2-points $p_1\in\pi_1^{-1}(p)$.
There exist  $x_1,y_1\in\mathcal O_{X_1,p_1}$ so that $x_1,y_1,z$ are permissible parameters at $p_1$, and
$$
x=x_1^{a_{11}}y_1^{a_{12}},\,\,y=x_1^{a_{21}}y_1^{a_{22}}
$$
with $a_{11}a_{22}-a_{12}a_{21}=\pm 1$. Let $x_1^{\overline a}y_1^{\overline b}$ be a  generator of 
$I\mathcal O_{T_1,q_1}$. Then $F=x_1^{\overline a}y_1^{\overline b}F_1(x_1,y_1,z)$ where $F_1(0,0,z)\ne 0$,
and we have an expression (\ref{2-point}) at $p_1$, where
$$
\begin{array}{lll}
u&=& (x_1^{a_1}y_1^{b_1})^{l_1}\\
v&=& P_1(x_1^{a_1}y_1^{b_1})+x_1^{d_1}y_1^{e_1}F_1
\end{array}
$$
where $P_1(x_1^{a_1}y_1^{b_1})=P(x^ay^b)$. Thus $\sigma_D(p_1)<\infty$ and $\sigma_D(q)<\infty$ if $q$ is the generic point of the 2-curve of $D_{X_1}$ containing $p_1$.

\end{proof}

We will say that  $X$ is 2-prepared (with respect to $f:X\rightarrow S$) if it satisfies the conclusions of Proposition \ref{Step2}.
We then have that $\Gamma_D(X)<\infty$.

If $X$ is 2-prepared, we have that  $\mbox{Sing}_{1}(X)$ is a union of (closed) curves whose generic point is a 1-point and isolated 1-points and 2-points. Further, $\mbox{Sing}_1(X)$ contains no 3-points.

\section{3-preparation}\label{Section3}

\begin{Lemma}\label{localform} Suppose that $X$ is 2-prepared. Suppose that $p\in X$ is such that
$\sigma_D(p)>0$. Let  $m=\sigma_D(p)+1$.
Then there exist permissible parameters $x,y,z$ at $p$ such that there exist
$\tilde x,y\in\mathcal O_{X,p}$, an \'etale cover $U$ of an affine neighborhood of $p$, such that  $x,z\in\Gamma(U,\mathcal O_X)$ and   $x, y, z$ are uniformizing parameters on $U$, and 
$x=\gamma \tilde x$ for some unit series $\gamma\in \hat{\mathcal O}_{X,p}$. 
We have an expression (\ref{1-point}) or (\ref{2-point}), if $p$ is respectively a 1-point or a 2-point, with
\begin{equation}\label{eq1}
F=\tau z^m+ a_2(x,y)z^{m-2}+\cdots+ a_{m-1}(x,y)z+a_m(x,y)
\end{equation}
where $m\ge 2$ and $\tau\in \hat{\mathcal O}_{X_1,p}=\mathfrak k[[x,y,z]]$ is a unit, and $a_i(x,y)\ne 0$ for $i=m-1$ or $i=m$. Further, if $p$ is a 1-point, then we can choose $x,y,z$ so that $x=y=0$ is a local equation of a generic curve through $p$  on $D$.

For all but finitely many points $p$ in the set of 1-points of $X$, there is an expression (\ref{eq1}) where
\begin{equation}\label{eq1+}
\begin{array}{l}
\mbox{$a_i$ is either zero or has an expression $a_i=\overline a_ix^{r_i}$ where $\overline a_i$ is a unit}\\ \mbox{and $r_i>0$ for $2\le i\le m$, and $a_m=0$ or $a_m=x^{r_m}\overline a_m$ where $r_m>0$ and $\mbox{ord}(\overline a_m(0,y))=1$.}
\end{array}
\end{equation}
\end{Lemma}

\begin{proof} There exist regular parameters $\tilde x,y, \overline z$ in $\mathcal O_{X,p}$ and
a unit $\gamma\in\hat{\mathcal O}_{X,p}$ such that $x=\gamma\tilde x, y, \overline  z$ are permissible parameters at $p$,  with 
$\mbox{ord}(F(0,0,\overline z))=m$. 
Thus there exists an affine neighborhood $\mbox{Spec}(A)$ of $p$ such that $V=\mbox{Spec}(R)$, where $R=A[\gamma^{\frac{1}{a}}]$ is an \'etale cover of $\mbox{Spec}(A)$, $x,y,\overline z$ are uniformizing parameters on $V$, and $u,v\in \Gamma(V,\mathcal O_X)$.
Differentiating with respect to the uniformizing parameters
$x,y,\overline z$ in $R$, set
\begin{equation}\label{eqLF2}
\tilde  z=\frac{\partial^{m-1}F}{\partial \overline z^{m-1}}=\omega(\overline z-\phi(x,y))
\end{equation}
where $\omega\in\hat{\mathcal O}_{X,p}$ is a unit series, and $\phi(x,y)\in \mathfrak k[[x,y]]$
is a nonunit series, by the formal implicit function theorem. Set $z=\overline z-\phi(x,y)$.
Since $R$ is normal, after possibly replacing $\mbox{Spec}(A)$ with a smaller affine neighborhood
of $p$, 
$$
\tilde z=\frac{1}{x^b}\frac{\partial^{m-1}v}{\partial \overline z^{m-1}}\in R.
$$
By Weierstrass preparation for Henselian local rings (Proposition 6.1 \cite{R}), $\phi(x,y)$ is integral over the local ring $\mathfrak k[x,y]_{(x,y)}$. Thus after possibly replacing $A$ with a smaller affine neighborhood of $p$, there exists an \'etale cover $U$ of $V$ such that $\phi(x,y)\in\Gamma(U,\mathcal O_X)$, and thus $z\in \Gamma(U,\mathcal O_X)$.

Let $G(x,y,z)=F(x,y,\overline z)$. We have that
$$
G=G(x,y,0)+\frac{\partial G}{\partial z}(x,y,0)z+\cdots+\frac{1}{(m-1)!}\frac{\partial^{m-1}G}{\partial z^{m-1}}(x,y,0)z^{m-1}+
\frac{1}{m!}\frac{\partial^{m}G}{\partial z^{m}}(x,y,0)z^{m}+\cdots
$$
We have
$$
\frac{\partial^{m-1}G}{\partial z^{m-1}}(x,y,0)=
\frac{\partial^{m-1}F}{\partial \overline z^{m-1}}(x,y,\phi(x,y))=0
$$
and
$$
\frac{\partial^{m}G}{\partial z^{m}}(x,y,0)=
\frac{\partial^{m}F}{\partial \overline z^{m}}(x,y,\phi(x,y))
$$
is a unit in $\hat{\mathcal O}_{X,p}$.
Thus we have the desired form (\ref{eq1}), but we must still show that $a_m\ne 0$ or $a_{m-1}\ne 0$.
If $a_i(x,y)=0$ for $i=m$ and $i=m-1$, we have that $z^2\mid F$ in $\hat{\mathcal O}_{X,p}$, since $m\ge 2$.
This implies that the ideal of $2\times 2$ minors
$$
I_2\left(\begin{array}{lll}
\frac{\partial u}{\partial x}& \frac{\partial u}{\partial y} &\frac{\partial u}{\partial z}\\
\frac{\partial v}{\partial x}& \frac{\partial v}{\partial y} &\frac{\partial v}{\partial z}
\end{array}\right)
\subset (z),
$$
which implies that $z=0$ is a component of $D$ which is impossible. Thus either $a_{m-1}\ne 0$ or $a_m\ne 0$.

Suppose that $C$ is a curve in $\mbox{Sing}_1(X)$ (containing a 1-point) and $p\in C$ is a general point. Let
$r=\sigma_D(p)$. Set $m=r+1$. Let $x,y,\overline z$ be permissible parameters at $p$ with $y,\overline z\in\mathcal O_{X,p}$, which are uniformizing parameters on an \'etale cover $U$ of an affine neighborhood of $p$ such that $x=\overline z=0$ are local equations of $C$ and we have a form (\ref{1-point}) at $p$ with
\begin{equation}\label{eqLF1}
F=\tau \overline z^m+a_1(x,y)\overline z^{m-1}+\cdots+a_m(x,y).
\end{equation}
For $\alpha$ in a Zariski open subset of $\mathfrak k$, $x,\overline y=y-\alpha, \overline z$ are permissible parameters at a point $q\in C\cap U$.
For most points $q$ on the curve $C\cap U$,
we have that $a_i(x,y)=x^{r_i}\overline a_i(x,y)$ where $\overline a_i(x,y)$ is a unit or zero for $1\le i\le m-1$ in $\hat{\mathcal O}_{X,q}$. Since $\sigma_D(p)=r$ at this point, we have that $1\le r_i$ for all $i$. We further have that 
if $a_m\ne 0$, then $a_m=x^{r_m}a'$ where $a'=f(y)+x\Omega$ where $f(y)$ is non constant.
Thus 
$$
0\ne \frac{\partial a_m}{\partial y}(0,y)=\frac{\partial F}{\partial y}(0,y,0).
$$
After possibly replacing $U$ with a smaller neighborhood of $p$, we have
$$
\frac{\partial F}{\partial y}=\frac{1}{x^b}\frac{\partial v}{\partial y}\in \Gamma(U,\mathcal O_X).
$$
Thus $\frac{\partial a_m}{\partial y}(0,\alpha)\ne 0$ for most $\alpha\in \mathfrak k$.
Since $r>0$, we have that $r_m>0$, and thus $r_i>0$ for all $i$ in (\ref{eqLF1}). We have
$$
\frac{\partial^{m-1}F}{\partial \overline z^{m-1}}=\xi \overline z+a_1(x,y),
$$
where $\xi$ is a unit series. Comparing the above equation with (\ref{eqLF2}), we observe that $\phi(x,y)$ is a unit series in $x$ and $y$ times $a_1(x,y)$. Thus $x$ divides $\phi(x,y)$. Setting $z=\overline z-\phi(x,y)$, we obtain an expression (\ref{eq1}) such that $x$ divides $a_i$ for all $i$. Now argue as in the analysis of (\ref{eqLF1}), after substituting $z=\overline z-\phi(x,y)$, to conclude that 
there is an expression (\ref{eq1}), where (\ref{eq1+}) holds at most points $q\in C\cap U$. Thus a form (\ref{eq1}) and (\ref{eq1+}) holds at all but
finitely many 1-points of $X$.
 
\end{proof}

\begin{Lemma}\label{LemmaB} Suppose that $X$ is 2-prepared,  $C$ is a curve in $\mbox{Sing}_{1}(X)$ 
containing a 1-point and
$p$ is a general point of $C$. Let $m=\sigma_D(p)+1$. Suppose that $\tilde x, y\in \mathcal O_{X,p}$ are such that
$\tilde x=0$ is a local equation of $D$ at $p$ and the germ $\tilde x=y=0$ intersects $C$ transversally at $p$.
Then there exists an \'etale cover $U$ of an affine neighborhood of $p$ and $z\in \Gamma(U,\mathcal O_X)$ such that $\tilde x, y,  z$ give a form (\ref{eq1}) at $p$.
\end{Lemma}

\begin{proof} There exists $\overline z\in \mathcal O_{X,p}$ such that $\tilde x, y,\overline z$ are regular parameters in $\mathcal O_{X,p}$ and $x=\overline z=0$ is a local equation of $C$ at $p$. There exists a unit
$\gamma\in \hat{\mathcal O}_{X,p}$ such that $x=\gamma\tilde x, y, \overline z$ are permissible parameters
at $p$. We have an expression of the form (\ref{1-point}),
$$
u=x^a,
v=P(x)+x^bF
$$
at $p$. Write $F=f(y,\overline z)+x\Omega$ in $\hat{\mathcal O}_{X,p}$. Let $I$ be the ideal in 
$\hat{\mathcal O}_{X,p}$ generated by $x$ and 
$$
\{\frac{\partial^{i+j}f}{\partial y^i\partial \overline z^j}\mid 1\le i+j\le m-1\}.
$$
The radical of $I$ is the ideal $(x,\overline z)$, as $x=\overline z=0$ is a local equation of
$\mbox{Sing}_{m-1}(X)$ at $p$. Thus
$\overline z$ divides $\frac{\partial^{i+j}}{\partial y^i\partial \overline z^j}$ for $1\le i+j\le m-1$ (with $m\ge 2$).
Expanding
$$
f=\sum_{i=0}^{\infty}b_i(y)\overline z^i
$$
(where $b_0(0)=0$) we see that $\frac{\partial b_0}{\partial y}=0$ (so that $b_0(y)=0$)
and $b_i(y)=0$ for $1\le i\le m-1$. Thus $\overline z^m$ divides $f(y,\overline z)$. Since
$\sigma_D(p)=m-1$, we have that $f=\tau\overline z^m$ where $\tau$ is a unit series. Thus $x,y,\overline z$ 
gives a form (\ref{1-point}) with $\mbox{ord}(F(0,0,\overline z))=m$. Now the proof of Lemma \ref{localform}
gives the desired conclusion.
\end{proof}

Let $\omega(m,r_2,\ldots, r_{m-1})$ be a function which associates a positive integer to a positive integer $m$, natural numbers $r_2,\ldots, r_{m-2}$ and a positive integer $r_{m-1}$. We will give a precise form of $\omega$ after Theorem \ref{1-pointres}.

\begin{Definition}\label{3-prep} $X$ is 3-prepared (with respect to $f:X\rightarrow S$) at a point $p\in D$  if $\sigma_D(p)=0$ or if $\sigma_D(p)>0$, $f$ is 2-prepared
with respect to  $D$ at p and there are permissible parameters $x,y,z$ at $p$  such that
$x,y,z$ are uniformizing parameters on an \'etale cover of an affine neighborhood of $p$ and 
 we have one of the following forms, with $m=\sigma_D(p)+1$:
\begin{enumerate}
\item[1)] $p$ is a 2-point,  and we have an expression (\ref{2-point}) with
\begin{equation}\label{eq2}
F=\tau_0 z^m+ \tau_2x^{r_2}y^{s_2}z^{m-2}+\cdots+ \tau_{m-1}x^{r_{m-1}}y^{s_{m-1}}z+\tau_mx^{r_m}y^{s_m}
\end{equation}
where $\tau_0\in \hat{\mathcal O}_{X,p}$ is a  unit, $\tau_i\in \hat{\mathcal O}_{X,p}$ are units (or zero), $r_i+s_i>0$  whenever $\tau_i\ne 0$ and
$(r_m+c)b-(s_m+d)a\ne 0$. Further, $\tau_{m-1}\ne 0$ or $\tau_m\ne 0$.
\item[2)] $p$ is a 1-point, and we have an expression (\ref{1-point}) with 
\begin{equation}\label{eq3}
F=\tau_0 z^m+\tau_2x^{r_2}z^{m-2}+\cdots+\tau_{m-1}x^{r_{m-1}}z+\tau_mx^{r_m}
\end{equation}
where $\tau_0\in \hat{\mathcal O}_{X,p}$  is a  unit, $\tau_i\in \hat{\mathcal O}_{X,p}$  are units (or zero) for $2\le i\le m-1$,
$\tau_m\in \hat{\mathcal O}_{X,p}$
and $\mbox{ord}(\tau_m(0,y,0))=1$ (or $\tau_m=0$). Further, $r_i>0$ if $\tau_i\ne 0$, and $\tau_{m-1}\ne 0$ or $\tau_m\ne 0$. 
\item[3)] $p$ is a 1-point, and we have an expression (\ref{1-point}) with 
\begin{equation}\label{eq4}
F=\tau_0 z^m+\tau_2x^{r_2}z^{m-2}+\cdots+\tau_{m-1}x^{r_{m-1}}z+x^t\Omega
\end{equation}
where $\tau_0\in \hat{\mathcal O}_{X,p}$  is a  unit, $\tau_i\in \hat{\mathcal O}_{X,p}$  are units (or zero) for $2\le i\le m-1$, $\Omega\in \hat{\mathcal O}_{X,p}$, $\tau_{m-1}\ne 0$ and $t>\omega(m,r_2,\ldots, r_{m-1})$
(where we set $r_i=0$ if $\tau_i=0$). Further, $r_i>0$ if $\tau_i\ne 0$.
\end{enumerate}
$X$ is 3-prepared  if $X$ is 3-prepared for all $p\in X$.
\end{Definition}

\begin{Lemma}\label{2-prep1} Suppose that $X$ is 2-prepared with respect to $f:X\rightarrow S$. Then
there exists a sequence of blow ups of 2-curves $\pi_1:X \rightarrow X_1$ such that $X_1$ is 3-prepared with respect to  $f\circ \pi_1$, except possibly at a finite number of 1-points.
\end{Lemma} 

\begin{proof} The conclusions follow from Lemmas \ref{localform}, \ref{Specialize} and \ref{Torgood}, and the
method of analysis above 2-points of the proof of \ref{Step2}.
\end{proof}

\begin{Lemma}\label{dim2} Suppose that $u,v\in \mathfrak k[[x,y]]$. Let $T_0=\mbox{Spec}(\mathfrak k[[x,y]])$.
Suppose that $u=x^a$ for some $a\in \ZZ_+$, or $u=(x^ay^b)^l$ where $\mbox{gcd}(a,b)=1$ for some $a,b,l\in \ZZ_+$. Let $p\in T_0$ be the maximal ideal $(x,y)$. Suppose that $v\in (x,y)\mathfrak k[[x,y]]$.
Then either  
$v\in \mathfrak k[[x]]$ or there exists a sequence of blow ups of points $\lambda:T_1\rightarrow T_0$ such that for all
$p_1\in \lambda^{-1}(p)$, we have regular parameters $x_1,y_1$ in $\hat{\mathcal O}_{T_1,p_1}$,  regular parameters $\tilde x_1,\tilde y_1$ in $\mathcal O_{T_1,p_1}$ and a unit $\gamma_1\in\hat{\mathcal O}_{T_1,p_1}$ such that  $x_1=\gamma_1\tilde x_1$,   and one of the following holds:
\begin{enumerate}
\item[1)]
$$
u=x_1^{a_1}, v=P(x_1)+x_1^by_1^c
$$
with $c>0$ or
\item[2)] There exists a unit $\gamma_2\in \hat{\mathcal O}_{T_1,p_1}$ such that $y_1=\gamma_2\tilde y_1$ and 
$$
u=(x_1^{a_1}y_1^{b_1})^{\ell_1},
v=P(x_1^{a_1}y_1^{b_1})+x_1^{c_1}y_1^{d_1}
$$
with $\mbox{gcd}(a_1,b_1)=1$ and $a_1d_1-b_1c_1\ne 0$.
\end{enumerate}
\end{Lemma}

\begin{proof} Let
$$
J=\mbox{Det}\left(\begin{array}{ll}
\frac{\partial u}{\partial x}&\frac{\partial u}{\partial y}\\
\frac{\partial v}{\partial x}&\frac{\partial v}{\partial y}
\end{array}\right).
$$
First suppose that $J=0$. Expand $v=\sum\gamma_{ij}x^iy^j$ with $\gamma_{ij}\in \mathfrak k$. If $u=x^a$, then
$\sum j\gamma_{ij}x^iy^{j-1}=0$ implies $\gamma_{ij}=0$ if $j>0$. Thus $v=P(x)\in \mathfrak k[[x]]$.
If $u=(x^ay^b)^l$, then
$$
0=J=lx^{la-1}y^{lb-1}(\sum_{i,j}(ja-ib)\gamma_{ij}x^iy^j)
$$
implies $\gamma_{ij}=0$ if $ja-ib\ne 0$, which implies that $v\in \mathfrak k[[x^ay^b]]$.

Now suppose that $J\ne 0$. Let $E$ be the divisor $uJ=0$ on $T_0$. There exists a sequence of blow ups of points $\lambda:T_1\rightarrow T_0$ such that $\lambda^{-1}(E)$ is a SNC divisor on $T_1$. Suppose that $p_1\in \lambda^{-1}(p)$. There exist regular parameters $\tilde x_1,\tilde y_1$ in $\hat{\mathcal O}_{T_1,p_1}$ such that if 
$$
J_1=\mbox{Det}\left(\begin{array}{ll}
\frac{\partial u}{\partial \tilde x_1}&\frac{\partial u}{\partial \tilde y_1}\\
\frac{\partial v}{\partial \tilde x_1}&\frac{\partial v}{\partial \tilde y_1}
\end{array}\right),
$$
then
\begin{equation}\label{eq91}
u=\tilde x_1^{a_1},\,\,J_1=\delta\tilde x_1^{b_1}\tilde y_1^{c_1}
\end{equation}
where $a_1>0$ and $\delta$ is a unit in $\hat{\mathcal O}_{T_1,p_1}$, or
\begin{equation}\label{eq92}
u=(\tilde x_1^{a_1}\tilde y_1^{b_1})^{l_1},\,\,J_1=\delta\tilde x_1^{c_1}\tilde y_1^{d_1}
\end{equation}
where $a_1,b_1>0$, $\mbox{gcd}(a_1,b_1)=1$ and $\delta$ is a unit in $\hat{\mathcal O}_{T_1,p_1}$.
Expand $v=\sum\gamma_{ij}\tilde x_1^i\tilde y_1^j$ with $\gamma_{ij}\in \mathfrak k$.

First suppose (\ref{eq91}) holds. Then
$$
a_1x_1^{a_1-1}\left(\sum_{i,j}j\gamma_{ij}\tilde x_1^i\tilde y_1^{j-1}\right)=\delta \tilde x_1^{b_1}\tilde y_1^{c_1}.
$$
Thus $v=P(\tilde x_1)+\epsilon \tilde x_1^e\tilde y_1^f$ where $P(\tilde x_1)\in \mathfrak k[[\tilde x_1]]$,
$e=b_1-a_1+a$, $f=c_1+1$ and $\epsilon$ is a unit series. Since $f>0$, we can make a formal change of variables, multiplying $\tilde x_1$ by an appropriate unit series to get the form 1) of the conclusions of the lemma.

Now suppose that (\ref{eq92}) holds. Then
$$
\tilde x_1^{a_1l_1-1}\tilde y_1^{b_1l_1-1}\left(\sum_{ij}(a_1l_1j-b_1l_1i)\gamma_{ij}\tilde x_1^i\tilde y_1^j\right)=\delta\tilde x_1^{c_1}\tilde y_1^{d_1}.
$$
Thus $v=P(\tilde x_1^{a_1}\tilde y_1^{b_1})+\epsilon \tilde x_1^e\tilde y_1^f$, where $P$ is a series in
$\tilde x_1^{a_1}\tilde y_1^{b_1}$, $\epsilon$ is a unit series, $e=c_1+1-a_1l_1$, $f=d_1+1-b_1l_1$.
Since $a_1l_1f-b_1l_1e\ne 0$, we can make a formal change of variables to reach 2) of the conclusions of the lemma.
\end{proof}

\begin{Lemma}\label{pointprep}
Suppose that $X$ is 2-prepared with respect to  $f:X\rightarrow S$. Suppose that $p\in D$ is a 1-point  with  $m=\sigma_D(p)+1>1$.
Let $u,v$ be permissible parameters for $f(p)$ and $x,y,z$ be permissible parameters for $D$ at $p$ such that a form (\ref{eq1}) holds at $p$. 
Let $U$ be an \'etale cover of an affine neighborhood of $p$ such that $x,y,z$ are uniformizing parameters on $U$.
Let $C$ be the curve in $U$ which has local equations $x=y=0$ at $p$.

Let $T_0=\mbox{Spec}(\mathfrak k[x,y])$, $\Lambda_0:U\rightarrow T_0$. Then there exists a sequence of quadratic transforms $T_1\rightarrow T_0$ such that if $U_1=U\times_{T_0}T_1$ and 
$\psi_1:U_1\rightarrow U$ is the induced sequence of blow ups of sections over
$C$, $\Lambda_1:U_1\rightarrow T_1$ is the projection,
  then $U_1$ is 2-prepared with respect to  $f\circ \psi_1$ at all $p_1\in \psi_1^{-1}(p)$. Further, for every point $p_1\in \psi_1^{-1}(p)$, there exist regular parameters $x_1,y_1$ in 
$\hat{\mathcal O}_{T_1,\Lambda_1(p_1)}$ such that $x_1,y_1,z$ are  permissible parameters at $p_1$, and
there exist regular 
parameters $\tilde x_1,\tilde y_1$ in $\mathcal O_{T_1,\Lambda_1(p_1)}$ 
such that if $p_1$ is a 1-point, 
$x_1=\alpha(\tilde x_1,\tilde y_1)\tilde x_1$ where $\alpha(\tilde x_1,\tilde y_1)\in \hat{\mathcal O}_{T_1,\Lambda_1(p_1)}$ is a unit series and $y_1=\beta(\tilde x_1,\tilde y_1)$ with $\beta(\tilde x_1,\tilde y_1)\in \hat{\mathcal O}_{T_1,\Lambda_1(p_1)}$,
and if $p_1$ is a 2-point, then $x_1=\alpha(\tilde x_1,\tilde y_1)\tilde x_1$  and $y_1=\beta(\tilde x_1,\tilde y_1)\tilde y_1$, where $\alpha(\tilde x_1,\tilde y_1), \beta(\tilde x_1,\tilde y_1)\in \hat{\mathcal O}_{T_1,\Lambda_1(p_1)}$ are unit series.
We have one of the following forms:
\begin{enumerate}
\item[1)] $p_1$ is a 2-point,  and we have an expression (\ref{2-point}) with
\begin{equation}\label{eq2'}
F=\tau z^m+ \overline a_2(x_1,y_1)x_1^{r_2}y_1^{s_2}z^{m-2}+\cdots+\overline a_{m-1}(x_1,y_1)x_1^{r_{m-1}}y_1^{s_{m-1}}z+\overline a_m x_1^{r_m}y_1^{s_m}
\end{equation}
where  $\tau\in \hat{\mathcal O}_{U_1,p_1}$ is a unit, $\overline a_i(x_1,y_1)\in \mathfrak k[[x_1,y_1]]$ are units (or zero) for $2\le i\le m-1$, $\overline a_m=0$ or 1 and if $\overline a_m=0$, then $\overline a_{m-1}\ne 0$.
Further, $r_i+s_i>0$ whenever $\overline a_i\ne 0$ and $a (r_m+c)b-(s_m+d)a\ne 0$. 
\item[2)] $p_1$ is a 1-point, and we have an expression (\ref{1-point}) with 
\begin{equation}\label{eq3'}
F=\tau z^m+\overline a_2(x_1,y_1)x_1^{r_2}z^{m-2}+\cdots+\overline a_{m-1}(x_1,y_1)x_1^{r_{m-1}}z+x_1^{r_m}y_1
\end{equation}
where $\tau\in \hat{\mathcal O}_{U_1,p_1}$ is a unit, $\overline a_i(x_1,y_1)\in \mathfrak k[[x_1,y_1]]$ are units (or zero) for $2\le i\le m-1$. Further,  $r_i>0$ (whenever $\overline a_i\ne 0$). 
\item[3)] $p_1$ is a 1-point,  and we have an expression (\ref{1-point}) with
\begin{equation}\label{eq4'}
F=\tau z^m+ \overline a_2(x_1,y_1)x_1^{r_2}z^{m-2}+\cdots+\overline a_{m-1}(x_1,y_1)x_1^{r_{m-1}}z+x_1^ty_1\Omega
\end{equation}
where  $\tau\in \hat{\mathcal O}_{U_1,p_1}$ is a unit, $\overline a_{i}(x_1,y_1)\in \mathfrak k[[x_1,y_1]]$ are units (or zero)  for $2\le i\le m-1$ and $r_i>0$ whenever $\overline a_i\ne 0$. We also have $t>\omega(m,r_2,\ldots, r_{m-1})$. Further,
 $\overline a_{m-1}\ne 0$ and  $\Omega\in \hat{\mathcal O}_{U_1,p_1}$.
\end{enumerate}
\end{Lemma}
 
\begin{proof} Let $\overline p=\Lambda_0(p)$.  Let $T=\{i\mid a_i(x,y)\ne 0\mbox{ and }2\le i<m\}$.
There exists a sequence of blow ups $\phi_1:T_1\rightarrow T_0$ of points over $\overline p$ such that
at all points $q\in \psi_1^{-1}(p)$, we have permissible parameters $x_1,y_1,z$ such that
$x_1,y_1$ are regular parameters in $\hat{\mathcal O}_{T_1,\Lambda_1(q)}$ 
 and 
we have that $u$  is a monomial in $x_1$ and $y_1$ times a unit in $\hat{\mathcal O}_{T_1,\Lambda_1(q)}$, where $g=\prod_{i\in T}a_i(x,y)$.

Suppose that $a_m(x,y)\ne 0$. Let $\overline v=x^ba_m(x,y)$ if (\ref{1-point}) holds and $\overline v= x^cy^da_m(x,y)$ if (\ref{2-point}) holds. We have $\overline v\not\in \mathfrak k[[x]]$ (respectively $\overline v\not\in \mathfrak k[[x^ay^b]]$). Then by  Theorem \ref{dim2} applied to $u,\overline v$, we 
have that there exists a further sequence of blow ups $\phi_2:T_2\rightarrow T_1$ of points over $\overline p$ such that at all points $q\in (\psi_1\circ\psi_2)^{-1}(p)$, we have permissible parameters $x_2,y_2,z$ such that $x_2,y_2$ are regular parameters in $\hat{\mathcal O}_{T_2,\Lambda_2(q)}$ 
 such that $u=0$ is a SNC divisor and   either
$$
u=x_2^{\overline a}, \overline v =  \overline P(x_2)+x_2^{\overline b}\overline y_2^{\overline c}
$$
with $\overline c>0$ or 
$$
u=(x_2^{\overline a}y_2^{\overline b})^t, \overline v =  \overline P(x_2^{\overline a}y_2^{\overline b})+x_2^{\overline c}\overline y_2^{\overline d}
$$
where $\overline a\overline d-\overline b\overline c\ne 0$.

If $q$ is a 2-point, we have thus achieved the conclusions of the lemma. Further, there are only finitely many
1-points $q$ above $p$ on $U_2$ where the conclusions of the lemma do not hold. 
At such a 1-point $q$, $F$ has an expression
\begin{equation}\label{eq40}
F=\tau z^m+\overline a_2(x_2,y_2)x_2^{r_2}y_2^{s_2}z^{m-2}+\cdots+\overline a_{m-1}(x_2,y_2)x_2^{r_{m-1}}y_2^{s_{m-1}}z+\overline a_mx_2^{r_m}y_2^{s_m}
\end{equation}
where $\overline a_m=0\mbox{ or }1$, $\overline a_i$ are units (or zero) for $2\le i\le m$.

Let 
$$
J=I_2\left(\begin{array}{lll}
\frac{\partial u}{\partial x_2}&\frac{\partial u}{\partial y_2}&\frac{\partial u}{\partial z}\\
\frac{\partial v}{\partial x_2}&\frac{\partial v}{\partial y_2}&\frac{\partial v}{\partial z}
\end{array}\right)=x^n(\frac{\partial F}{\partial y_2},\frac{\partial F}{\partial z})
$$
for some positive integer $n$. Since $D$ contains the locus where $f$ is not smooth, we have that
the localization $J_{\mathfrak p}=(\hat{\mathcal O}_{U_2,q})_{\mathfrak p}$, where $\mathfrak p$ is the prime ideal $(y_2,z_2)$ in $\hat{\mathcal O}_{U_2,q}$. 

We compute 
$$
\frac{\partial F}{\partial z}=\overline a_{m-1}x_2^{r_{m-1}}y_2^{s_{m-1}}+\Lambda_1z
$$
and
$$
\frac{\partial F}{\partial y_2}=s_m\overline a_my_2^{s_m-1}x_2^{r_m}+\Lambda_2z
$$
for some $\Lambda_1,\Lambda_2\in\hat{\mathcal O}_{U_2,q}$, to see that either
$\overline a_{m-1}\ne 0$ and $s_{m-1}=0$, or $\overline a_m\ne 0$ and $s_m=1$.

Let $q$ be one of these points, and let $\phi_3:T_3\rightarrow T_2$ be the blow up of $\Lambda_2(q)$.
We then have that the conclusions of the lemma hold in the form (\ref{eq2'}) at the 2-point
which has permissible parameters $x_3,y_3,z$ defined by $x_2=x_3y_3$ and $y_2=y_3$.
At a 1-point which has permissible parameters $x_3,y_3,z$ defined by $x_2=x_3, y_2=x_3(y_3+\alpha)$
with $\alpha\ne 0$, we have that  a form (\ref{eq3'}) holds. Thus the only case where we may possibly have not achieved the conclusions of the lemma is at the 1-point which has permissible parameters $x_3,y_3,z$ defined by $x_2=x_3$ and $y_2=x_3y_3$. 
We continue to blow up, so that there is at most one point where the conclusions of the lemma do not hold. This point is a 1-point, which has permissible parameters $x_3,y_3,z$ where 
$x_2=x_3$ and $y_2=x_3^ny_3$ where we can take $n$ as large as we like.  We thus have a form
\begin{equation}\label{eq8}
u=x_3^a,
v=P(x_3)+x_3^bF_3
\end{equation}
with $F_3=\tau z^m+ \overline b_{2}x_3^{r_2}z^{m-2}+\cdots +\overline b_{m-1}x_3^{r_{m-2}}z+x_3^t\Omega$,
where either $\overline b_i(x_3,y_3)$ is a unit or is zero,  $\overline b_{m-1}\ne 0$,
and $t>\omega(m,r_2,\ldots,r_{m-1})$ if $\overline a_{m-1}\ne 0$ and $s_{m-1}=0$ which is of the form of (\ref{eq4'}),
or we have a form (\ref{eq3'}) (after replacing $y_3$ with $y_3$ times a unit series in $x_3$ and $y_3$) if $\overline a_m\ne 0$ and $s_m=1$.

\end{proof}

\begin{Lemma}\label{algpointprep}
Suppose that $X$ is 2-prepared with respect to  $f:X\rightarrow S$. Suppose that $p\in D$ is a 1-point  with  $\sigma_D(p)>0$. Let $m=\sigma_D(p)+1$.
Let $x,y,z$ be permissible parameters for $D$ at $p$ such that a form (\ref{eq1}) holds at $p$. 

Let notation be as in Lemma \ref{pointprep}. For $p_1\in \psi_1^{-1}(p)$ let 
$\overline r(p_1)=m+1+r_m$, if a form (\ref{eq3'}) holds at $p_1$, and
$$
\overline r(p_1)=
\left\{\begin{array}{ll}
\mbox{max}\{m+1+r_{m}, m+1+s_m\}&\mbox{ if }\overline a_{m}=1\\
\mbox{max}\{m+1+r_{m-1}, m+1+s_{m-1}\}&\mbox{ if }\overline a_{m}=0
\end{array}\right.
$$
if a form (\ref{eq2'})holds at $p_1$. 
Let
$\overline r(p_1)= m+1+r_{m-1}$ 
if a form (\ref{eq4'}) holds at $p_1$.

Let 
$r'=\mbox{max}\{\overline r(p_1)\mid p_1
\in \psi_1^{-1}(p)\}$. Let 
\begin{equation}\label{eq21}
r=r(p)=m+1+r'.
\end{equation}

Suppose that $x^*\in \mathcal O_{X,p}$ is such that 
$x=\overline \gamma x^*$ for some unit $\overline \gamma\in \hat{\mathcal O}_{X,p}$ with
$\overline \gamma\equiv 1\mbox { mod }m_p^r\hat{\mathcal O}_{X,p}$.

Let $V$ be an affine neighborhood of $p$ such that $x^*,y\in\Gamma(V,\mathcal O_X)$, and let $C^*$ be the curve in $V$ which has local equations $x^*=y=0$ at $p$.

Let $T^*_0=\mbox{Spec}(\mathfrak k[x^*,y])$. Then there exists a sequence of blow ups of points $T^*_1\rightarrow T^*_0$ above $(x^*,y)$ such that if $V_1=V\times_{T^*_0}T^*_1$ and 
$\psi^*_1:V_1\rightarrow V$ is the induced sequence of blow ups of sections over
$C^*$, $\Lambda_1^*:V_1\rightarrow T^*_1$ is the projection,
  then $V_1$  is 2-prepared at all $p_1^*\in (\psi_1^*)^{-1}(p)$. Further, for every point $p_1^*\in (\psi_1^*)^{-1}(p)$, there exist 
 $\hat x_1,\overline y_1\in\hat{\mathcal O}_{V_1,p_1^*}$ such that $\hat x_1,\overline y_1,z$ are permissible parameters at $p_1^*$ and
 we have one of the following forms:

\begin{enumerate}
\item[1)] $p_1^*$ is a 2-point,  and we have an expression (\ref{2-point}) with
\begin{equation}\label{eq2*}
F=\overline \tau_0 z^m+ \overline \tau_2  \hat x_1^{r_2}\overline y_1^{s_2}z^{m-2}+\cdots+\overline \tau_{m-1} \hat x_1^{r_{m-1}}\overline y_1^{s_{m-1}}z+\overline\tau_m \hat x_1^{r_m}\overline y_1^{s_m}
\end{equation}
where $\overline \tau_0\in \hat{\mathcal O}_{V_1,p_1^*}$ is a unit, $\overline \tau_i\in \hat{\mathcal O}_{V_1,p_1^*}$ are units (or zero) for $0\le i \le m-1$, $\overline\tau_m$ is zero or 1, 
$\overline\tau_{m-1}\ne 0$ if $\overline\tau_m=0$, $r_i+s_i>0$ if $\overline \tau_i\ne 0$, 
and
$$
(r_m+c)b-(s_m+d)a\ne 0.
$$
\item[2)] $p_1^*$ is a 1-point, and we have an expression (\ref{1-point}) with 
\begin{equation}\label{eq3*}
F=\overline \tau_0 z^m+\overline \tau_2 \hat x_1^{r_2}z^{m-2}+\cdots+\overline \tau_{m-1} \hat x_1^{r_{m-1}}z+\overline \tau_m\hat x_1^{r_m}
\end{equation}
where $\overline \tau_0\in \hat{\mathcal O}_{V_1,p_1^*}$ is a  unit, $\overline \tau_i\in \hat{\mathcal O}_{V_1,p_1^*}$ are units (or zero),  and $\mbox{ord}(\overline \tau_m(0,\overline y_1,0)=1$. Further, $r_i>0$ if $\overline \tau_i\ne 0$.
\item[3)] $p_1^*$ is a 1-point, and we have an expression (\ref{1-point}) with 
\begin{equation}\label{eq4*}
F=\overline \tau_0 z^m+\overline \tau_2 \hat x_1^{r_2}z^{m-2}+\cdots+\overline \tau_{m-1} \hat x_1^{r_{m-1}}z+x_1^t\overline\Omega
\end{equation}
where $\overline \tau_0\in \hat{\mathcal O}_{V_1,p_1^*}$ is a  unit, $\overline \tau_i\in \hat{\mathcal O}_{V_1,p_1^*}$ are units (or zero), $\overline\Omega\in \hat{\mathcal O}_{V_1,p_1^*}$, $\overline \tau_{m-1}\ne 0$ and $t>\omega(m,r_2,\ldots,r_{m-1})$. Further, $r_i>0$ if $\overline \tau_i\ne 0$.
\end{enumerate}
\end{Lemma}

\begin{proof}   The isomorphism
$T_0^*\rightarrow T_0$ obtained by substitution of $x^*$ for $x$
and subsequent base change by the morphism  $T_1\rightarrow T_0$ of Lemma \ref{pointprep}, induces a sequence of
blow ups of points $T_1^*\rightarrow  T_0^*$.
The base change $\psi_1^*:V_1=V\times_{T_0^*}T_1^*\rightarrow V\cong V\times_{ T_0^*}T_0^*$
factors as a sequence of blow ups of sections over $C^*$.
Let $\Lambda_1^*:V_1\rightarrow T_1^*$ be the natural projection.

Let $p_1^* \in (\psi_1^*)^{-1}(p)$, and let $p_1\in \psi_1^{-1}(p)\subset U_1$ be the corresponding point. 
 
  First suppose that $p_1$ has a form (\ref{eq3'}). With the notation of Lemma \ref{pointprep}, we have polynomials $\phi, \psi$ such that
  $$
  x=\phi(\tilde x_1,\tilde y_1),
  y=\psi(\tilde x_1,\tilde y_1)
  $$
  determines the birational extension $\mathcal O_{T_0,p_0}\rightarrow \mathcal O_{T_1,\Lambda_1(p_1)}$,
  and we have a formal change of variables
  $$
  x_1=\alpha(\tilde x_1,\tilde y_1)\tilde  x_1, y_1=\beta(\tilde x_1,\tilde y_1)
  $$
  for some unit series $\alpha$ and series $\beta$. We further have expansions
   $$
   a_i(x,y)=x_1^{r_i}\overline a_i(x_1,y_1)
   $$
   for $2\le i\le m-1$ where $\overline a_i(x_1,y_1)$ are unit series or zero, and
   $$
   a_m(x,y)=x_1^{r_m}y_1.
   $$
   
  We have $x=\overline\gamma x^*$ with $\overline\gamma\equiv 1\mbox{ mod }m_p^r\hat{\mathcal O}_{X,p}$.
   Set $ y^*=y$.  At $p_1^*$, we have regular parameters $x_1^*, y_1^*$ in $\mathcal O_{ T_1^*, \Lambda_1^*(p_1^*)}$ such that 
   $$
    x^*=\phi(x_1^*,y_1^*),
   y^*=\psi(x_1^*,y_1^*),
   $$
   and $x_1^*, y_1^*, \tilde z$ are regular parameters in $\mathcal O_{V_1,\overline p_1^*}$ (recall that $z=\sigma\tilde z$ in Lemma \ref{localform}). 
   We have regular parameters $\overline x_1, \overline y_1, \in \hat{\mathcal O}_{ T_1^*,\Lambda_1^*(p_1^*)}$
   defined by
   $$
   \overline x_1=\alpha(x_1^*,y_1^*)x_1^*,
   \overline y_1=\beta(x_1^*,y_1^*).
   $$
   
   We have $u=x^a=x_1^{a_1}$ where $a_1=ad$ for some $d\in\ZZ_+$. Since $[\alpha(\tilde x_1,\tilde y_1)\tilde x_1]^d=x$, we have that $[\alpha(x_1^*,y_1^*)x_1^*]^d=x^*$.
   Set $\hat x_1=\overline\gamma^{\frac{1}{d}}\overline x_1=\overline\gamma^{\frac{1}{d}}\alpha(x_1^*,y_1^*)x_1^*$.
   We have that $\overline\gamma^{\frac{1}{d}}\alpha(x_1^*,y_1^*)$ is a unit in $\hat{\mathcal O}_{V_1,p_1^*}$, and $x=\hat x_1^d$. Thus $x_1=\hat x_1$ (with an appropriate choice of root $\overline\gamma^{\frac{1}{d}}$). We have $u=\hat x_1^{ad}$, so that $\hat x_1,\overline y_1,z$ are permissible parameters at $p_1^*$.

   For $2\le i\le m-1$, we have
   $$
   a_i(x,y)=a_i(\overline \gamma x^*,y^*)\equiv a_i(x^*,y^*)\mbox{ mod }m_p^r\hat{\mathcal O}_{V,p}
   $$
   and
   $$
   \begin{array}{lll}
   a_i(x^*,y^*)&=&a_i(\phi(x_1^*,y_1^*),\psi(x_1^*,y_1^*))\\
   &=&\overline x_1^{r_i}\overline a_i(\overline x_1,\overline y_1)\\
   &\equiv& x_1^{r_i}\overline a_i(x_1,\overline y_1)\mbox{ mod }m_p^r{\mathcal O}_{V_1,p_1^*}.
   \end{array}
   $$
   We further have

  $$
  a_m(x^*,y^*)\equiv x_1^{r_m}\overline y_1\mbox{ mod }m_p^r\hat{\mathcal O}_{ V_1,p_1^*}.
  $$
  Thus we have expressions
  \begin{equation}\label{eq20}
  \begin{array}{lll}
  u&=&x_1^{da}\\
  v&=& P(x_1^d)+x_1^{bd}P_1(x_1)
  +x_1^{bd}(\overline\tau z^m+x_1^{r_2}\overline a_2(x_1,\overline y_1)z^{m-2}+\cdots +x_1^{r_m} \overline y_1+h)
  \end{array}
  \end{equation}
  where $\overline\tau\in\hat{\mathcal O}_{V_1,p_1^*}$ is a unit series and
  $$
  h\in m_p^r\hat{\mathcal O}_{V_1,p_1^*}\subset (x_1,z)^r.
  $$
  
  Set $s=r-m$, and write
  $$
  \begin{array}{lll}
  h&=&z^m\Lambda_0(x_1,\overline y_1,z)+z^{m-1}x_1^{1+s}\Lambda_{1}(x_1,\overline y_1)+z^{m-2}x_1^{2+s}\Lambda_2(x_1,\overline y_1)
  +\cdots\\
  &&+zx_1^{(m-1)+s}\Lambda_{m-1}(x_1,\overline y_1)+x_1^{m+s}\Lambda_m(x_1,\overline y_1)
  \end{array}
  $$
  with $\Lambda_0\in m_{p_1^*}\hat{\mathcal O}_{ V_1,p_1^*}$ and
  $\Lambda_i\in \mathfrak k[[x_1,\overline y_1]]$ for $1\le i\le m$.
  
  Substituting into (\ref{eq20}), we obtain an expression
  $$
  \begin{array}{lll}
  u&=&x_1^{da}\\
  v&=& P(x_1^d)+x_1^{bd}P_1(x_1)
  +x_1^{bd}(\overline \tau_0z^m+x_1^{r_2}\overline\tau_2z^{m-2}+\cdots+x_1^{r_{m-1}}\overline \tau_{m-1}z+x_1^{r_m}\overline \tau_m)
  \end{array}
  $$
  where $\overline\tau_0\in\hat{\mathcal O}_{V_1,p_1^*}$ is a  unit, $\overline\tau_i\in\hat{\mathcal O}_{V_1,p_1^*}$ are units (or zero),
  for $1\le i\le m-1$ and $\overline \tau_m\in \mathfrak k[[x_1,\overline y_1]]$ with $\mbox{ord}(\overline\tau_m(0,\overline y_1))=1$.

  We have $\overline \tau_0=\overline\tau+\Lambda_0$,
  $\tau_i=\overline a_i(x_1,\overline y_1)$ for $2\le i\le m-1$, and
  $$
  \overline\tau_m=\overline y_1+z^{m-1}x_1^{1+s-r_m}\Lambda_{1}(x_1,\overline y_1)+\cdots
  +x_1^{m+s-r_m}\Lambda_m(x_1,\overline y_1)).
  $$
  We thus have the desired form
  (\ref{eq3*}).
  
  In the case when $p_1$ has a form (\ref{eq4'}), a similar argument to the analysis of (\ref{eq3'}) shows that $p_1^*$ has a form  (\ref{eq4*}).
  
  Now suppose that $p_1$ has a form (\ref{eq2'}). We then have
  \begin{equation}\label{eq62}
  m_p\mathcal O_{U_1,p_1}\subset (x_1y_1,z)\mathcal O_{U_1,p_1},
  \end{equation}
  unless there exist regular parameters $x_1',y_1'\in \mathcal O_{T_1,\Lambda_1(p_1)}$ such that
  $x_1',y_1',z$ are regular parameters in $\mathcal O_{U_1,p_1}$ and
  \begin{equation}\label{eq60}
  x=x_1', y=(x_1')^ny_1'
  \end{equation}
  or
  \begin{equation}\label{eq61}
  x=x_1'(y_1')^n, y=y_1'
  \end{equation}
  for some $n\in\NN$. If (\ref{eq60}) or (\ref{eq61}) holds, then $\hat{\mathcal O}_{V_1,p_1^*}=\hat{\mathcal O}_{U_1,p_1}$, and (taking  $\hat x_1=x_1$, $\overline y_1=y_1$) we have that a form (\ref{eq2*}) holds at $p_1^*$. We may thus assume that (\ref{eq62}) holds.

  With the notation of Lemma \ref{pointprep}, we have polynomials $\phi, \psi$ such that
  $$
  x=\phi(\tilde x_1,\tilde y_1),
  y=\psi(\tilde x_1,\tilde y_1)
  $$
  determines the birational extension $\mathcal O_{T_0,p_0}\rightarrow \mathcal O_{T_1,\Lambda_1(p_1)}$,
  and we have a formal change of variables
  $$
  x_1=\alpha(\tilde x_1,\tilde y_1)\tilde  x_1, y_1=\beta(\tilde x_1,\tilde y_1)\tilde y_1
  $$
  for some unit series $\alpha$ and $\beta$. We further have expansions
   $$
   a_i(x,y)=x_1^{r_i}y_1^{s_i}\overline a_i(x_1,y_1)
   $$
   for $2\le i\le m-1$ where $\overline a_i(x_1,y_1)$ are unit series or zero, and
   $$
   a_m(x,y)=x_1^{r_m}y_1^{s_m}\overline a_m,
   $$
   where $\overline a_m=0$ or 1.
  We have $x=\overline\gamma x^*$ with $\overline\gamma\equiv 1\mbox{ mod }m_p^r\hat{\mathcal O}_{X,p}$.
   Set $ y^*=y$.  At $p_1^*$, we have regular parameters $x_1^*, y_1^*$ in $\mathcal O_{ T_1^*, \Lambda_1^*(p_1^*)}$ such that 
   $$
    x^*=\phi(x_1^*,y_1^*),
   y^*=\psi(x_1^*,y_1^*),
   $$
   and $x_1^*, y_1^*, \tilde z$ are regular parameters in $\mathcal O_{V_1,\overline p_1^*}$ (recall that $z=\sigma\tilde z$ in Lemma \ref{localform}). 
   We have regular parameters $\overline x_1, \overline y_1, \in \hat{\mathcal O}_{ T_1^*,\Lambda_1^*(p_1^*)}$
   defined by
   $$
   \overline x_1=\alpha(x_1^*,y_1^*)x_1^*,
   \overline y_1=\beta(x_1^*,y_1^*)y_1^*.
   $$
  We calculate
  $$
  u=x^a=(x_1^{a_1}y_1^{b_1})^{t_1}=[\alpha(\tilde x_1,\tilde y_1)\tilde x_1]^{a_1t_1}[\beta(\tilde x_1,\tilde y_1)\tilde y_1]^{b_1t_1}
  $$
  which implies
  $$
  (x^*)^a=[\alpha(x_1^*,y_1^*)x_1^*]^{a_1t_1}[\beta(x_1^*,y_1^*)y_1^*]^{b_1t_1}=\overline x_1^{a_1t_1}\overline y_1^{b_1t_1}.
  $$
  Set $\hat x_1=\overline\gamma^{\frac{a}{a_1t_1}}\overline x_1$ to get
  $u=(\hat x_1^{a_1}\overline y_1^{b_1})^{t_1}$, 
  so that $\hat x_1,\overline y_1,z$ are permissible parameters at $p_1^*$.

   For $2\le i\le m$, we have
   $$
   a_i(x,y)=a_i(\overline \gamma x^*,y^*)\equiv a_i(x^*,y^*)\mbox{ mod }m_p^r\hat{\mathcal O}_{V,p}
   $$
   and
   $$
   \begin{array}{lll}
   a_i(x^*,y^*)&=&a_i(\phi(x_1^*,y_1^*),\psi(x_1^*,y_1^*))\\
   &=&\overline x_1^{r_i}\overline y_1^{s_i}\overline a_i(\overline x_1,\overline y_1)\\
   &\equiv& \hat x_1^{r_i}\overline y_1^{s_i}\overline a_i(\hat x_1,\overline y_1)\mbox{ mod }m_p^r{\mathcal O}_{V_1,p_1^*}.
   \end{array}
   $$
   
  Thus we have expressions
  \begin{equation}\label{eq201}
  \begin{array}{lll}
  u&=&(\hat x_1^{a_1}\overline y_1^{b_1})^{t_1}\\
  v&=& P((\hat x_1^{a_1}\overline y_1^{b_1})^{\frac{t_1}{a}})+(\hat x_1^{a_1}\overline y_1^{b_1})^{\frac{t_1}{a}b}P_1(\hat x_1^{a_1}\overline y_1^{b_1})
  +(\hat x_1^{a_1}\overline y_1^{b_1})^{\frac{t_1}{a}b}(\overline\tau z^m\\
  &&+\hat x_1^{r_2}\overline y_1^{s_2}\overline a_2(\hat x_1,\overline y_1)z^{m-2}+\cdots +\hat x_1^{r_m}\overline y_1^{s_m}\overline a_m+h)
  \end{array}
  \end{equation}
  where $\overline\tau\in\hat{\mathcal O}_{V_1,p_1^*}$ is a unit series and
  $$
  h\in m_p^r\hat{\mathcal O}_{V_1,p_1^*}\subset (\hat x_1\overline y_1,z)^r.
  $$
  
  Set $s=r-m$, and write
  \begin{equation}\label{eq202}
  \begin{array}{lll}
  h&=&z^m\Lambda_0(x_1,\overline y_1,z)+z^{m-1}(\hat x_1\overline y_1)^{1+s}\Lambda_{1}(\hat x_1,\overline y_1)+z^{m-2}(\hat x_1\overline y_1)^{2+s}\Lambda_2(\hat x_1,\overline y_1)
  +\cdots\\
  &&+z(\hat x_1\overline y_1)^{(m-1)+s}\Lambda_{m-1}(\hat x_1,\overline y_1)+(\hat x_1\overline y_1)^{m+s}\Lambda_m(\hat x_1,\overline y_1)
  \end{array}
  \end{equation}
  with $\Lambda_0\in m_{p_1^*}\hat{\mathcal O}_{V_1,p_1^*}$ and
  $\Lambda_i\in \mathfrak k[[\hat x_1,\overline y_1]]$ for $1\le i\le m$.
  
  First suppose that $\overline a_m=1$.
  Substituting into (\ref{eq201}), we obtain an expression
  $$
   \begin{array}{lll}
  u&=&(\hat x_1^{a_1}\overline y_1^{b_1})^{t_1}\\
  v&=& P((\hat x_1^{a_1}\overline y_1^{b_1})^{\frac{t_1}{a}})+(\hat x_1^{a_1}\overline y_1^{b_1})^{\frac{t_1}{a}b}P_1(\hat x_1^{a_1}\overline y_1^{b_1})\\
  &&+(\hat x_1^{a_1}\overline y_1^{b_1})^{\frac{t_1}{a}b}(\overline\tau_0 z^m+\hat x_1^{r_2}\overline y_1^{s_2}\overline \tau_2z^{m-2}+\cdots +\hat x_1^{r_m}\overline y_1^{s_m}\overline \tau_m)
  \end{array}
  $$

  where $\overline\tau_0, \overline\tau_m\in\hat{\mathcal O}_{V_1,p_1^*}$ are  units, $\overline\tau_i\in\hat{\mathcal O}_{V_1,p_1^*}$ are units (or zero) for $2\le i\le m-1$.

  We have $\overline \tau_0=\overline\tau+\Lambda_0$,
  $\tau_i=\overline a_i(\hat x_1,\overline y_1)$ for $2\le i\le m-1$, and
  $$
  \overline\tau_m=\overline a_m+z^{m-1}\hat x_1^{1+s-r_m}\overline y_1^{1+s-s_m}\Lambda_{1}(\hat x_1,\overline y_1)+\cdots
  +\hat x_1^{m+s-r_m}\overline y_1^{m+s-s_m}\Lambda_m(\hat x_1,\overline y_1).
  $$
  We thus have the desired form
  (\ref{eq2*}).
  
  Now suppose that $\overline a_m=0$. Then $\overline a_{m-1}\ne 0$, and $z$ divides $h$ in (\ref{eq201}), so that $\Lambda_m=0$ in (\ref{eq202}).
  Substituting into (\ref{eq201}), we obtain an expression
  $$
   \begin{array}{lll}
  u&=&(\hat x_1^{a_1}\overline y_1^{b_1})^{t_1}\\
  v&=& P((\hat x_1^{a_1}\overline y_1^{b_1})^{\frac{t_1}{a}b})+(\hat x_1^{a_1}\overline y_1^{b_1})^{\frac{t_1}{a}b}P_1(\hat x_1^{a_1}\overline y_1^{b_1})\\
  &&+(\hat x_1^{a_1}\overline y_1^{b_1})^{\frac{t_1}{a}b}(\overline\tau_0 z^m+\hat x_1^{r_2}\overline y_1^{s_2}\overline \tau_2z^{m-2}+\cdots + \hat x_1^{r_{m-1}}\overline y_1^{s_{m-1}}\overline\tau_{m-1}z)
  \end{array}
  $$

  where $\overline\tau_0, \overline\tau_{m-1}\in\hat{\mathcal O}_{V_1,p_1^*}$ are  units, $\overline\tau_i\in\hat{\mathcal O}_{V_1,p_1^*}$ are units (or zero) for $2\le i\le m-2$.

  We have $\overline \tau_0=\overline\tau+\Lambda_0$,
  $\tau_i=\overline a_i(\hat x_1,\overline y_1)$ for $2\le i\le m-2$, and
  $$
  \overline\tau_{m-1}=\overline a_{m-1}+z^{m-1}\hat x_1^{1+s-r_{m-1}}\overline y_1^{1+s-s_{m-1}}\Lambda_{1}(\hat x_1,\overline y_1)+\cdots
  +\hat x_1^{m-1+s-r_{m-1}}\overline y_1^{m-1+s-s_{m-1}}\Lambda_{m-1}(\hat x_1,\overline y_1).
  $$
  
  We thus have the  form
  (\ref{eq2*}).

\end{proof}

\begin{Lemma}\label{1-ptprep} Suppose that $X$ is 2-prepared.  Suppose that  $p\in X$ is a 1-point 
with $\sigma_D(p)>0$ and $E$ is the component of $D$ containing $p$. Suppose that $Y$ is a finite set of points in $X$ (not containing $p$). Then 
there exists an affine neighborhood $U$ of $p$ in $X$ such that
\begin{enumerate}
\item[1)] $Y\cap U=\emptyset$.
\item[2)] $[E-U\cap E]\cap \mbox{Sing}_{1}(X)$ is a finite set of points.
\item[3)] $U\cap D=U\cap E$ and there exists $\overline x\in \Gamma(U,\mathcal O_X)$ such that
$\overline x=0$ is a local equation of $E$ in $U$.
\item[4)] There exists an \'etale map $\pi:U\rightarrow \AA_k^3=\mbox{Spec}(\mathfrak k[\overline x,\overline y,\overline z])$.
\item[5)] The Zariski closure $C$ in $X$ of the curve in $U$ with local equations $\overline x=\overline y=0$ satisfies the following:
\begin{enumerate}
\item[i)] $C$ is a nonsingular curve through $p$.
\item[ii)] $C$ contains no 3-points of $D$.
\item[iii)] $C$ intersects 2-curves of $D$ transversally at prepared points.
\item[iv)] $C\cap \mbox{Sing}_1(X)\cap (X-U)=\emptyset$.
\item[v)] $C\cap Y=\emptyset$.
\item[vi)] $C$ intersects $\mbox{Sing}_1(X)-\{p\}$ transversally at general points of curves in $\mbox{Sing}_1(X)$.
\item[vii)] There exist permissible parameters $x,y,z$ at $p$, with $\tilde x=\overline x, y=\overline y$, which satisfy the hypotheses of lemma \ref{localform}. 
\end{enumerate}
\end{enumerate}
\end{Lemma}

\begin{proof} Let $H$ be an effective, very ample divisor on $X$ such that $H$ contains $Y$ and $D-E$, but $H$ does not contain $p$ and does not contain any one dimensional components of $\mbox{Sing}_1(X,D)\cap E$. There exists $n>0$ such that $E+nH$ is ample, $\mathcal O_X(E+nH)$ is generated by global sections and a general member  $H'$ of the linear system $|E+nH|$ does not contain any one dimensional components of $\mbox{Sing}_1(X,D)\cap E$, and does not contain $p$. $H+H'$ is ample, so $V=X-(H+H')$ is affine. Further, there exists $f\in \mathfrak k(X)$, the function field of $X$, such that $(f)=H'-(E+nH)$. Thus $\overline x=\frac{1}{f}\in\Gamma(V,\mathcal O_X)$ as $X$ is normal and $\overline x$ has no poles on $V$. $\overline x=0$ is a local equation of $E$ on $V$. We have that $V$ satisfies the conclusions  1), 2)  and 3) of the lemma.

Let $R=\Gamma(V,\mathcal O_X)$. $R=\cup_{s=1}^{\infty}\Gamma(X,\mathcal O_X(s(H+H'))$ is a finitely
generated $\mathfrak k$-algebra. Thus for $s\gg 0$, $R$ is generated by $\Gamma(X,\mathcal O_X(s(H+H'))$ as a $\mathfrak k$-algebra.

From the exact sequences
$$
0\rightarrow \Gamma(X,\mathcal O_X(s(H+H'))\otimes\mathcal I_p)\rightarrow \Gamma(X,\mathcal O_X(s(H+H'))\rightarrow \mathcal O_{X,p}/m_p\cong k
$$
and the fact that $1\in\Gamma(X,\mathcal O_X(s(H+H'))$, we have that $R$ is generated by $\Gamma(X,\mathcal O_X(s(H+H'))\otimes \mathcal I_P)$ as a $\mathfrak k$-algebra for all $s\gg0$.

For $s\gg 0$, and a general member $\sigma$ of $\Gamma(X,\mathcal O_{X}(s(H+H'))\otimes\mathcal I_p)$ we have that the curve $\overline C= B\cdot E$, where $B$ is the divisor $B =(\sigma)+s(H+H')$, satisfies 
the conclusions of 5) of the lemma; since each of the conditions 5i) through 5vii) is an open condition on $\Gamma(X,\mathcal O_X(s(H+H')\otimes\mathcal I_p))$, we need only establish that each condition holds on a nonempty subset. This follows from the fact that $H+H'$ is ample,  Bertini's theorem
applied to the base point free linear system $|\phi^*(s(H+H'))-A|$, where $\phi:W\rightarrow X$ is the blow up of $p$
with exceptional divisor $A$, and the fact that 
$$
\phi_*(\mathcal O_W(\phi^*(s(H+H')-A))=\mathcal O_X(s(H+H'))\otimes\mathcal I_p.
$$

For fixed $s\gg 0$,
let $\overline x,\overline y_1,\ldots,\overline y_n$ be a $\mathfrak k$-basis of $\Gamma(X,\mathcal O_X(s(H+H'))\otimes\mathcal I_p)$, so that $R=\mathfrak k[\overline x, \overline y_1,\ldots,\overline y_n]$. We have shown that there exists a Zariski open
set $\overline Z$ of $k^n$ such that for $(b_1,\ldots, b_n)\in \overline Z$, the curve $C$ in $X$ which is the Zariski closure of the curve with local equation $\overline x=b_1\overline y_1+\cdots+b_n\overline y_n=0$ in $V$ satisfies 5) of the conclusions of the lemma.

Let $C_1,\ldots, C_t$ be the curves in $\mbox{Sing}_1(X)\cap V$, and let $p_i\in C_i$ be  closed points
such that $p,p_1,\ldots,p_t$ are distinct. Let $Q_0$ be the maximal ideal of $p$ in $R$, and $Q_i$ be the maximal ideal
in $R$ of $p_i$ for $1\le i\le t$. We have that $\overline x$ is nonzero in $Q_i/Q_i^2$ for all $i$. For a matrix $A=(a_{ij})
\in \mathfrak k^{2n}$, and $1\le i\le 2$, let 
$$
L_i^A(\overline y_1,\ldots,\overline y_n)=\sum_{j=1}^na_{ij}\overline y_j.
$$
There exist $\alpha_{jk}\in \mathfrak k$ such that $Q_k=(\overline y_1-\alpha_{1,k},\ldots,\overline y_n-\alpha_{n,k})$
for $0\le k\le t$. By our construction, we have $\alpha_{1,0}=\cdots=\alpha_{n,0}=0$. For each
$0\le k\le t$, there exists a non empty Zariski open subset $Z_k$ of $k^{2n}$ such that 
$$
\overline x,L_1^A(\overline y_1,\ldots,\overline y_n)-L_1^A(\alpha_{1,k},\ldots,\alpha_{n,k}),
L_{2}^A(\overline y_1,\ldots,\overline y_n)-L_{2}^A(\alpha_{1,k},\ldots,\alpha_{n,k})
$$
is a $\mathfrak k$-basis of $Q_k/Q_{k+1}^2$. Suppose $(a_{1,1},\ldots, a_{1,n})\in\overline Z$ and
$A\in Z_0\cap\cdots \cap Z_t$. 

We will show that $\overline x, L_1^A, L_{2}^A$ are algebraically independent over $\mathfrak k$. 
Suppose not. Then there exists a nonzero polynomial $h\in \mathfrak k[t_1,t_2, t_{3}]$ such that
$h(\overline x, L_1^A, L_{2}^A)=0$. Write $h=H+h'$ where $H$ is the leading form of $h$, and $h'=h-H$ is  a polynomial of larger order than the degree $r$ of $H$. 
Now $H(\overline x,L_1^A, L_{2}^A)=-h'(\overline x,L_1^A, L_{2}^A)$, so that 
$H(\overline x,L_1^A, L_{2}^A)=0$ in $Q_0^r/Q_0^{r+1}$. Thus $H=0$, since $R_{Q_0}$ is a regular local ring,
which is a contradiction. Thus $\overline x,L_1^A, L_2^A$ are algebraically independent. Without loss of generality, we may assume that $L_i^A=\overline y_i$ for $1\le i\le 2$.  

Let $S=\mathfrak k[\overline x, \overline y_1,\overline y_2]$, a polynomial ring in $3$ variables over $\mathfrak k$.
$S\rightarrow R$ is unramified at $Q_i$ for $0\le i\le t$ since 
$$
(\overline x,\overline y_1-\alpha_{1,i},\overline y_{2}-\alpha_{2,i})R_{Q_i}=Q_iR_{Q_i}
$$
for $0\le i\le t$.

Let $W$ be the closed locus in $V$ where $V\rightarrow \mbox{Spec}(S)$ is not \'etale. 
We have that $p,p_1,\ldots,p_t\not\in W$, so there exists an ample effective divisor $\overline H$ on $X$ such that 
$W\subset  \overline H$ and $p,p_1,\ldots,p_t\not\in \overline H$. Let $U=V-\overline H$. $U$ is affine,
and $U\rightarrow \mbox{Spec}(S)\cong \AA^3$ is \'etale, so satisfies 4) of the conclusions of the lemma.

\end{proof}

\begin{Lemma}\label{LemmaA} Suppose $X$ is 2-prepared with respect to $f:X\rightarrow S$, $p\in D$ is a  prepared point, and $\pi_1:X_1\rightarrow X$ is the blow up of $p$. Then all points of $\pi_1^{-1}(p)$ are prepared.
\end{Lemma}

\begin{proof} The conclusions follow from substitution of local equations of the blow up of a point into a prepared form (\ref{1-point}), (\ref{2-point}) or (\ref{3-point}).
\end{proof}

\begin{Lemma}\label{LemmaC} Suppose that $X$ is 2-prepared with respect to $f:X\rightarrow S$, and that $C$
is a permissible curve for $D$, which is not a 2-curve. Suppose that $p\in C$ satisfies $\sigma_D(p)=0$. Then there exist permissible parameters $x,y,z$ at $p$ such that one of the following forms hold:
\begin{enumerate}
\item[1)] $p$ is a 1-point of $D$ of the form of (\ref{1-point}), $F=z$ and $x=y=0$ are formal local equations of $C$ at $p$.
\item[2)] $p$ is a 1-point of $D$ of the form of (\ref{1-point}), $F=z$ and $x=z=0$ are formal local equations of $C$ at $p$.
\item[3)] $p$ is a 1-point of $D$ of the form of (\ref{1-point}), $F=z$, $x=z+y^r\sigma(y)=0$ are formal local equations of $C$ at $p$,
where $r>1$ and $\sigma$ is a unit series.
\item[4)] $p$ is a 2-point of $D$ of the form of (\ref{2-point}), $F=z$, $x=z=0$ are formal local equations of $C$ at $p$.
\item[5)] $p$ is a 2-point of $D$ of the form of (\ref{2-point}), $F=z$, $x=g(y,z)=0$ are formal local equations of $C$ at $p$,
where $g(y,z)$ is not divisible by $z$.
\item[6)] $p$ is a 2-point of $D$ of the form of (\ref{2-point}), $F=1$ (so that $ad-bc\ne 0$) and $x=z=0$ are formal local equations of $C$
at $p$.
\end{enumerate}
Further, there are at most a finite number of 1-points on $C$ satisfying condition 3) (and not satisfying condition 1) or 2)).
\end{Lemma}

\begin{proof} Suppose that $p$ is a 1-point. We have permissible parameters $x,y,z$ at $p$ such that a form (\ref{1-point}) holds at $p$ with $F=z$. There exists a series $g(y,z)$ such that $x=g=0$ are formal local
equations of $C$ at $p$. By the formal implicit function theorem, we get one of the forms 1), 2) or 3).
A similar argument shows that one of the forms 4), 5) or 6) must hold if $p$ is a 2-point.

Now suppose that $p\in C$ is a 1-point, $\sigma_D(p)=0$ and a form 3) holds at $p$.
There exist permissible parameters $x,y,z$ at $p$, with an expression (\ref{1-point}), such that  $x=z=0$ are formal local equations of $C$ at $p$ and $x,y,z$ are uniformizing parameters on an \'etale cover $U$ of an neighborhood  of $p$, where we can choose $U$ so that 
$$
\frac{\partial F}{\partial y}=\frac{1}{x^b}\frac{\partial v}{\partial y}\in\Gamma(U,\mathcal O_X).
$$
Since there is not a form 2) at $p$, we have that $z$ does not divide $F(0,y,z)$, so that $F(0,y,0)\ne 0$. Since $F$ has no constant term, we have that $\frac{\partial F}{\partial y}(0,y,0)\ne 0$. There exists a Zariski open subset of $\mathfrak k$ such that $\alpha\in\mathfrak k$ implies $x,y-\alpha,z$ are regular parameters at a point $q\in U$. There exists a Zariski open subset of $\mathfrak k$ of such $\alpha$ so that
$\frac{\partial F}{\partial y}(0,\alpha,0)\ne 0$. Thus $x,y-\alpha,z$ are permissible parameters at $q$ giving a form 1) at $q\in C$.

\end{proof}

\begin{Lemma}\label{LemmaD} Suppose that $X$ is 2-prepared. Suppose that $C$ is a permissible curve on $X$ which is not a 2-curve and $p\in C$ satisfies $\sigma_D(p)=0$.
Further suppose that either a form 3) or 5) of the conclusions of Lemma \ref{LemmaC} hold at $p$.
Then there exists a sequence of blow ups of points $\pi_1:X_1\rightarrow X$ above $p$ such that $X_1$ is 2-prepared and $\sigma_{D_1}(p_1)=0$ for all $p_1\in \pi_1^{-1}(p)$, and the strict transform of $C$ on $X_1$ is permissible, and has the form 4) or 6) of Lemma \ref{LemmaC} at the point above $p$.
\end{Lemma}

\begin{proof} 
If $p$ is a 1-point, let $\pi':X'\rightarrow X$ be the blow ups of $p$, and let $C'$ be the strict transform of $C$ on $X'$. Let $p'$ be the point on $C'$ above $p$. Then $p'$ is a 2-point and $\sigma_D(p')=0$. We may 
thus assume that $p$ is a 2-point and a form 5) holds at $p$. For $r\in\ZZ_+$, 
let
$$
X_r\rightarrow X_{r-1}\rightarrow \cdots \rightarrow X_1\rightarrow X
$$
be the sequence of blow ups of the point $p_i$ which is the intersection of the strict transform $C_i$ of $C$ on $X_i$ with the preimage of $p$.

There exist permissible parameters $x,y,z$ at $p$ such that $x=z=0$ are formal local equations of $C$ at $p$, and a form (\ref{2-point}) holds at $p$ with $F=x\Omega+f(y,z)$. We have that $\mbox{ord }f(y,z)=1$, $\mbox{ord }\Omega(0,y,z)\ge 1$, $y$ does not divide $f(y,z)$ and $z$ does not divide $f(y,z)$.

At $p_r$, we have permissible parameters $x_r,y_r,z_r$ such that
$$
x=x_ry_r^r,\,\,y=y_r,\,\,z=z_ry_r^r.
$$
$x_r=z_r=0$ are local equations of $C_r$ at $p_r$.  We have a form (\ref{2-point}) at $p_r$ with 
$$
\begin{array}{lll}
u&=& (x_r^ay_r^{ar+b})^l\\
v&=& P(x_r^ay_r^{ar+b})+x_r^cy_r^{cr+d+r}F'
\end{array}
$$
where
$$
F'=x_r\Omega+\frac{f(y_r,z_ry_r^r)}{y_r^r},
$$
if $\frac{f(y_{r-1},z_{r-1}y_{r-1}^{r-1})}{y_{r-1}^{r-1}}$ is not a unit series. Thus for $r$ sufficiently large, we have that $F'$ is a unit, so that a form 6) holds at $p_r$.

\end{proof}

\begin{Lemma}\label{LemmaE} Suppose that $X$ is 2-prepared and that $C$ is a permissible curve on $X$. Suppose that $q\in C$ is a point with $\sigma_D(q)=0$ which has a form 1), 4) or 6)
of Lemma \ref{LemmaC}. Let $\pi_1:X_1\rightarrow X$ be the blow up of $C$. Then $X_1$ is 3-prepared in a neighborhood of $\pi_1^{-1}(q)$. Further, $\sigma_{D_1}(q_1)=0$ for all 
$q_1\in\pi_1^{-1}(q)$.
\end{Lemma}

\begin{proof} The conclusions follow from substitution of local equations of the blow up of $C$ into the forms 1), 4) and 6) of Lemma \ref{LemmaC}.
\end{proof}

\begin{Proposition}\label{Step3.1}  Suppose that $X$ is 2-prepared.  Then there exists a  sequence of permissible blow ups  $\pi_1:X_1\rightarrow X$, 
such that $X_1$ is 3-prepared. We further have that $\sigma_D(p_1)\le \sigma_D(p)$ for all $p\in X$ and $p_1\in\pi_1^{-1}(p)$.
\end{Proposition}

\begin{proof} Let $T$ be the points $p\in X$ such that $X$ is not 3-prepared at $p$. By Lemmas \ref{2-prep1}
and \ref{Torgood}, after we perform a sequence of blow ups of 2-curves, we may assume that $T$ is a finite set consisting of 1-points of $D$.

Suppose that $p\in T$. Let $T'=T\setminus\{p\}$. Let $U=\mbox{Spec}(R)$ be the affine neighborhood of $p$ in $X$ and let $C$ be the curve in $X$ of the conclusions of Lemma \ref{1-ptprep} (with $Y=T'$), so that $C$ has local equations $\overline x=\overline y=0$ in $U$.

Let $\Sigma_1=C\cap \mbox{Sing}_1(X)$. $\Sigma_1=\{p=p_0,\ldots,p_r\}$ is the union of $\{p\}$ and a finite set of general points of curves in $\mbox{Sing}_1(X)$, which must be 1-points. 
We have that $\Sigma_1\subset U$. Let
$$
\Sigma_2=\{q\in C\cap U\mid \sigma_D(q)=0\mbox{ and a form 2) of Lemma \ref{LemmaC} holds at $q$}\}.
$$
$\Sigma_2$ is a finite set by Lemma \ref{LemmaC}.  Let $\Sigma_3=C\setminus U$, a finite set of 1-points and 2-points which are
 prepared.

Set $U'=U\setminus \Sigma_2$. 
There exists a unit $\tau\in R$ and $a\in \ZZ_+$ such that $u=\tau\overline x^a$.

By 5 vi), 5 vii)  of Lemma \ref{1-ptprep} and Lemma \ref{LemmaB},  there exist $z_i\in \hat{\mathcal O}_{X,p_i}$ such that  for all $p_i\in \Sigma_1$, $x=\tau^{\frac{1}{a}}\overline x, \overline y, z_i$ are permissible parameters at $p_i$
giving a form (\ref{eq1}). 

Let $t=\mbox{max}\{ r(p_i)\mid 0\le i\le r\}$, where $r(p_i)$ are calculated from (\ref{eq21})) of  Lemma \ref{algpointprep}. There exists $\lambda\in R$ such that $\lambda\equiv \tau^{-\frac{1}{a}}\mbox{ mod }m_{p_i}^t\hat{\mathcal O}_{X,p_i}$ for $0\le i\le r$. Let $x^*=\lambda^{-1}\overline x$, $\overline\gamma=\tau^{\frac{1}{a}}\lambda$. Then $x=\tau^{\frac{1}{a}}\overline x=\overline \gamma x^*$ with $\overline\gamma\equiv 1\mbox{ mod }m_{p_i}^t\hat{\mathcal O}_{X,p_i}$ for $0\le i\le r$.
Let $U'=U\setminus \Sigma_2$.

Let $T_0^*=\mbox{Spec}(\mathfrak k[x^*,\overline y])$, and let $T_1^*\rightarrow T_0^*$ be a sequence of blow ups of points
above $(x^*,\overline y)$ such that the conclusions of Lemma \ref{algpointprep} hold on $ U_1'= U'\times_{T_0^*}T_1^*$ above all
$p_i$ with $0\le i\le r$.  The projection
$\lambda_1: U_1'\rightarrow  U'$ is a sequence of blow ups of sections over $C$.
$\lambda_1$ is permissible  and $\lambda_1^{-1}(C\cap(U'\setminus \Sigma_1))$ is prepared  by Lemma \ref{LemmaE}.

All points of $\Sigma_2\cup \Sigma_3$ are  prepared. Thus by Lemma \ref{LemmaA},
Lemmas \ref{LemmaD} and Lemma \ref{LemmaE}, by interchanging some blowups of points above $\Sigma_2\cup\Sigma_4$ between blow ups of sections over $C$, we may extend $\lambda_1$ to a sequence of permissible blow ups over $X$ to obtain the desired sequence of permissible blow ups $\pi_1:X_1\rightarrow X$ such that $X_1$
 is 2-prepared. $\pi_1$ is an isomorphism over $T'$, $X_1$
 is 3-prepared over $\pi_1^{-1}(X_1\setminus T')$, and $\sigma_D(p_1)\le\sigma_D(p)$ for all $p\in X_1\setminus T'$.

By induction on $|T|$, we may iterate this procedure a finite number of times to obtain the conclusions of Proposition \ref{Step3.1}.

\end{proof}

The following proposition is proven in a similar way.

\begin{Proposition}\label{local3prep} Suppose that $X$ is 1-prepared and $D'$ is a union of irreducible components of $D$. Suppose that there exists a neighborhood $V$ of $D'$ such that $V$ is 2-prepared and $V$ is 3-prepared at all 2-points and 3-points of $V$.

Let $A$ be a finite set of 1-points of $D'$, such that $A$ is contained in $\mbox{Sing}_1(X)$ and $A$  contains the points where $V$ is not 3-prepared, and let
$B$ be a finite set of 2-points of $D'$. Then there exists a sequence of permissible blow ups
$\pi_1:X_1\rightarrow X$ such that
\begin{enumerate}
\item[1)] $X_1$ is 3-prepared in a neighborhood of $\pi_1^{-1}(D')$.
\item[2)] $\pi_1$ is an isomorphism over $X_1\setminus D'$.
\item[3)] $\pi_1$ is an isomorphism in a neighborhood of $B$.
\item[4)] $\pi_1$ is an isomorphism over generic points of 2-curves on $D'$ and over 3-points of $D'$.
\item[5)] Points on the intersection of the strict transform of $D'$ on $X_1$ with $\pi_1^{-1}(A)$ are 2-points of $D_{X_1}$.
\item[6)] $\sigma_D(p_1)\le \sigma_D(p)$ for all $p\in X$ and $p_1\in\pi_1^{-1}(p)$.
\end{enumerate}
\end{Proposition}

\section{Reduction of $\sigma_D$ above a 3-prepared point.}\label{Section4}

\begin{Theorem}\label{1-pointres}  Suppose that $p\in X$  is a 1-point such that $X$ is 3-prepared at $p$, and $\sigma_D(p)>0$. Let $x,y,z$ be permissible parameters at $p$ giving  a form (\ref{eq3})  at $p$. 
Let $U$ be an \'etale cover of an affine neighborhood  of $p$ in which $ x,  y,  z$ are uniformizing parameters. Then $x  z =0$ gives a toroidal structure $\overline D$ on $U$.
Let $I$ be the ideal in $\Gamma(U,\mathcal O_X)$ generated by $ z^m$, $ x^{r_m}$ if $ \tau_m\ne 0$,
and by 
$$
\{  x^{r_i} z^{m-i}\mid 2\le i\le m-1\mbox{ and } \tau_i\ne 0\}.
$$
Suppose that $\psi:U'\rightarrow U$ is a toroidal morphism with respect to  $\overline D$ such that $U'$ is nonsingular and $I\mathcal O_{U'}$
is locally principal. Then (after possibly replacing $U$ with a smaller neighborhood of $p$) $U'$ is 2-prepared and $\sigma_D(q)<\sigma_D(p)$ for all $q\in U'$. 

 There is (after possibly replacing $U$ with a smaller neighborhood of $p$) a unique, minimal toroidal morphism $\psi:U'\rightarrow U$ with respect to $\overline D$ with has
 the property that $U'$ is nonsingular, 2-prepared and $\Gamma_D(U')<\sigma_D(p)$.
This map $\psi$  factors as a sequence of permissible blowups 
$\pi_i:U_i\rightarrow U_{i-1}$ of  sections $C_i$ over the two curve $C$ of $\overline D$. $U_i$ is 1-prepared for $U_i\rightarrow S$. 
We have that the curve $C_i$ blown up in $U_{i+1}\rightarrow U_i$ is in $\mbox{Sing}_{\sigma_D(p)}(U_i)$ if $C_i$ is not a 2-curve of $D_{U_i}$, and that $C_i$ is in $\mbox{Sing}_1(U_i)$ if $C_i$ is  a 2-curve of $D_{U_i}$.
\end{Theorem}

\begin{proof} Suppose that $\psi:U'\rightarrow U$ is toroidal for $\overline D$ and $U'$ is nonsingular. 
Let $\overline D'=\psi^{-1}(\overline D)$.

The set of 2-curves of 
$\overline D'$ is the disjoint union of the 2-curves of $D_{U'}$ and the 2-curve which is the intersection of the strict transform of the surface $ z=0$ on $U'$ with $D_{U'}$. $\psi$ factors as a sequence of blow ups of 2-curves of (the preimage of) $\overline D$. We will  verify the following three statements,
from which the conclusions of the  theorem follow.

\begin{equation}\label{state1}
\begin{array}{l}
\mbox{If $q\in\psi^{-1}(p)$ and $I\mathcal O_{U',q}$ is principal, then $\sigma_D(q)<\sigma_D(p)$.}\\
\mbox{In particular, $\sigma_D(q)<\sigma_D(p)$ if $q$ is a 1-point of $\overline D'$.}
\end{array}
\end{equation}

\begin{equation}\label{state2}
\begin{array}{l}
\mbox{If $C'$ is a 2-curve of $D_{U'}$, then $U'$ is prepared at  $q= C'\cap\psi^{-1}(p)$}\\
\mbox{if and only if $\sigma_D(q)<\infty$}\\
\mbox{if and only if $I\mathcal O_{U',q}$ is principal}\\
\mbox{if and only
if $U'$ is prepared at all $q'\in C'$ in a neighborhood of $q$.}
\end{array}
\end{equation}

\begin{equation}\label{state3}
\begin{array}{l}
\mbox{If $C'$ is the 2-curve of $\overline D'$ which is the intersection of $D_{U'}$ with the strict transform}\\
\mbox{of $\tilde z=0$ in $U'$, then $\sigma_{D}(q)\le \sigma_D(p)$ if $q=C'\cap \psi^{-1}(p)$, and $\sigma_D(q')=\sigma_D(q)$}\\ 
\mbox{for  $q'\in C'$ in a neighborhood of $q$.}
\end{array}
\end{equation}

Suppose that $q\in\psi^{-1}(p)$ is a 1-point for $\overline D'$. Then $I\hat{\mathcal O}_{U',q}$ is principal.
At $q$, we have permissible parameters $x_1,y,z_1$ defined by
\begin{equation}\label{eq71}
x=x_1^{a_1}, z=x_1^{b_1}(z_1+\alpha)
\end{equation}
for some $a_1,b_1\in\ZZ_+$ and $0\ne\alpha\in \mathfrak k$. Substituting into (\ref{eq3}), we have
$$
u=x_1^{aa_1},
v=P(x_1^{a_1})+x_1^{ba_1}G
$$
where 
$$
G=\tau_0x_1^{b_1m}(z_1+\alpha)^m + \tau_2x_1^{a_1r_2+b_1(m-2)}(z_1+\alpha)^{m-2}+
\cdots+\tau_{m-1}x_1^{a_1r_{m-1}+b_1}(z_1+\alpha)+\tau_mx_1^{a_1r_m}.
$$
Let $x_1^s$ be a local generator of $I\hat{\mathcal O}_{U',q}$. Let $G'=\frac{G}{x_1^s}$.

If $z^m$ is a local generator of $I\hat{\mathcal O}_{U',q}$, then $G'$ has an expansion
$$
G'=\tau'(z_1+\alpha)^m+g_2(z_1+\alpha)^{m-2}+\cdots+g_{m-1}(z_1+\alpha)+g_m
+x_1\Omega_1+y\Omega_2
$$
where $0\ne \tau'=\tau(0,0,0)\in \mathfrak k$, $g_2,\ldots,g_m\in \mathfrak k$ and $\Omega_1,\Omega_2\in \hat{\mathcal O}_{U',q}$. We have $\mbox{ord}(G'(0,0,z_1))\le m-1$.
Setting $F'=G'-G'(x_1,0,0)$ and $P'(x_1)=P(x_1^{a_1})+x_1^{ba_1+b_1m}G'(x_1,0,0)$,
we have an expression
$$
u=x_1^{aa_1},
v=P'(x_1)+x_1^{ba_1+b_1m}F'
$$
of the form of (\ref{1-point}). Thus $U'$ is 2-prepared at $q$ with $\sigma_{D'}(q)<m-1=\sigma_D(p)$.

Suppose that $z^m$ is not a local generator of $I\hat{\mathcal O}_{U',q}$, but there exists some $i$ with
$2\le i\le m-1$ such that $x^{r_i}z^{m-i}$ is a local generator of $I\hat{\mathcal O}_{U',q}$. Let
$h$ be the smallest $i$ with this property. Then $G'$ has an expression
$$
G'=g_h(z_1+\alpha)^{m-h}+\cdots+g_m+x_1\Omega_1+y_1\Omega_2
$$
for some $g_i\in \mathfrak k$ with $g_h\ne 0$ and $\Omega_1,\Omega_2\in\hat{\mathcal O}_{U',q}$. As in the previous case, we have that $U'$ is 2-prepared at $q$ with $\sigma_{D}(q)< m-h-1<m-1=\sigma_D(p)$.

Suppose that $z^m$ is not a local generator of $I\hat{\mathcal O}_{U',q}$ and
$x^{r_i}z^{m-i}$ is not a local generator of $I\hat{\mathcal O}_{U',q}$ for $2\le i\le m-1$. Then $x_1^{r_m}$ is a local generator of $I\mathcal O_{U',q}$, and we have an expression
$$
G'=\Lambda+x_1\Omega_1,
$$
where $\Lambda(x_1,y,z_1)=\tau_m(x_1^{a_1},y,x_1^{b_1}(z_1+\alpha))$ and $\Omega_1\in\hat{\mathcal O}_{U',q}$.
Then 
$$
\mbox{ord }\Lambda(0,y,0)=\mbox{ord }\tau_m(0,y,0)=1,
$$
and we have that $U'$ is prepared at $q$.

Now suppose that $q\in\psi^{-1}(p)$ is a 2-point for $D_{U'}$. We have permissible parameters $x_1,y,z_1$ in
$\hat{\mathcal O}_{U',q}$ such that
\begin{equation}\label{eq72}
x=x_1^{a_1}z_1^{b_1},
z=x_1^{c_1}z_1^{d_1}
\end{equation}
with $a_1,b_1>0$ and $a_1d_1-b_1c_1=\pm 1$. Substituting into (\ref{eq3}), we have
$$
u=x_1^{a_1a}z_1^{b_1a},
v=P(x_1^{a_1}z_1^{b_1})+x_1^{a_1b}z_1^{b_1b}G
$$
where 
$$
G=\tau_0x_1^{c_1m}z_1^{d_1m}+\tau_2x_1^{r_2a_1+c_1(m-2)}z_1^{r_2b_1+d_1(m-2)}+
\cdots+\tau_{m-1}x_1^{a_1r_{m-1}+c_1}z_1^{b_1r_{m-1}+d_1}+\tau_mx_1^{a_1r_m}z_1^{b_1r_m}.
$$

Let $C'$ be the 2-curve of $D_{U'}$ containing $q$. Since $\mbox{ord }(\tau_m(0,y,0))=1$ (if $\tau_m\ne 0$) we see that the three statements $\sigma_D(q)<\infty$,
$\sigma_D(q)=0$ and $I\mathcal O_{U',q}$ is principal are equivalent. 
Further, we have that $\sigma_D(q')=\sigma_D(q)$ for $q'\in C'$ in a neighborhood of $q$.

Suppose that  $I\mathcal O_{U',q}$ is principal and let $x_1^sz_1^t$ be a local generator of
$I\hat{\mathcal O}_{U',q}$. Let $G'=G/x_1^sz_1^t$. We have that
$$
u=(x_1^{a_1}z_1^{b_1})^a,\,\, v=P(x_1^{a_1}z_1^{b_1})+x_1^{a_1b+s}z_1^{bb_1+t}G'
$$ has the form (\ref{2-point}),
since we have made a monomial substitution in $x$ and $z$.  If $z^m$ or $x^{r_i}z^{m-i}$ for some $i<m$ is a local generator of $I\hat{\mathcal O}_{U',q}$, then $G'$ is a unit in $\hat{\mathcal O}_{U',q}$. If none
of $z^m$, $x^{r_i}z^{m-i}$ for $i<m$ are local generators of $I\hat{\mathcal O}_{U',q}$, then
$$
G'=\Lambda +x_1\Omega_1+z_1\Omega_2,
$$
where
$$
\Lambda(x_1,y_1,z_1)=\tau_m(x_1^{a_1}z_1^{b_1},y,x_1^{c_1}z_1^{d_1})
$$
and $\Omega_1,\Omega_2\in\hat{\mathcal O}_{U',q}$.
Thus
$$
\mbox{ord }\Lambda(0,y,0)=\mbox{ord }\tau_m(0,y,0)=1.
$$
We thus have that $U'$ is prepared at $q$.

The final case is when $q\in\psi^{-1}(p)$ is on the 2-curve $C'$ of $\overline D'$ which is the intersection of $D_{U'}$ with the strict transform of $ z=0$ in $U'$. Then there exist permissible parameters $x_1,y,z_1$ at $q$ such that
\begin{equation}\label{eq73}
x=x_1, z=x_1^{b_1}z_1
\end{equation}
for some $b_1\in\ZZ_+$. 
The equations $x_1=z_1=0$ are local equations of $C'$ at $q$. Let 
$$
s=\mbox{min}\{b_1m,r_i+b_1(m-i)\mbox{ with } \tau_i\ne 0\mbox{ for }2\le i\le m-1, r_m\mbox{ if }\tau_i\ne 0\}.
$$
We have an expression of the form (\ref{1-point}) at $q$,
$$
\begin{array}{lll}
u&=& x_1^a\\
v&=&P(x_1^a)+x_1^{ab+s}G'
\end{array}
$$
with
$$
G'=\tau_0x_1^{b_1m-s}z_1^m+\tau_2x_1^{r_2+b_1(m-2)-s}z_1^{m-2}+\cdots+
\tau_{m-1}x_1^{r_{m-1}+b_1-s}z_1+\tau_mx_1^{r_m-s}.
$$
We see that $\sigma_D(q)\le \sigma_D(p)$ (with $\sigma_D(q)<\sigma_D(p)$ if $s=r_i+b_1(m-i)$ for some $i$ with $2\le i\le m-1$ or $s=r_m$) and $\sigma_D(q')=\sigma_D(q)$ for $q'$ in a neighborhood of $q$ on $C'$.

Suppose that $I\mathcal O_{U',q}$ is principal. Then $x^{r_m}$ generates $I\hat{\mathcal O}_{U',q}$.
We have that $G'=x_1^{r_m}\Omega$ where $\Omega\in\hat{\mathcal O}_{U',q}$ satisfies $\mbox{ord }\Omega(0,y,0)=1$. Thus $U'$ is prepared at $q$.

\end{proof}

We will now construct the function $\omega(m,r_2,\ldots,r_{m-1})$ where $m>1$, $r_i\in \NN$ for $2\le i\le m-1$ and $r_{m-1}>0$.

Let $I$ be the ideal in the polynomial ring $\mathfrak k[x,z]$ generated by $z^m$ and $x^{r_i}z^{m-i}$ for all $i$ such that $2\le i\le m-1$ and $r_i>0$. Let $\mathfrak m=(x,z)$ be the maximal ideal of $\mathfrak \mathfrak k[x,z]$. Let $\Phi:V_1\rightarrow V=\mbox{Spec}(\mathfrak k[x,z])$ be the toroidal morphism with respect to the divisor $xz=0$ on $V$ such that $V_1$ is the minimal nonsingular surface such that
\begin{enumerate}
\item[1)] $I\mathcal O_{V_1,q}$ is principal if $q\in \Phi^{-1}(\mathfrak m)$ is not on the strict transform of $z=0$.
\item[2)] If $q$ is the intersection point of the strict transform of $z=0$ and $\Phi^{-1}(\mathfrak m)$, so
that $q$ has regular parameters $x_1,z_1$, with $x=x_1, z=x_1^bz_1$ for some $b\in\ZZ_+$, then $r_i+b_1(m-i)<b_1m$ for some $2\le i\le m-1$ with $r_i>0$.
\end{enumerate}

Every $q\in \Phi^{-1}(\mathfrak m)$ which is not on the strict transform of $z=0$ has regular parameters
$x_1,z_1$ at $q$ which are related to $x,z$ by one of the following expressions:
\begin{equation}\label{eq80}
x=x_1^{a_1},\,\, z=x_1^{b_1}(z_1+\alpha)
\end{equation}
for some $0\ne \alpha\in \mathfrak k$ and $a_1,b_1>0$, or
\begin{equation}\label{eq81}
x=x_1^{a_1}z_1^{b_1},\,\,z=x_1^{c_1}z_1^{d_1}
\end{equation}
with $a_1,b_1>0$ and $a_1d_1-b_1c_1=\pm 1$. There are only finitely many values of $a_1,b_1$ occurring in expressions (\ref{eq80}), and $a_1,b_1,c_1,d_1$ occurring in expressions (\ref{eq81}).

The point $q$ on the intersection of the strict  transform of $z=0$ and $\Phi^{-1}(\mathfrak m)$ has regular parameters $x_1,z_1$ defined by
\begin{equation}\label{eq82}
x=x_1,\,\, z=x_1^{b_1}z_1
\end{equation}
for some $b_1>0$.

Now we define $\omega=\omega(m,r_2,\ldots,r_{m-1})$ to be a number such that 
$$
\omega>\mbox{max}\{\frac{b_1}{a_1}m,r_i+\frac{b_1}{a_1}(m-i)\mbox{ for }2\le i\le m-1\mbox{ such that }r_i>0\}.
$$
For all expressions (\ref{eq80}),
$$
\omega>\mbox{max}\{\frac{c_1}{a_1}m, \frac{d_1}{b_1}m,r_i+\frac{c_1}{a_1}(m-i), r_i+\frac{d_1}{b_1}(m-i)\mbox{ for }2\le i\le m-1\mbox{ such that }r_i>0\}
$$
for all expressions (\ref{eq81}), and
$$
\omega>\mbox{max}\{b_1m,r_i+b_1(m-i)\mbox{ for }2\le i\le m-1\mbox{ such that }r_i>0\}
$$
in (\ref{eq82}).

\begin{Theorem}\label{1-pointspec}  Suppose that $p\in\mbox{
Sing}_{1}(X)$ is a 1-point and $X$ is 3-prepared at $p$. Let $x,y,z$ be permissible parameters at $p$ giving  a form (\ref{eq4})  at $p$. 
Let $U$ be an \'etale cover of an affine neighborhood  of $p$ in which $ x,  y,  z$ are uniformizing parameters. Then $ x  z =0$ gives a toroidal structure $\overline D$ on $U$.

 There is (after possibly replacing $U$ with a smaller neighborhood of $p$) a unique, minimal toroidal morphism $\psi:U'\rightarrow U$ with respect to $\overline D$ with has
 the property that $U'$ is nonsingular, 2-prepared and $\Gamma_D(U')<\sigma_D(p)$.
This map $\psi$  factors as a sequence of permissible blowups 
$\pi_i:U_i\rightarrow U_{i-1}$ of  sections $C_i$ over the two curve $C$ of $\overline D$. $U_i$ is 1-prepared for $U_i\rightarrow S$. 
We have that the curve $C_i$ blown up in $U_{i+1}\rightarrow U_i$ is in $\mbox{Sing}_{\sigma_D(p)}(U_i)$ if $C_i$ is not a 2-curve of $D_{U_i}$, and that $C_i$ is in $\mbox{Sing}_1(U_i)$ if $C_i$ is  a 2-curve of $D_{U_i}$.
\end{Theorem}

\begin{proof} The proof is similar to that of Theorem \ref{1-pointres}, using the fact that $t>\omega(m,r_2,\ldots,r_{m-1})$ as defined above.
\end{proof}

\begin{Theorem}\label{2-pttheorem} Suppose that $p\in X$ is a 2-point and $X$ is 3-prepared at $p$
with $\sigma_D(p)>0$. Let $x,y,z$ be permissible parameters at $p$ giving a form (\ref{eq2}) at $p$.
Let $U$ be an \'etale cover of an affine neighborhood of $p$ in which $ x, y, z$ are uniformizing parameters on $U$. Then $ x y z=0$ gives a toroidal structure $\overline D$ on $U$.
Let $I$ be the ideal in $\Gamma(U,\mathcal O_X)$ generated by $ z^m$, $ x^{r_m} y^{s_m}$ if $\tau_m\ne 0$ and  
$$
\{ x^{r_i} y^{s_i} z^{m-i}\mid 2\le i\le m-1\mbox{ and }\tau_{i}\ne 0\}.
$$
Suppose that $\psi:U_1\rightarrow U$ is a toroidal morphism with respect to $\overline D$ such that
$U_1$ is nonsingular and  $I{\mathcal O}_{U_1}$ is locally principal.  Then
(after possibly replacing $U$ with a smaller neighborhood of $p$)
$U_1$ is 2-prepared for $U_1\rightarrow S$, with $\sigma_{D}(q)<\sigma_D(p)$ for all $q\in U_1$.
 \end{Theorem}
 
 \begin{proof} Suppose that $q\in \psi^{-1}(p)$ is a 1-point for $\psi^{-1}(\overline D)$. 
 Then $q$ is also a 1-point for $D_{U_1}$. Since  $\psi$ is toroidal with respect to $\overline D$, there exist regular parameters $\hat x_1,\hat y_1,\hat z_1$ in $\hat{\mathcal O}_{X_1,q}$ and a matrix $A=(a_{ij})$ with nonegative integers as coefficients such that $\mbox{Det }A=\pm 1$, and we have an expression
 \begin{equation}\label{eq35}
 \begin{array}{lll}
 x&=& \hat x_1^{a_{11}}(\hat y_1+\alpha)^{a_{12}}(\hat z_1+\beta)^{a_{13}}\\
 y&=& \hat x_1^{a_{21}}(\hat y_1+\alpha)^{a_{22}}(\hat z_1+\beta)^{a_{23}}\\
 z&=& \hat x_1^{a_{31}}(\hat y_1+\alpha)^{a_{32}}(\hat z_1+\beta)^{a_{33}}
 \end{array}
 \end{equation}
 with $a_{11}, a_{21}, a_{31}\ne 0$ and $0\ne \alpha,\beta\in \mathfrak k$. Set
 $$\overline x_1=\hat x_1(\hat y_1+\alpha)^{\frac{a_{12}}{a_{11}}}(\hat z_1+\beta)^{\frac{a_{13}}{a_{11}}}\in\hat{\mathcal O}_{X_1,q}.
 $$
 Substituting into (\ref{eq35}), we have
 \begin{equation}
 \begin{array}{lll}
 x&=& \overline x_1^{a_{11}}\\
 y&=& \overline x_1^{a_{21}}(\hat y_1+\alpha)^{a_{22}-\frac{a_{21}a_{12}}{a_{11}}}(\hat z_1+\beta)^{a_{23}-\frac{a_{21}{a_{13}}}{a_{11}}}\\
 z&=& \overline x_1^{a_{31}}(\hat y_1+\alpha)^{a_{32}-\frac{a_{31}a_{12}}{a_{11}}}(\hat z_1+\beta)^{a_{33}-\frac{a_{31}a_{13}}{a_{11}}}.
 \end{array}
 \end{equation}
 
 Let $B=(b_{ij})$ be the adjoint matrix of $A$.
 Let $\overline\alpha=\alpha^{\frac{b_{33}}{a_{11}}}\beta^{-\frac{b_{23}}{a_{11}}}$, $\overline\beta=\alpha^{-\frac{b_{32}}{a_{11}}}
 \beta^{\frac{b_{22}}{a_{11}}}$.
 Set
 $$
 \overline y_1=\frac{y}{\overline x_1^{a_{21}}}-\overline \alpha,
 \overline z_1=\frac{z}{\overline x_1^{a_{31}}}-\overline\beta.
 $$
 We will show that $\overline x_1,\overline y_1,\overline z_1$ are regular parameters in $\hat{\mathcal O}_{X_1,q}$. We have that
 $$
 \begin{array}{lll}
 (\hat y_1+\alpha)^{a_{22}-\frac{a_{21}a_{12}}{a_{11}}}(\hat z_1+\beta)^{a_{23}-\frac{a_{21}a_{13}}{a_{11}}}
 &=&
 \overline\alpha+\frac{b_{33}}{a_{11}}\alpha^{\frac{b_{33}}{a_{11}}-1}\beta^{-\frac{b_{23}}{a_{11}}}\hat y_1-
 \frac{b_{23}}{a_{11}}\alpha^{\frac{b_{33}}{a_{11}}}\beta^{-\frac{b_{23}}{a_{11}}-1}\hat z_1+\cdots\\
 (\hat y_1+\alpha)^{a_{32}-\frac{a_{31}a_{12}}{a_{11}}}(\hat z_1+\beta)^{a_{33}-\frac{a_{31}a_{13}}{a_{11}}}
 &=&
 \overline\beta-\frac{b_{32}}{a_{11}}\alpha^{-\frac{b_{32}}{a_{11}}-1}\beta^{\frac{b_{22}}{a_{11}}}\hat y_1+
 \frac{b_{22}}{a_{11}}\alpha^{-\frac{b_{32}}{a_{11}}}\beta^{\frac{b_{22}}{a_{11}}-1}\hat z_1+\cdots
 \end{array}
 $$

 Let 
 $$
 C=\left(\begin{array}{ll}
 \frac{b_{33}}{a_{11}}\alpha^{\frac{b_{33}}{a_{11}}-1}\beta^{-\frac{b_{23}}{a_{11}}}&
 -\frac{b_{23}}{a_{11}}\alpha^{\frac{b_{33}}{a_{11}}}\beta^{-\frac{b_{23}}{a_{11}}-1}\\
 -\frac{b_{32}}{a_{11}}\alpha^{-\frac{b_{32}}{a_{11}}-1}\beta^{\frac{b_{22}}{a_{11}}}&
 \frac{b_{22}}{a_{11}}\alpha^{-\frac{b_{32}}{a_{11}}}\beta^{\frac{b_{22}}{a_{11}}-1}
 \end{array}\right).
 $$

 We must show that $C$ has rank 2. $C$ has the same rank as
 $$
 \left(\begin{array}{ll}
 b_{33}\beta&-b_{23}\alpha\\
 b_{32}\beta&-b_{22}\alpha
 \end{array}\right)
 =
 \left(\begin{array}{ll}
 b_{33}&b_{23}\\
 b_{32}&b_{22}
 \end{array}\right)
 \left(\begin{array}{ll}
 \beta&0\\
 0&-\alpha
 \end{array}\right).
 $$
 Since $\alpha,\beta\ne 0$, $C$ has the same rank as 
 $$
 B'=\left(\begin{array}{ll}
 b_{33}&b_{23}\\
 b_{32}&b_{22}
 \end{array}\right).
 $$
 Since $B$ has rank 3, 
 $$
 \left(\begin{array}{lll}
 b_{21}&b_{22}&b_{23}\\
 b_{31}&b_{32}&b_{33}
 \end{array}\right)
 $$
 has rank 2.
 Since
 $$
 \left(\begin{array}{l} b_{21}\\ b_{31}\end{array}\right)
 =-\frac{a_{21}}{a_{11}}\left(\begin{array}{l} b_{22}\\ b_{32}\end{array}\right)
 +\frac{a_{31}}{a_{11}}\left(\begin{array}{l} b_{23}\\ b_{33}\end{array}\right),
 $$
 we have that $B'$ has rank 2, and hence $C$ has rank 2.  Thus $\overline x_1,\overline y_1,\overline z_1$
 are regular parameters in $\hat{\mathcal O}_{X_1,q}$. We have
 $$
 x=\overline x_1^{a_{11}}, y=\overline x_1^{a_{21}}(\overline y_1+\overline \alpha),
 z=\overline x_1^{a_{31}}(\overline z_1+\overline\beta).
 $$
 We have that $u=(x^ay^b)^{\ell}$. Let
 $$
 t=-\frac{b}{a_{11}a+a_{21}b},
 $$
 and set $\overline x_1=x_1(y_1+\overline\alpha)^t$. Define $\overline y_1=y_1$, $\tilde \alpha=\overline\alpha$,
 $\tilde \beta=\overline \alpha^{ta_{31}}\overline\beta$ and $z_1=(\overline y_1+\overline\alpha)^{ta_{31}}(z_1+\overline\beta)-\tilde \beta$. Then $x_1,y_1,z_1$ are permissible parameters at $q$, with $u=x_1^{(aa_{11}+ba_{21})l}$,
 $$
 x=x_1^{a_{11}}(y_1+\tilde\alpha)^{ta_{11}},
 y=x_1^{a_{21}}(y_1+\tilde\alpha)^{ta_{21}+1},
 z=x_1^{a_{31}}(z_1+\tilde\beta).
 $$

 Thus we have shown that there exist (formal) permissible parameters
 $x_1,y_1,z_1$ at $q$ such that
 $$
 x=x_1^{e_1}(y_1+\tilde\alpha)^{\lambda_1}, y=x_1^{e_2}(y_1+\tilde\alpha)^{\lambda_2}, z=x_1^{e_3}(z_1+\tilde\beta)
 $$
 where $e_1,e_2,e_3\in\ZZ_+$, $\tilde \alpha,\tilde \beta\in \mathfrak k$ are nonzero,  $\lambda_1,\lambda_2\in\QQ$ are both nonzero, and 
 $u=x_1^{b_1l}$,
 where $b_1=ae_1+be_2$, $a\lambda_1+b\lambda_2=0$.
 We then have an expression
 $$
 v=P(x_1^{ae_1+be_2})+x_1^{ce_1+de_2}G,
 $$
 where
 $$
 \begin{array}{lll}
 G&=&
 (y_1+\tilde \alpha)^{c\lambda_1+d\lambda_2}
 [\tau_0 x_1^{e_3m}(z_1+\tilde\beta)^m\\
 &&+ \tau_2 x_1^{r_2e_1+s_2e_2+(m-2)e_3}(y_1+\tilde \alpha)^{r_2\lambda_1+s_2\lambda_2}(z_1+\tilde \beta)^{m-2}+\cdots\\
 && + \tau_{m-1}x_1^{r_{m-1}e_1+s_{m-1}e_2+e_3}(y_1+\tilde\alpha)^{r_{m-1}\lambda_1+s_{m-1}\lambda_2}(z_1+\tilde\beta)\\
 &&+\tau_m x_1^{r_me_1+s_me_2}y_1^{r_m\lambda_1+s_m\lambda_2}].
 \end{array}
 $$
 Let $\tau'=\tau_0(0,0,0)$. Let $x_1^s$ be a  generator of $I\hat{\mathcal O}_{U_1,q}$. Let $G'=\frac{F}{x_1^s}$.
 
 If $z^m$ is a local generator of $I\hat{\mathcal O}_{U_1,q}$, then $G'$ has  an expression
 $$
 G'=\tau'\tilde\alpha^{\phi}(z_1+\tilde\beta)^{m}+g_2(z_1+\tilde\beta)^{m-2}+\cdots +g_{m-1}(z+\tilde\beta)+g_m+x_1\Omega_1+y_1\Omega_2
 $$
 for some $g_i\in \mathfrak k$ and $\Omega_1,\Omega_2\in \hat{\mathcal O}_{U_1,q}$, where $\phi=c\lambda_1+d\lambda_2$.
 Setting $F'=G'-G'(x_1,0,0)$, and
 $P'(x_1)=P(x_1^{ae_1+be_2})+x_1^{ce_1+de_2+s}G'(x_1,0,0)$,
 we have that 
 $$
 u=x_1^{b_1l},
 v=P'(x_1)+x_1^{ce_1+de_2+s}F'
 $$
 has the form (\ref{1-point})
 and   $\sigma_{D}(q)\le \mbox{ord }F'(0,0,z_1)-1\le m-2 <m-1=\sigma_D(p)$ since $0\ne \tilde\beta$.
 
 Suppose that $z^m$ is not a local generator of $I\hat{\mathcal O}_{U_1,q}$, but
 there exists some $i$ with $2\le i\le m-1$ such that $\tau_ix^{r_i}y^{s_i}z^{m-i}$
 is a local generator of $I\hat{\mathcal O}_{U_1,q}$. Let $h$ be the smallest $i$ with this property. Then  
 $G'$ has  an expression
 $$
 G'=g_h(z_1+\tilde\beta)^{m-h}+\cdots +g_{m-1}(z_1+\tilde\beta)+g_m+x_1\Omega_1+y_2\Omega_2
 $$
 for some $g_i\in \mathfrak k$ with $g_h\ne 0$ As in the previous case, we have  
 $$
 \sigma_{D}(q)\le m-h-1<m-1=\sigma_D(p).
 $$
 
Suppose that  $z^m$ is not a local generator of $I\hat{\mathcal O}_{U_1,q}$, and
  $\tau_ix^{r_i}y^{s_i}z^{m-i}$
 is not a local generator of $I\hat{\mathcal O}_{U_1,q}$ for $2\le i\le m$.
 Then $x^{r_s}y^{r_s}$ is a local generator of $I\hat{\mathcal O}_{U_1,q}$, and 
 $G'$ has  an expression
 $$
 G'=\tau_m'(y_1+\tilde\alpha)^{\phi+r_m\lambda_1+s_m\lambda_2}+x_1\Omega
 $$
 where $\tau_m'=\tau_m(0,0,0)$
 for some  $\Omega\in \hat{\mathcal O}_{U_1,q}$. 
 Suppose, if possible, that
 $\phi+r_m\lambda_1+s_m\lambda_2=0$.
 Since $\phi+r_m\lambda_1+s_m\lambda_2=(c+r_m)\lambda_1+(d+s_m)\lambda_2$, we then have that
 the nonzero vector $(\lambda_1,\lambda_2)$ satisfies
 $a\lambda_1+b\lambda_2=(c+r_m)\lambda_1+(d+s_m)\lambda_2=0$. Thus the determinant
 $a(d+s_m)-b(c+r_m)=0$, a contradiction to our assumption that $F$ satisfies (\ref{2-point}).

 Now since $\phi+r_m\lambda_1+s_m\lambda_2\ne 0$ and
  $\tilde \alpha\ne 0$, we have
 $1=\mbox{ord }G'(0,y_1,0)<m$, so that $\sigma_{D}(q)=0<m-1=\sigma_D(p)$.
 
 \vskip .2truein
 Suppose that $q\in\psi^{-1}(p)$ is a 2-point of $\psi^{-1}(\overline D)$.
 Then there exist  (formal) permissible parameters
 $\hat x_1,\hat y_1,\hat z_1$ at $q$ such that
 \begin{equation}\label{eq31}
 x=\hat x_1^{e_{11}}\hat y_1^{e_{12}}(\hat z_1+\hat \alpha)^{e_{13}}, y=\hat x_1^{e_{21}}\hat y_1^{e_{22}}(\hat z_1+\hat \alpha)^{e_{23}}, z=\hat x_1^{e_{31}}\hat y_1^{e_{32}}(\hat z_1+\hat \alpha)^{e_{33}}
 \end{equation}
 where $e_{ij}\in\NN$, with $\mbox{Det}(e_{ij})=\pm1$, and $\hat \alpha\in \mathfrak k$ is nonzero.
 We further have
 $$
 e_{11}+e_{12}>0, e_{21}+e_{22}>0\mbox{ and }e_{31}+e_{32}>0.
 $$

 First suppose that $e_{11}e_{22}-e_{12}e_{21}\ne 0$. Then $q$ is a 2-point of $D_{U_1}$.
 
 There exist $\lambda_1,\lambda_2\in\QQ$ such that upon setting
 $$
 \hat x_1=x_1(z_1+\hat\alpha)^{\lambda_1}\mbox{ and }
 \hat y_1=y_1(z_1+\hat\alpha)^{\lambda_2},
 $$
 we have
 $$
 x=x_1^{e_{11}}y_1^{e_{12}},
 y=x_1^{e_{21}}y_1^{e_{22}},
 z=x_1^{e_{31}}y_1^{e_{32}}(z_1+\hat\alpha)^r,
 $$
 where
 $$
 \left(\begin{array}{lll}
 e_{11}&e_{12}&e_{13}\\
 e_{21}&e_{22}&e_{23}\\
 e_{31}&e_{32}&e_{33}
 \end{array}\right)
 \left(\begin{array}{l}
 \lambda_1\\ \lambda_2\\ 1\end{array}\right)
 =\left(\begin{array}{l} 0\\0\\r\end{array}\right).
 $$
 By Cramer's rule,
 $$
 r=\pm\frac{1}{e_{11}e_{22}-e_{12}e_{21}}\ne 0.
 $$
 Now set $z_1=(z_1+\hat\alpha)^r-\hat\alpha^r$ and $\alpha=\hat\alpha^r$ to obtain  permissible
 parameters $x_1,y_1,z_1$ at $q$ with
 $$
 x=x_1^{e_{11}}y_1^{e_{12}}, y=x_1^{e_{21}}y_1^{e_{22}}, z=x_1^{e_{31}}y_1^{e_{32}}(z_1+\alpha).
 $$
 
We have an expression
$$
u=((x_1^{e_{11}}y_1^{e_{12}})^{a}(x_1^{e_{21}}y_1^{e_{22}})^{b})^{\ell}
=(x_1^{t_1}y_1^{t_2})^{\ell_1}
$$
where $t_1,t_2,\ell_1\in\ZZ_+$ and $\mbox{gcd}(t_1,t_2)=1$.

 We then have an expression
 $$
 v=P((x_1^{t_1}y_1^{t_2})^{\frac{\ell_1}{\ell}})+x_1^{ce_{11}+de_{21}}y_1^{ce_{12}+de_{22}}G,
 $$
 where
 $$
 \begin{array}{lll}
 G&=&
 [\tau_0 x_1^{me_{31}}y_1^{me_{32}}(z_1+\alpha)^{m}+ \tau_2 x_1^{r_2e_{11}+s_2e_{21}+(m-2)e_{31}}y_1^{r_2e_{12}+s_2e_{22}+(m-2)e_{32}}(z_1+\alpha)^{m-2}+\cdots\\
 && +
  \tau_{m-1}x_1^{r_{m-1}e_{11}+s_{m-1}e_{21}+e_{31}}y_1^{r_{m-1}e_{12}+s_{m-1}e_{22}+e_{32}}(z_1+\beta)
 +\tau_m x_1^{r_me_{11}+s_me_{21}}y_1^{r_me_{12}+s_me_{22}}].
 \end{array}
 $$
 Let $\tau'=\overline\tau_0(0,0,0)$. Let $x_1^sy_1^t$ be a  generator of $I\hat{\mathcal O}_{U_1,q}$. Let $G'=\frac{G}{x_1^sy_1^t}$.

 If $z^m$ is a local generator of $I\hat{\mathcal O}_{U_1,q}$, then $G'$ has  an expression
 $$
 G'=\tau'(z_1+\alpha)^{m}+g_2(z_1+\alpha)^{m-2}+\cdots +g_{m-1}(z-\alpha)+g_m+x_1\Omega_1+y_1\Omega_2
 $$
 for some $g_i\in \mathfrak k$ and $\Omega_1,\Omega_2\in \hat{\mathcal O}_{U_1,q}$.
 Let 
 \begin{equation}\label{eq30}
 \overline P(x_1^{t_1}y_1^{t_2})=\sum_{t_2i-t_1j=0} \frac{1}{i!j!}\frac{\partial (x_1^{ce_{11}+de_{21}}y_1^{ce_{12}+de_{22}}G)}{\partial x_1^i\partial y_1^j}(0,0,0)x_1^iy_1^j
 \end{equation}
 and $F'=G'-\frac{\overline P(x_1^{t_1}y_1^{t_2})}{x_1^{ce_{11}+de_{21}+s}y_1^{ce_{12}+de_{22}+t}}$.
 Set 
 $P'(x_1^{t_1}y_1^{t_2})=P((x_1^{t_1}y_1^{t_2})^{\frac{\ell_1}{\ell}})+\overline P(x_1^{t_1}y_1^{t_2})$.
 We have that 
 $$
 u=(x_1^{t_1}y_1^{t_2})^{\ell_1},
 v=P'(x_1^{t_1}y_1^{t_2})+x_1^{ce_{11}+de_{21}+s}y_1^{ce_{12}+de_{22}+t}F'
 $$
 has the form (\ref{2-point}), 
 and   $\sigma_{D}(q)= \mbox{ord }F'(0,0,z_1)-1\le m-2 <m-1=\sigma_D(p)$ since $0\ne \alpha$.
 
 Suppose that $z^m$ is not a local generator of $I\hat{\mathcal O}_{U_1,q}$, but
 there exists some $i$ with $2\le i\le m-1$ such that $\tau_ix^{r_i}y^{s_i}z^{m-i}$
 is a local generator of $I\hat{\mathcal O}_{U_1,q}$. Let $h$ be the smallest $i$ with this property. Then  
 $G'$ has  an expression
 $$
 G'=g_h(z_1+\beta)^{m-h}+\cdots +g_m+x_1\Omega_1+y_2\Omega_2
 $$
 for some $g_i\in \mathfrak k$ with $g_h\ne 0$ As in the previous case, we have  $\sigma_{D}(q)\le m-h-1<m-1=\sigma_D(p)$.
 
Suppose that  $z^m$ is not a local generator of $I\hat{\mathcal O}_{U_1,q}$, and
  $\tau_ix^{r_i}y^{s_i}z^{m-i}$
 is not a local generator of $I\hat{\mathcal O}_{U_1,q}$ for $2\le i\le m-1$.
 Then $x^{r_m}y^{r_m}$ is a local generator of $I\hat{\mathcal O}_{U_1,q}$, and 
 then $G'$ has  an expression
 $$
 G'=1+x_1\Omega_1+y_1\Omega_2
 $$
 for some  $\Omega_1,\Omega_2\in \hat{\mathcal O}_{U_1,q}$. 
 
 We now claim that after  replacing $G'$ with
 $F'=G'-\frac{\overline P(x_1^{t_1}y_1^{t_2})}{x_1^{ce_{11}+de_{21}+s}y_1^{ce_{12}+de_{22}+t}}$,
 where $\overline P$ is defined by (\ref{eq30}), we have that
 $F'(0,0,0)\ne 0$.  If this were not the case, we would have
 $$
 \begin{array}{lll}
 0&=&\mbox{Det}\left(\begin{array}{ll}
 (c+r_m)e_{11}+(d+s_m)e_{21}&(c+r_m)e_{12}+(d+s_m)e_{22}\\
 ae_{11}+be_{21}&ae_{12}+be_{22}
 \end{array}
 \right)\\
 &=&\mbox{Det}\left(\begin{array}{ll}
 c+r_m&d+s_m\\
 a&b
 \end{array}\right)
 \mbox{Det}\left(\begin{array}{ll}
 e_{11}&e_{12}\\
 e_{21}&e_{22}
 \end{array}\right).
 \end{array}
 $$
 Since $e_{11}e_{22}-e_{21}e_{12}\ne 0$ (by our assumption), we get
 $$
 0=\mbox{Det}\left(\begin{array}{ll}
 c+r_m&d+s_m\\
 a&b
 \end{array}\right)
 $$
 which is a contradiction to our assumption that $F$ satisfies (\ref{2-point}).
 Since
 $F'(0,0,0)\ne 0$, we have that $\sigma_{D}(q)=0<m-1=\sigma_D(p)$.
 
 Now suppose that $q$ is a 2-point of $\psi^{-1}(\overline D)$ with $e_{11}e_{22}-e_{21}e_{12}= 0$ in (\ref{eq31}).
 
 We make a substitution 
 $$
 \hat x_1=x_1(z_1+\alpha)^{\phi_1},
 \hat y_1=y_1(z_1+\alpha)^{\phi_2},
 \hat z_1=z_1
 $$
 where $\alpha=\hat\alpha$ and $\phi_1,\phi_2\in\QQ$ satisfy
 $$
 \begin{array}{lll}
 0&=& a(\phi_1e_{11}+\phi_2e_{12}+e_{13})+b(\phi_1e_{21}+\phi_2e_{22}+e_{23})\\
 &=& \phi_1(ae_{11}+be_{21})+\phi_2(ae_{12}+be_{22})+ae_{13}+be_{23}.
 \end{array}
 $$
 We have $ae_{11}+be_{21}>0$ and $ae_{12}+be_{22}>0$ since $a,b>0$ and by the condition satisfied by  the $e_{ij}$ stated after (\ref{eq31}).
 
 Let
 $$
 \lambda_1=\phi_1e_{11}+\phi_2e_{12}+e_{13},
 \lambda_2=\phi_1e_{21}+\phi_2e_{22}+e_{23},
 \lambda_3=\phi_1e_{31}+\phi_2e_{32}+e_{33}.
 $$

 Then $x_1,y_1, z_1$ are  permissible  parameters at $q$ such that
 \begin{equation}\label{eq33}
 x=x_1^{e_{11}}y_1^{e_{12}}(z_1+\alpha)^{\lambda_1}, y=x_1^{e_{21}}y_1^{e_{22}}(z_1+\alpha)^{\lambda_2}, z=x_1^{e_{31}}y_1^{e_{32}}(z_1+\alpha)^{\lambda_3}
 \end{equation}
 with $\lambda_1,\lambda_2,\lambda_3\in\QQ$, and $a\lambda_1+b\lambda_2=0$.

Now suppose that $e_{11}>0$ and $e_{12}>0$, which is the case where $q$ is a 2-point of $D_{U_1}$.  
Write 
$$
u=((x_1^{e_{11}}y_1^{e_{12}})^{a}(x_1^{e_{21}}y_1^{e_{22}})^{b})^{\ell}
=(x_1^{t_1}y_1^{t_2})^{\ell_1}
$$
where $t_1,t_2,\ell_1\in \ZZ_+$ and $\mbox{gcd}(t_1,t_2)=1$.
 
 We then have an expression
 $$
 v=P((x_1^{t_1}y_1^{t_2})^{\frac{\ell_1}{\ell}})+x_1^{ce_{11}+de_{21}}y_1^{ce_{12}+de_{22}}G,
 $$
 where
 $$
 \begin{array}{lll}
 G&=&
 (z_1+\alpha)^{c\lambda_1+d\lambda_2}[\tau_0 x_1^{me_{31}}y_1^{me_{32}}(z_1+\alpha)^{m\lambda_3}\\
 &&+ \tau_2 x_1^{r_2e_{11}+s_2e_{21}+(m-2)e_{31}}y_1^{r_2e_{12}+s_2e_{22}+(m-2)e_{32}}(z_1+\alpha)^{r_2\lambda_1+s_2\lambda_2+(m-2)\lambda_3}+\cdots\\
 && + \tau_{m-1}x_1^{r_{m-1}e_{11}+s_{m-1}e_{21}+e_{31}}y_1^{r_{m-1}e_{12}+s_{m-1}e_{22}+e_{32}}(z_1+\alpha)^{\lambda_1r_{m-1}+\lambda_2s_{m-1}+\lambda_3}\\
 &&+\tau_m x_1^{r_me_{11}+s_me_{21}}y_1^{r_me_{12}+s_me_{22}}(z_1+\alpha)^{r_m\lambda_1+s_m\lambda_2}].
 \end{array}
 $$
  Let $x_1^sy_1^t$ be a  generator of $I\hat{\mathcal O}_{U_1,q}$. Let $G'=\frac{F}{x_1^{s}y_1^t}$.
 
 We will now establish that, with our assumptions, there is a unique element of
 the set $S$ consisting of 
 $z^m$,  and  
$$
\{x^{r_i}y^{s_i}z^{m-i}\mid 2\le i\le m\mbox{ and } \tau_{i}\ne 0\}
 $$
which is a generator of $I\hat{\mathcal O}_{U_1,q}$; that is, is equal to $x_1^sy_1^t$ times a unit in $\hat{\mathcal O}_{U_1,q}$.
Let $r_0=0$ and  $s_0=0$. Suppose that 
$x^{r_i}y^{r_i}z^{m-i}$ (with $0\le i\le m$) is a generator of 
$I\hat{\mathcal O}_{U_1,q}$.  We have $x^{r_i}y^{s_i}z^{m-i}=x_1^sy_1^t(z_1+\alpha)^{\gamma_i}$
where
$$
\begin{array}{l}
r_ie_{11}+s_ie_{21}+(m-i)e_{31}=s\\
r_ie_{12}+s_ie_{22}+(m-i)e_{32}=t\\
r_i\lambda_1+s_i\lambda_2+(m-i)\lambda_3=\gamma_i.
\end{array}
$$
Let 
\begin{equation}\label{eq6}
A=\left(\begin{array}{lll}
e_{11}&e_{21}&e_{31}\\
e_{12}&e_{22}&e_{32}\\
\lambda_1&\lambda_2&\lambda_3
\end{array}\right).
\end{equation}
We have 
\begin{equation}\label{eq7}
A\left(\begin{array}{l} r_i\\s_i\\m-i\end{array}\right)=\left(\begin{array}{l}s\\t\\ \gamma_i\end{array}\right).
\end{equation}
Let $\omega=\mbox{Det}(A)$. 
$$
A=
\left(\begin{array}{lll}
1&0&0\\
0&1&0\\
\phi_1&\phi_2&1\end{array}\right)
\left(\begin{array}{lll}
e_{11}&e_{21}&e_{31}\\
e_{12}&e_{22}&e_{32}\\
e_{13}&e_{23}&e_{33}
\end{array}\right)
$$
implies $\omega=\mbox{Det}(A)=\pm1$.

By Cramer's rule, we have
$$
\begin{array}{lll}
\omega(m-i)&=&\mbox{Det}\left(\begin{array}{lll}
e_{11}&e_{21}&s\\
e_{12}&e_{22}&t\\
\lambda_1&\lambda_2&\gamma_i
\end{array}\right)\\
&=& s\mbox{Det}
\left(\begin{array}{ll}
e_{12}&e_{22}\\
\lambda_1&\lambda_2
\end{array}\right)
-t\mbox{Det}
\left(\begin{array}{ll}
e_{11}&e_{21}\\
\lambda_1&\lambda_2
\end{array}\right)
+\gamma_i\mbox{Det}
\left(\begin{array}{ll}
e_{11}&e_{21}\\
e_{12}&e_{22}
\end{array}\right).
\end{array}
$$
Since $e_{11}e_{21}-e_{12}e_{22}=0$ by assumption, we have that
$$
i=m-\frac{1}{\omega}
\left(s\mbox{Det}
\left(\begin{array}{ll}
e_{12}&e_{22}\\
\lambda_1&\lambda_2
\end{array}\right)
-t\mbox{Det}
\left(\begin{array}{ll}
e_{11}&e_{21}\\
\lambda_1&\lambda_2
\end{array}\right)\right).
$$
In particular, there is a unique element  $x^{r_i}y^{r_i}z^{m-i}\in S$ which is a generator of $I\hat{\mathcal O}_{U_1,q}$. We have $x^{r_i}y^{s_i}z^{m-i}=x_1^st_1^t(z_1+\alpha)^{\gamma_i}$.

We thus have that $G=x_1^sy_1^t[g(z_1+\alpha)^{\gamma_i+c\lambda_1+d\lambda_2}+x_1\Omega_1+y_1\Omega_2]$
for some $\Omega_1,\Omega_2\in\hat{\mathcal O}_{U_1,q}$ and $0\ne g\in \mathfrak k$.

We now establish that we cannot have that $\gamma_i+c\lambda_1+d\lambda_2=0$ and  $x_1^{ce_{11}+de_{21}+s}y_1^{ce_{12}+de_{22}+t}$ is a power of $x_1^{t_1}y_1^{t_2}$. We will suppose that both of these conditions do hold, and derive a contradiction.
Now we know that $x^ay^b=x_1^{ae_{11}+be_{21}}y_1^{ae_{12}+be_{22}}$
is a power of $x_1^{t_1}y_1^{t_2}$. By (\ref{eq6}),  (\ref{eq7}) and our assumptions, we have that
$$
A\left(\begin{array}{l} a\\b\\0\end{array}\right)
$$
and
$$
A\left(\begin{array}{l} c+r_i\\d+s_i\\m-i\end{array}\right)
$$
are rational multiples of 
$$
\left(\begin{array}{l}
t_1\\t_2\\0\end{array}\right).
$$
Since $\omega=\mbox{Det}(A)\ne 0$, we have that $(c+r_i,d+s_i,m-i)$ is a rational multiple of $(a,b,0)$.
Thus $x^cy^dx^{r_i}y^{s_i}z^{m-i}$ is a power of $x^ay^b$, a contradiction to our assumption that $F$ 
satisfies (\ref{2-point}).

Let 
 $$
 \overline P(x_1^{t_1}y_1^{t_2})=\sum_{t_2i-t_1j=0} \frac{1}{i!j!}\frac{\partial (x_1^{ce_{11}+de_{21}}y_1^{ce_{12}+de_{22}}G)}{\partial x_1^i\partial y_1^j}(0,0,0)x_1^iy_1^j,
 $$
 and $F'=G'-\frac{\overline P(x_1^{t_1}y_1^{t_2})}{x_1^{ce_{11}+de_{21}+s}y_1^{ce_{12}+de_{22}+t}}$.
 Set

 $P'(x_1^{t_1}y_1^{t_2})=P((x_1^{t_1}y_1^{t_2})^{\frac{\ell_1}{\ell}})+\overline P(x_1^{t_1}y_1^{t_2})$.
 We have that 
 $$
 u=(x_1^{t_1}y_1^{t_2})^{\ell_1},
 v=P'(x_1^{t_1}y_1^{t_2})+x_1^{ee_{11}+fe_{21}}y_1^{ce_{21}+de_{22}}F'
 $$
 has the form (\ref{2-point})
 and   $\sigma_{D}(q)= 0\le m-2=\sigma_D(p)$.

 Now suppose that $q\in\psi^{-1}(p)$ is a 2-point of $\psi^{-1}(\overline D)$,
 $e_{11}e_{22}-e_{12}e_{21}=0$ in (\ref{eq31}), and $e_{11}=0$ or $e_{12}=0$. Without loss of generality,
 we may assume that $e_{12}=0$. $q$ is a 1-point of $D_{U_1}$, and we have permissible parameters (\ref{eq33}) at $q$.
 Since $\mbox{Det}(e_{ij})=\pm 1$, we have that $e_{32}=1$, and $e_{11}e_{23}-e_{21}e_{13}=\pm 1$.
 Replacing $y_1$ with $y_1(z_1+\alpha)^{\lambda_3}$ in (\ref{eq33}), we find permissible parameters $x_1,y_1,z_1$ at $q$ such that
 \begin{equation}
 \label{eq130}
 x=x_1^{e_{11}}(z_1+\alpha)^{\lambda_1},\,\,
 y=x_1^{e_{21}}(z_1+\alpha)^{\lambda_2},\,\,
 z=x_1^{e_{31}}y_1,
 \end{equation}
 with $e_{11}, e_{21}>0$ and $a\lambda_1+b\lambda_2=0$.
 We have
 $$
 \begin{array}{lll}
 u&=& x_1^{(ae_{11}+be_{21})l}=x_1^{l_1}\\
 v&=& P(x_1^{ae_{11}+be_{21}})+x_1^{ce_{11}+de_{21}}G
 \end{array}
 $$
 where
 $$
 \begin{array}{lll}
 G&=& (z_1+\alpha)^{c\lambda_1+d\lambda_2}[\tau_0x_1^{me_{31}}y_1^m+\tau_2x_1^{r_2e_{11}+s_2e_{21}+(m-2)e_{31}}
 y_1^{m-2}(z_1+\alpha)^{r_2\lambda_1+s_2\lambda_2}+\cdots\\
 &&+\tau_{m-1}x_1^{r_{m-1}e_{11}+s_{m-1}e_{21}+e_{31}}y_1(z_1+\alpha)^{r_{m-1}\lambda_1+s_{m-1}\lambda_2}\\
 &&+\tau_mx_1^{r_me_{11}+s_me_{21}}(z_1+\alpha)^{r_m\lambda_1+s_m\lambda_2}].
 \end{array}
 $$
 Since $I\hat{\mathcal O_{U_1,q}}$ is principal and $\tau_m$ or $\tau_{m-1}\ne 0$, we have that $x_1^{r_me_{11}+s_me_{21}}$ is a generator of $I\hat{\mathcal O}_{U_1,q}$ if $\tau_m\ne 0$,  and $x_1^{r_{m-1}e_{11}+s_{m-1}e_{21}+e_{31}}y_1$ is a generator of $I\hat{\mathcal O}_{U_1,q}$ if $\tau_m=0$ 
 and $\tau_{m-1}\ne 0$.
 
 First suppose that $\tau_m\ne 0$ so that
 $$
 G=x_1^{r_me_{11}+s_me_{21}}[g_m(z_1+\alpha)^{(c+r_m)\lambda_1+(d+s_m)\lambda_2}+x_1\Omega+y_1\Omega_2]
 $$
 with $0\ne g_m\in\mathfrak k$, $\Omega_1,\Omega_2\in\hat{\mathcal O}_{U_1,q}$.
 Since $\lambda_1,\lambda_2$ are not both zero, $a\lambda_1+b\lambda_2=0$ and $a(d+s_m)-b(c+r_m)\ne 0$, we have that $(c+r_m)\lambda_1+(d+s_m)\lambda_2\ne 0$.
 Let $\overline P(x_1)=G(x_1,0,0)$, and $P'(x_1)=P(x_1^{ae_{11}+be_{21}})+\overline P(x_1)$.
 Let 
 $$
 F'=\frac{1}{x_1^{ce_{11}+de_{21}}}(G-\overline P(x_1)).
 $$
 Then 
 $$
 \begin{array}{lll}
 u&=& x_1^{l_1}\\
 v&=& P'(x_1)+x_1^{ce_{11}+d e_{21}}F'
 \end{array}
 $$
 is of the form (\ref{1-point}) with $\mbox{ord }F'(0,y_1,z_1)=1$. Thus $\sigma_D(q)=0<\sigma_D(p)$.
 
 Now suppose that $\tau_m=0$, so that 
 $$
 G=x_1^{r_{m-1}e_{11}+s_{m-1}e_{21}+e_{31}}[
 g_{m-1}y_1(z_1+\alpha)^{(c+r_{m-1})\lambda_1+(d+s_{m-1})\lambda_2}+x_1\Omega_1]
 $$
 with $0\ne g_{m-1}\in\mathfrak k$ and $\Omega_1\in\hat{\mathcal O}_{U_1,q}$. Thus $\sigma_D(q)=0<\sigma_D(p)$.

 The final case is when $q$ is a 3-point for $\psi^{-1}(\overline D)$, so that $q$ is a 3-point or a 2-point of $D_{U_1}$. Then we have permissible parameters $x_1,y_1,z_1$ at $q$
 such that
 $$
 x=x_1^{e_{11}}y_1^{e_{12}}z_1^{e_{13}},
 y=x_1^{e_{21}}y_1^{e_{22}}z_1^{e_{23}},
 z=x_1^{e_{31}}y_1^{e_{32}}z_1^{e_{33}}
 $$
 with $\omega=\mbox{Det}(e_{ij})=\pm 1$. 
 Thus there is a unique element of
 the set $S$ consisting of 
 $z^m$ and  
$$
\{x^{r_i}y^{s_i}z^{m-i}\mid 2\le i\le m\mbox{ and }\overline \tau_{i}\ne 0\}
 $$
which is a generator $x_1^{s_1}y_1^{s_2}z_1^{s_3}$ of $I\hat{\mathcal O}_{U',q}$. 
Thus $\sigma_{D}(q)=0$ if $q$ is a 3-point of $D_{U_1}$. If $q$ is a 2-point of $D_{U_1}$, we may assume that $e_{13}=e_{23}=0$. Then $e_{33}=1$. Since $\tau_m\ne 0$ or $\tau_{m-1}\ne 0$, we calculate that $\sigma_D(q)=0$.

\end{proof}

\section{Global reduction of $\sigma_D$}\label{Section5}

\begin{Lemma}\label{SingLoc}
Suppose that $X$ is 2-prepared and $p\in X$ is 3-prepared. Suppose that $r=\sigma_D(p)>0$.
\begin{enumerate}
\item[a)] Suppose that $p$ is a 1-point. Then there exists a unique curve $C$ in $\mbox{Sing}_1(X)$ containing $p$.
The curve $C$ is contained in $\mbox{Sing}_r(X)$.
If $x,y,z$ are permissible parameters at $p$ giving an expression (\ref{eq3}) or (\ref{eq4}) at $p$, then $x=z=0$ are formal local equations of $C$ at $p$.
\item[b)] Suppose that $p$ is a 2-point and $C$ is  a curve in $\mbox{Sing}_r(X)$ containing $p$. If $x,y,z$ are permissible parameters at $p$ giving an expression (\ref{eq2}) at $p$, then $x=z=0$ or $y=z=0$ are formal local equations of $C$ at $p$.
\end{enumerate}
\end{Lemma}

\begin{proof} We first prove a). Let $\mathcal I\subset \mathcal O_X$ be the ideal sheaf defining the reduced scheme $\mbox{Sing}_1(X)$. Then $\mathcal I_p\hat{\mathcal O}_{X,p}=
\sqrt{(x,\frac{\partial F}{\partial y},\frac{\partial F}{\partial z})}=(x,z)$ is an ideal on $U$ defining $\mbox{Sing}_1(U)$. Thus the unique curve $C$ in $\mbox{Sing}_1(X)$ through $p$ has (formal) local equations
$x=z=0$ at $p$. At  points near $p$ on $C$, a  form (\ref{eq3}) or (\ref{eq4}) continues to hold with $m=r+1$. Thus the curve is in $\mbox{Sing}_r(X)$.

 We now prove b). Suppose that $C\subset \mbox{Sing}_r(X)$ is a curve containing $p$. By Theorem \ref{2-pttheorem}, there exists a toroidal morphism $\Psi:U_1\rightarrow U$ where $U$ is an \'etale cover of an affine
neighborhood of $p$, and $\overline D$ is the local toroidal structure on $U$ defined (formally at $p$) by $xyz=0$, such that all points $q$ of $U_1$ satisfy $\sigma_D(q)<r$. Hence the strict transform on $U_1$ of the preimage of $C$ on $U$ must be empty. Since
$\Psi$ is toroidal for $\overline D$ and $X$ is 3-prepared at $p$, $C$ must have local equations $x=z=0$ or
$y=z=0$ at $p$.
\end{proof}

\begin{Definition}\label{canonical}
Suppose that $X$ is 3-prepared. We  define a canonical sequence of blow ups over a curve in $X$, under the following conditions:
\begin{enumerate}
\item[1)] Suppose that $C$ is a curve in $X$ such that $t=\sigma_D(q)>0$ at the generic point $q$ of $C$,
and all points of $C$ are 1-points of $D$. Then we have that $C$ is nonsingular and $\sigma_D(p)=t$ for all $p\in C$ by  Lemma \ref{SingLoc}.  By Lemma \ref{SingLoc} and Theorem \ref{1-pointres} or \ref{1-pointspec}, there exists a unique
minimal sequence of permissible blow ups of sections over $C$, $\pi_1:X_1\rightarrow X$, such that $X_1$ is 2-prepared and $\sigma_D(p)<t$ for all $p\in\pi_1^{-1}(C)$. We will call the morphism $\pi_1$ the canonical sequence of blow ups over $C$.
\item[2)] Suppose that $C$ is a permissible curve in $X$ which contains a 1-point such that $\sigma_D(p)=0$ for all $p\in C$, and a condition 1), 4) or 6) of Lemma \ref{LemmaC} holds at all $p\in C$. Let $\pi_1:X_1\rightarrow X$ be the blow up of $C$. Then by Lemma \ref{LemmaE}, $X_1$ is 3-prepared and $\sigma_D(p)=0$ for $p\in\pi_1^{-1}(C)$.
We will call the morphism $\pi_1$ the canonical blow up of $C$.
\end{enumerate}
\end{Definition}

\begin{Theorem}\label{maintheorem}  

Suppose that $X$ is 2-prepared. Then there exists a
sequence of permissible blowups $\psi:Y\rightarrow X$ such that $Y$ is  prepared.
\end{Theorem}

Before proving this theorem, we introduce some notation, and give some idea of the main difficulty of the proof.

Suppose that  $p\in X$ is a 2-point such that $X$ is 3-prepared at $p$ and $\sigma_D(p)=r>0$. We can then define 
 $(U_p,\overline D_p,I_p,\nu_p^1,\nu_p^2)$ as in Theorem \ref{2-pttheorem}, where $\nu_p^t$
are valuations on $U_p$ which dominate the two curves $C_1$, $C_2$ which are the intersection of $E_p$ with $D_{U_p}$ on $U_p$ (where $\overline D_p=D_{U_p}+E_p$), and which have the property that if $\pi:V\rightarrow U_p$ is a birational morphism, then the center
$C(V,\nu_p^t)$ of $\nu_p^t$ on $V$ is the unique curve on the strict transform of $E_p$ on $V$ which dominates $C_t$. We will call
$(U_p,\overline D_p,I_p,\nu_p^1,\nu_p^2)$ a local resolver.
We will think of $U_p$ as a germ, so we will
feel free to replace $U_p$ with a smaller neighborhood of $p$ whenever it is convenient.

If $\pi:Y\rightarrow X$ is a birational morphism, then we define 
$C(Y,\nu_p^i)$ to be the closed curve in $Y$ which is the center of $\nu_p^t$ on $Y$. 
We define $\Lambda(Y,\nu_p^t)$ to be the point $C(Y,\nu_p^t)\cap \pi^{-1}(p)$.
This defines a valuation which is composite with $\nu_p^t$.

We define $W(Y,p)$ to be the germ in  $Y$ of the image of points in $\pi^{-1}(U_p)=Y\times_XU_p$ such that $I_p\mathcal O_Y\mid \pi^{-1}(U_p)$ is not invertible. $W(Y,p)$ is a subset of the union of the set of generic points of 2-curves for $\overline D_p$ in $Y\times _XU_p$, and the set of all points of
$\pi^{-1}(p)$.

If $\pi:Y\rightarrow X$ is a morphism, 
define  $\mbox{Preimage}(Y,Z)=\pi^{-1}(Z)$ for
$Z$ a subset of $X$.

Suppose that $\pi:Y\rightarrow X$ is a composition of permissible blow ups which is toroidal for $\overline D_p$ above 
$\overline Y:=\pi^{-1}(U_p)$. The blow up of a three point for $\overline D_p$ or of a 2-curve for $\overline D_p$ which $\pi$ contracts to $p$
extends readily to a permissible blow up of $Y$, as does a permissible blow up of a 2-curve of $D$. The only remaining case of the blow up of a
3-point or 2-curve of $\overline D_p$ on $\overline Y$ is the blow up of one of the  two curves $C(Y,\nu_p^1)$ or $C(Y,\nu_p^2)$. Of course such a curve may
only be permissible over $U_p$. 

We can principalize $I_p$ above $U_p$ by the following algorithm:
First perform any sequence $\overline Y\rightarrow U_p$ consisting of blow ups of 3-points of $\overline D_p$ and 2-curves of $\overline D_p$, 
with the restriction that the map is an isomorphism over the generic points of $C(U_p,\nu_p^1)$ and $C(U_p,\nu_p^2)$.
Now construct $\overline Y_1\rightarrow \overline Y$ be blowing  up  $C(\overline Y,\nu_p^t)$ for some $t$, such that $I_p\mathcal O_{\overline Y,\eta}$ is not principal, where $\eta$ is the generic point of $C(\overline Y,\nu_p^t)$. Then once again perform any sequence of
blow ups $\overline Y_2\rightarrow \overline Y_1$ consisting of blow ups of 3-points of $\overline D_p$ and 2-curves of $\overline D_p$, 
with the restriction that the map is an isomorphism over the generic points of $C(\overline Y_1,\nu_p^1)$ and $C(\overline Y_2,\nu_p^2)$.
Now we define $\overline Y_3\rightarrow \overline Y_2$ to be the blow up of $C(\overline Y_2,\nu_p^t)$ for some $t$, such that $I_p\mathcal O_{\overline Y_2,\xi}$ is not principal, where $\xi$ is the generic point of $C(\overline Y_2,\nu_p^t)$.
A chain of blowups of this type will eventually produce a $\overline Y_n$ such that $I_p\mathcal O_{\overline Y_n,\eta_t}$ is  principal, where
$\eta_t$ is the generic point of $C(\overline Y_n,\nu_p^t)$ for $t=1,2$. If this has been accomplished, then we may perform a final sequence of blowups $\overline Y_{n+1}\rightarrow \overline Y_n$,
consisting of blow ups of 3-points of $\overline D_p$ and 2-curves of $\overline D_p$, 
with the restriction that the map is an isomorphism over the generic points of $C(\overline Y_1,\nu_p^1)$ and $C(\overline Y_2,\nu_p^2)$, such that
$I_p\mathcal O_{\overline Y_{n+1}}$ is locally principal. We thus have that $\sigma_D(q)<r$ for all points $q\in \overline{Y}_{n+1}$ (by Theorem \ref{2-pttheorem}). 

The essential difficulty in extending this local argument to a proof of Theorem \ref{maintheorem} is to extend the local blow ups of $C(\overline Y_i,\nu_p^t)$ to permissible global blow ups above $X$, which do not interfere with the the local resolution procedures above other points of $X$.

 We will construct   sequences 
\begin{equation}\label{eqmt12}
Y_n\rightarrow Y_{n-1}\rightarrow \cdots\rightarrow Y_0=X
\end{equation}
where each $Y_i$ has an associated finite set $S(Y_i)$, which we will often abbreviate as $S(i)$. We require that $S(0)=\emptyset$, and that
$$
\mbox{$S(i)$ is contained in the disjoint union of the $Y_j$ with $j<i$.}
$$
Each morphism $Y_{i+1}\rightarrow Y_i$ is  a  permissible blow up, or the identity map with
$Y_{i+1}=Y_i$ and $S(i+1)=S(i)\cup\{p\}$ for some $p\in Y_i$, which is a 2-point for $D$ with $\sigma_D(p)>0$,  such that $Y_i$ is 3-prepared at $p$, and we introduce a local resolver $(U_p,\overline D_p,\nu_p^1,\nu_p^2)$ at $p$, or $Y_{i+1}=Y_i$ and $S(i+1)$ is a subset of $S(i)$.
We require that $S(i)$ be contained in the disjoint union of the $Y_j$ with $j<i$, and
$p\in S(i)\cap Y_j$ implies $p$ is  a 3-prepared 2-point in $Y_j\setminus \left(\cup_{p'\in S(j)} W(Y_j,p')\right)$, with $\sigma_D(p)>0$, and there is a given
local resolver $(U_p,\overline D_p,\nu_p^1,\nu_p^2)$ in $Y_j$ for $p$. Let $W(Y_i)=\cup_{p'\in S(i)} W(Y_i,p')$. We will often write
$W(i)=W(Y_i)$.
We require that each morphism $Y_{i+1}\rightarrow Y_i$ be an {\it admissible blow up}, which we define to be a permissible blow up such that for all $p\in S(i)$, $Y_{i+1}\rightarrow Y_i$ is toroidal for $\overline D_p$ above a neighborhood of $W(Y_i,p)$.

A sequence (\ref{eqmt12}) will be called an admissible sequence.  In the first approximation, $S(Y_i)$ may be seen as the set of ``bad points'' $p\in Y_j$ (for $j<i$) with
``bad preimages'' in $Y_i$. Their preimages are not fully 3-prepared, or contain  singular points or $I_p\mathcal O_{Y_i}$ is not invertible.
By performing a succession of admissible sequences, we want to obtain that $S(Y_n)=\emptyset$.

Define
$$
\sigma(Y_i):=\max\{ \{\sigma_D(p)\mid p\in Y_i\setminus W(i) \}\cup \{\sigma_D(q)\mid q\in S(i)\}\}.
$$

\begin{Definition}\label{goodcurve} Suppose that  $Y_{i_0}\rightarrow X$ is an admissible sequence, and $C$ is a curve in $D_{Y_i}$ which contains a 1-point of $D$.
Let $\eta$ be the generic point of $C$.
$C$ is called a good curve if one of the following conditions hold:
\begin{enumerate}
\item[1.] If $\sigma_D(\eta)=0$, then $\sigma_D(p)=0$ for all $p\in C\setminus W(i_0)$ and $p\in C\cap W(i_0)$ implies $p=\Lambda(Y_{i_0},\nu_b^t)$ and $C=C(Y_{i_0},\nu_b^t)$ for some
$b\in S(i_0)$ and $t$.
\item[2.] If $\sigma_D(\eta)>0$, then $C\setminus W(i_0)$ is a set of 3-prepared 1-points and $p\in C\cap W(i_0)$ implies $p=\Lambda(Y_{i_0},\nu_b^t)$, $C=C(Y_{i_0},\nu_b^t)$ for some $b\in S(i_0)$ and $t$ (in particular, $p$ is a 2-point of $D$).
\end{enumerate}
\end{Definition}

We will be particularly concerned with sequences (\ref{eqmt12}) which admit expressions 
\begin{equation}\label{eqmt1}
Y=Y_n=Y_{i_s}\rightarrow  \cdots\rightarrow Y_{i_2}\rightarrow Y_{i_1}\rightarrow Y_0=X
\end{equation}
where each $Y_{i_{j+1}}\rightarrow Y_{i_j}$ is the sequence 
$$
Y_{i_{j+1}}\rightarrow Y_{i_{j+1}-1}\rightarrow \cdots\rightarrow Y_{i_j+1}\rightarrow Y_{i_j},
$$
such that each of the $Y_{i_{j+1}}\rightarrow Y_{i_j}$ in (\ref{eqmt1}) is one of the following, called an admissible transformation:
\begin{enumerate}
\item[1.] The blow up of a prepared point of $D$, and $S(i_{j+1})=S(i_j)$.
\item[2.]  The blow up of a 3-point or a 2-curve  of $D$, and $S(i_{j+1})=S(i_j)$.
\item[3.] The blow up of a 3-point or 2-curve for $\overline D_p$, contained in $W(Y_{i_j},p)$  (with $p\in S(i_j)\cap Y_k$), which contracts to $p$ under $Y_{i_j}\rightarrow Y_k$ and
$S(i_{j+1})=S(i_j)$.
\item[4.] $Y_{i_{j+1}}=Y_{i_j}$ and $S(i_{j+1})=S(i_j)\cup\{p\}$ for some $p\in Y_{i_j}\setminus W(i_j)$, which is a 2-point for $D$ such that $Y_{i_j}$ is 3-prepared at $p$, and $\sigma_D(p)>0$, and we introduce a local resolver $(U_p,\overline D_p,\nu_p^1,\nu_p^2)$ at $p$. 
\item[5.] The sequence of permissible blow ups of Proposition \ref{local3prep}, applied to a union of irreducible components $E$ of $D$ such that 
all 2 and 3 points for $D$ in a neighborhood of $E$ are 3-prepared, and
$W(i)\cap E$ contains only
a finite set of 2-points (which we take to be the set $B$ of Proposition \ref{local3prep}), over which $Y_{i_{j+1}}\rightarrow Y_{i_j}$ is an isomorphism. The effect of this transformation is to make all points in
a neighborhood of $\mbox{Preimage}(Y_{i_{j+1}},E)$ 3-prepared. We have $S(i_{j+1})=S(i_j)$.
\item[6.] The ``canonical sequence of blow ups''  above a good curve $C$ in $D_{Y_{i_j}}$ (This transformation will be defined after Lemma \ref{blowupproc}).
We will generally
have $S(i_{j+1})\setminus S(i_j)\ne \emptyset$.
\item[7.] $Y_{i_{j+1}}=Y_{i_j}$ and $S(i_{j+1})=S(i_j)\setminus \{p\in S(i_j)\mid W(Y_{i_j},p)=\emptyset\}$.
\end{enumerate}

\begin{Lemma}\label{mt16} Suppose that (\ref{eqmt1}) is an admissible sequence consisting entirely of admissible transformations of types 1 - 5 and 7. Then  for $0\le j\le n$ in (\ref{eqmt1}), the following conditions (\ref{eqmt7}) - (\ref{eqmt11}) hold:
 \begin{equation}\label{eqmt7}
\mbox{The closed sets $W(Y_{i_j},p)\cap \mbox{Preimage}(Y_{i_j},p)$ are pairwise disjoint for $p\in S(i_j)$. }
\end{equation}

\begin{equation}\label{eqmt8}
\mbox{All points of $Y_{i_j}\setminus W(i_j)$ are 2-prepared.}
\end{equation}

\begin{equation}\label{eqmt9}
\mbox{For $p\in  S(i_j)$, $Y_{i_{j+1}}\rightarrow Y_{i_j}$ is toroidal for $\overline D_p$ above a neighborhood of $W(Y_{i_j},p)$. }
\end{equation}

\begin{equation}\label{eqmt10}
\sigma(Y_{i_{j+1}})\le \sigma(Y_{i_j})
\end{equation}

\begin{equation}\label{eqmt11}
\begin{array}{l}
\mbox{Suppose that $r=\sigma(Y_{i_j})$ and $\sigma(Y_{i_j}\setminus W(i_j))<r$.
Then $\sigma(Y_{i_{j+1}}\setminus W(i_{j+1}))<r$}\\
\mbox{and if $p\in S(i_{j+1})\setminus S(i_j)$, then $\sigma_D(p)<r$.}
\end{array}
\end{equation}
\end{Lemma}

 \begin{proof} For admissible transformations of types 1 - 5 (\ref{eqmt7}) - (\ref{eqmt10}) hold since $D\subset \overline D_p$ for all $p\in S(i_j)$, and
 by Lemma \ref{LemmaA} (for transformations of type 1), Lemma \ref{Torgood} (for transformations of type 2), Theorem \ref{2-pttheorem} (for transformations of type 3) and Propositions \ref{local3prep} (for transformations of type 5).
 
 (\ref{eqmt11}) holds for all admissible transformations of types 1 - 5, since for $p\in S(i_j)$, $0<\sigma_D(p)\le r$. Thus by  Theorem \ref{2-pttheorem},
 we have 
 $$
 \sigma_D(q)<r\mbox{ if }q\in \mbox{Preimage}(Y_{i_{j+1}},W(Y_{i_j},p)),
 $$
  since $Y_{i_{j+1}}\rightarrow Y_{i_j}$ is toroidal for $\overline D_p$
 above a neighborhood of $W(Y_{i_j},p)$ and $I_p\mathcal O_{Y_{i_{j+1}},q}$ is invertible.
 
 \end{proof}

\begin{Remark} (\ref{eqmt7}) tells us that if $p\in Y_{i_j}\cap W(i_j)$ is a (closed) point, then there is a unique $q\in S(i_j)$ such that $p\in W(Y_{i_j},q)$. This observation is important in the structure of the proof of Theorem \ref{maintheorem}.
\end{Remark}

\begin{Lemma}\label{fin3prepA} Suppose that $Y_{i_0}\rightarrow X$ is an admissible sequence, and $r=\sigma(Y_{i_0})>0$. Then there exists an admissible sequence $Y_{i_j}\rightarrow Y_{i_0}$,  consisting of admissible transformations of types 2,3, 4  and 5, such that all points of 
$\mbox{Sing}_r(Y_{i_j}\setminus W(i_j))$ are  3-prepared 1-points. 
 \end{Lemma} 

\begin{proof} 
We will prove the lemma by constructing an admissible sequence 
$$
Y_{i_{n}}\rightarrow Y_{i_{n-1}}\rightarrow  \cdots \rightarrow Y_{i_1}\rightarrow Y_{i_0}
$$
where each $Y_{i_{j+1}}\rightarrow Y_{i_j}$ is an admissible transformation of type 2 or 3 for $j\le n-3$, $Y_{i_{n-1}}\rightarrow Y_{i_{n-2}}$ is an admissible transformation of type 5 (so that $S(i_j)=S(i_0)$ for $j\le n-1$) and
$Y_{i_{n}}\rightarrow Y_{i_{n-1}}$ is  a transformation of type 4, and for all $j$,
\begin{equation}\label{eqmt3}
\begin{array}{l}
\mbox{If $F$ is a component of $D_{Y_{i_j}}$ such that $F\subset \mbox{Preimage}(Y_{i_j},S(i_0))$}\\
\mbox{then $F\cap \mbox{Sing}_r(Y_{i_j}\setminus W(i_j))=\emptyset$.}
\end{array}
\end{equation}

Let $Y_{i_1}\rightarrow Y_{i_0}$ be a sequence of permissible blow ups of 2-curves  of  $D$  such that if $p\in S(i_0)$, then for $j=1$, we have that
\begin{equation}\label{eqmt2}
W(Y_{i_j},p)\subset C(Y_{i_j},\nu_p^1)\cup C(Y_{i_j},\nu_p^2)\cup \mbox{Preimage}(Y_j,p).
\end{equation}
 We have that (\ref{eqmt3}) holds for $j=1$ (by Theorem \ref{2-pttheorem} and since $\sigma(Y_{i_0})=r$).

Let $Y_{i_2}\rightarrow Y_{i_1}$ be a sequence of blow ups of 2-curves and 3-points of $D$,  such that for $j=2$, we have that
\begin{equation}\label{eqmt41}
\begin{array}{l}
\mbox{If $E_1$ and $E_2$ are distinct components of $Y_{i_j}$ such that $E_1$ contains a curve $C(Y_{i_j},\nu_p^s)$}\\
 \mbox{and $E_2$ contains a curve $C(Y_{i_j},\nu_q^t)$for some $p,q\in S(i_0)$ and $s,t$, then $E_1\cap E_2=\emptyset$}
 \end{array}
\end{equation}
and
\begin{equation}\label{eqmt42}
\begin{array}{l}
\mbox{If $E$ is a component of $D_{Y_{i_j}}$, and $p\in S(i_0)$, $t$ are such that}\\ 
\mbox{$\Lambda(Y_{i_j},\nu_p^t)\in E$ but $C(Y_{i_j},\nu_p^t)\not\subset E$,
then $E$ contracts to $p$.}
\end{array}
\end{equation}

Suppose that $p\in S(i_0)$, and $E$ is a component of $D_{Y_{i_2}}$ which contains $C(Y_{i_2},\nu_p^1)$ for some $t$ ($E$ can contain at most one of these two curves). Let
$\eta_t$ be the generic point of $C(Y_{i_2},\nu_p^t)$.

Since (\ref{eqmt2}) holds for $j=2$,  $W(Y_{i_2},p)$ intersects $E$ in a union of 2-curves and 3-points 
for $\overline D_p$ which contract to $p$, as well as (possibly) the point $\eta_t$.

Let $\gamma_{p,t}$ be the  2-curve for $\overline D_p$ in $E$ which contains
the point $\Lambda(Y_{i_2},\nu_p^t)$ (and is not equal to $C(Y_{i_2},\nu_p^t)$). Let 
\begin{equation}\label{eqmt31}
Z=W(Y_{i_2},p)\cap E\setminus\{\gamma_{p,t},\eta_t\}.
\end{equation}
If $W(Y_{i_2},p)\cap E\subset Z\cup \{\eta_t\}$, then let $Y_{i_3}=Y_{i_2}$.
Otherwise,  the 2-curve $\gamma_{p,t}$ for $\overline D_p$ is in $W(Y_{i_2},p)\cap E$.  In this case we let $Y_{i_3}\rightarrow Y_{i_2}$ be a sequence of
blow ups of 2-curves for $\overline D_p$, which are sections over  $\gamma_{p,t}$, and lie in the strict transform of $E$.
Under each such blow up, the strict transform of $E$ maps isomorphically to $E$, so we may in fact identify $E$ with its strict transform and $\gamma_{p,t}$ with its section.
After enough such blow ups,  on the strict transform $E_3$ of $E$ (which is isomorphic to $E$), we have that $\gamma_{p,t}$ is not 
contained in $W(Y_{i_3},p)\cap E_3$.

 Let $G=W(Y_{i_3},p)\cap E_3\setminus\left(C(Y_{i_3},\nu_p^1)\cup C(Y_{i_3},\nu_p^2) \right)$. $G$ is a closed subset of $E_3$ which is disjoint from $C(Y_{i_3},\nu_p^1)\cup C(Y_{i_3},\nu_p^2)$. Thus there exists an open neighborhood $V$ of $G$ in $\mbox{Preimage}(Y_{i_3},U_p)$
which is disjoint from $C(Y_{i_3},\nu_p^1)\cup C(Y_{i_3},\nu_p^2)$. There exists a sequence of blow ups of 3-points and 2-curves for $\overline D_p$ (which contract to $p$)
$V_1\rightarrow V$  such that $I_p| \mathcal O_{V_1}$ is locally principal.  $V_1\rightarrow V$ extends to an admissible sequence of transformations of type 3, $Y_{i_4}\rightarrow Y_{i_3}$,
such that the strict transform $E_4$ of $E$ on $Y_{i_4}$ satisfies
$$
W(Y_{i_4},p)\cap E_4\subset C(Y_{i_4},\nu_p^1)\cup C(Y_{i_4},\nu_p^2).
$$

Repeat this last step (the construction of $Y_{i_4}\rightarrow Y_{i_2}$) for all $p\in S(i_0)$ and components $E$ of $D_{Y_{i_4}}$ which contain $C(Y_{i_{4}},\nu_p^t)$ for some $p\in S(i_0)$ and $t$, remembering that (\ref{eqmt41}) holds, 
to obtain $Y_{i_5}\rightarrow Y_{i_4}$ where (\ref{eqmt3}), (\ref{eqmt2}), (\ref{eqmt41}) and (\ref{eqmt42}) continue to hold for  $j=5$, and we also have that for $j=5$,
\begin{equation}\label{eqmt43}
\begin{array}{l}
\mbox{If $p\in S(i_0)$ and $E$ is a component of $D_{Y_{i_j}}$ such that $C(Y_{i_j},\nu_p^t)\subset E$}\\
\mbox{for some $t$, then $W(Y_{i_j},p)\cap E\subset C(Y_{i_j},\nu_p^1)\cup C(Y_{i_j},\nu_p^2)$.}
\end{array}
\end{equation}

Suppose that  $E$ is a component of $D_{Y_{i_5}}$ such that $E\cap \mbox{Sing}_r(Y_{i_5}\setminus W(i_5))\ne \emptyset$,
 and $p\in S(i_0)$. If $C(Y_{i_5},\nu_p^t)\subset E$ for some $t$, then $E\cap W(Y_{i_5},p)\subset \{\eta_t\}$, where $\eta_t$ is the generic point of $C(Y_{i_5},\nu_p^t)$
(by (\ref{eqmt43})). If $C(Y_{i_5},\nu_p^t)\not\subset E$ for $t=1,2$, then $\Lambda(Y_{i_5},\nu_p^1),\Lambda(Y_{i_5},\nu_p^2)\not\in E$ by (\ref{eqmt42})
and (\ref{eqmt3}),
and thus by (\ref{eqmt43}), we have that $E\cap W(i_5,p)\cap \left(C(Y_{i_5},\nu_p^1)\cup C(Y_{i_5},\nu_p^2)\right)=\emptyset$.

Thus we can construct an allowable sequence of transformations of type 3, $Y_{i_6}\rightarrow Y_{i_5}$, so that if $E_6$ is the strict transform on $Y_{i_6}$ of a component $E$ of $D_{Y_{i_5}}$ such that $E\cap \mbox{Sing}_r(Y_{i_6}\setminus W(i_6))\ne\emptyset$, then 
\begin{equation}\label{eqmt44}
W(Y_{i_6})\cap E\subset \{\eta_{p,t}\mid \eta_{p_t}\mbox{ is the generic point of a curve $C(Y_{i_6},\nu_p^t)$ which lies on $E$\}.}
\end{equation}
By (\ref{eqmt3}), we have that all exceptional components $F$ of $Y_{i_6}\rightarrow Y_{i_5}$  satisfy $F\cap \mbox{Sing}_r(Y_{i_6}\setminus W(i_6))=\emptyset$. Thus all components $E$ of $Y_{i_6}$ which satisfy $E\cap\mbox{Sing}_r(Y_{i_6}\setminus W(i_6))\ne \emptyset$ must satisfy (\ref{eqmt44}). Thus for $j=6$, we have that 
\begin{equation}\label{eqmt45}
\begin{array}{l}
\mbox{If $E$ is a component of $D_{Y_{i_j}}$ such that $E\cap\mbox{Sing}_r(Y_{i_j}\setminus W(i_j))\ne\emptyset$, then}\\ 
W(Y_{i_j})\cap E\subset \{\eta_{p,t}\mid \eta_{p_t}\mbox{ is the generic point of a curve $C(Y_{i_j},\nu_p^t)$ which lies on $E$\}.}
\end{array}
\end{equation}

By Lemmas \ref{Torgood} and  \ref{2-prep1}, there exists a further sequence $Y_{i_7}\rightarrow Y_{i_6}$ of  blow ups of 3-points and 2-curves of $D$,   such that 
$Y_{i_7}\setminus W(i_7)$ is 3-prepared, except possibly at a finite number of 1-points. The conditions of equations (\ref{eqmt3}), (\ref{eqmt2}),  (\ref{eqmt41}), (\ref{eqmt42})  and (\ref{eqmt45}) continue to hold on $Y_{i_7}$ (although we may have that  some 2-curves for $D$ are blown up which do not contract to points of $S(i_0)$).

We now apply Proposition \ref{local3prep} to the union $H$ of irreducible components $E$ of $D$ for $Y_{i_7}$ which contain a point of
$\mbox{Sing}_r(Y_{i_7}\setminus W(i_7))$, with
$$
A=\{q\in H\mid Y_{i_7}\mbox{ is not 3-prepared at $q$ (which are necessarily 1-points of $D$)}\}
$$
being sure that none of the finitely many 2-points for $D$
$$
B=\{\Lambda(Y_{i_7},\nu_p^t)\mid p\in S(i_0)\}
$$
are in the image of the general curves blown up, to construct an admissible transformation $Y_{i_8}\rightarrow Y_{i_7}$ of type 5, so that if $E$ is an irreducible component of $D$ for $Y_{i_8}$ which contains a point of $\mbox{Sing}_r(Y_{i_8}\setminus W(i_8))$, then all points of 
$E\setminus W(i_8)$ are 3-prepared. We also will have that the conditions of (\ref{eqmt3}), (\ref{eqmt2}),  (\ref{eqmt41}), (\ref{eqmt45})  and (\ref{eqmt45})  hold on points of $E$.

We now perform a sequence of admissible transformations of type 4, introducing local resolvers at all 2-points $p\in Y_{i_8}\setminus W(i_8)$ such that $\sigma_D(p)=r$ (the finite set of these points are all necessarily 3-prepared).

\end{proof}

\begin{Lemma}\label{fin3prepB} Suppose that $Y_{i_0}\rightarrow X$ is an admissible sequence, and $C$ is a curve contained in $D_{Y_i}$ such that $C$ is not a 2-curve and $C\not\subset W(i_0)$.
Let $\eta$ be the generic point of $C$.
Then there exists an admissible sequence $Y_{i_j}\rightarrow Y_{i_0}$, consisting of admissible transformations of types 2, 3, 4 and 5, such that if $C_j$ is the strict transform of $C$ in $Y_{i_j}$,
then 
\begin{enumerate}
\item[1.] If $\sigma_D(\eta)>0$, then all points of $C_j\setminus W(i_j)$ are 3-prepared 1-points.
\item[2.] If $\sigma_D(\eta)=0$, then all points $q$ of $C_j\setminus W(i_j)$ are  1-points or 2-points with $\sigma_D(q)=0$.
\end{enumerate}
\end{Lemma} 

\begin{proof} The proof follows from the arguments of the proof of Lemma \ref{fin3prepA}, applied only to the component $E$ of $D$ containing $C$.
In the case where $\sigma_D(\eta)=0$, the set $A$ of the hypotheses of Proposition \ref{local3prep} used in the construction, will be the union of
the set of 1-points of the strict transform of $E$ which are not 3-prepared, and the 1-points $q$ on the strict transform of $C$ such that $\sigma_D(q)>0$.
\end{proof}

\begin{Lemma}\label{A} Suppose that $Y_{i_0}\rightarrow X$ is an admissible sequence, and $C$ is a curve in $D_{Y_{i_0}}$ which contains a 1-point. Suppose
that $p\in C\cap W(i_0)$ is a 2-point for $D$. Then there exists an admissible sequence $Y_{i_j}\rightarrow Y_{i_0}$, consisting entirely of transformations of types 2 and 3, such that if $C_j$ is the strict transform of $C$ in $Y_{i_j}$, then the following holds.
Suppose that $q\in \mbox{Preimage}(Y_{i_j},p)\cap C_j$. Then $q$ is a 2-point for $D$,  and we further have that if $q\in W(i_j)$, then $q=\Lambda(Y_{i_j},\nu_b^t)$
and $C_j=C(Y_{i_j},\nu_b^t)$ for some $b\in S(i_j)$.
\end{Lemma}

\begin{proof} We have  that $b\in C\cap W(Y_{i_0},b)$ for some  $b\in S(i_0)$. If $C=C(Y_{i_0},\nu_b^t)$ for some $t$, then
we have obtained the conclusions of the lemma, so suppose that $C\ne C(Y_{i_0},\nu_b^t)$ for any $t$.
Since $C$ is not a 2-curve for $D$, there
exists a sequence of blow ups of 3-points for $\overline D_b$, $Y_{i_1}\rightarrow Y_{i_0}$, such that the strict transform $C_1$ of $C$ on $Y_{i_1}$
has the property that the set $C_1\cap\mbox{Preimage}(Y_{i_1},p)$ consists of 2-points for $D$.  We further may obtain that either $C_1\cap \mbox{Preimage}(Y_{i_1},p)$ is disjoint from $W(Y_{i_1},b)$, in which case we have achieved the conclusions of the lemma, or that 
$$
C_1\cap \mbox{Preimage}(Y_{i_1},p)\mbox{  has non trivial intersection with }W(Y_{i_1},b),
$$
 but
$\Lambda(Y_{i_1},\nu_b^t)\not\in C_1$ for any $t$.
Assume that this last case holds, and $q\in C_1\cap \mbox{Preimage}(Y_{i_1},p)$. Then there is a unique 2-curve $\gamma$ of $\overline D_b$, which is also a 2-curve for $D$, 
such that $q\in \gamma$. There is a finite sequence of blow ups $Y_{i_2}\rightarrow Y_{i_1}$ of 2-curves for $\overline D_b$, which are sections over $\gamma$, such that if $C_2$ is the strict transform of $C_1$ in $Y_{i_2}$, and $a\in C_2\cap\mbox{Preimage}(Y_{i_2},q)$, then $I_b\mathcal O_{Y_{i_2},a}$ is principal,
so that $C_2\cap \mbox{Preimage}(Y_{i_2},q)$ is disjoint from $W(i_2)$.

We now apply this procedure above any other points of $C_1\cap \mbox{Preimage}(Y_{i_1},p)$, to construct a further sequence of blow ups of 2-curves $Y_{i_3}\rightarrow Y_{i_2}$ such that the strict transform $C_3$ of $C_2$ on $Y_{i_3}$ satisfies the condition that $C_3\cap \mbox{Preimage}(Y_{i_3},p)$ is disjoint from $W(i_3)$.
\end{proof}

 \begin{Lemma}\label{blowupproc} Suppose that $Y_{i_0}\rightarrow X$ is an admissible sequence and $C$ is a curve in $D_{Y_{i_0}}$ which contains
 a 1-point. Suppose that $p\in C$ is a 2-point. Then there exists an admissible sequence $Y_{i_j}\rightarrow Y_{i_0}$, consisting entirely of 
 transformations of types 2, 3 and 4, satisfying the following properties. Let $C_j$ be the strict transform of $C$ in $Y_{i_j}$.  Suppose that $q\in\mbox{Preimage}(Y_{i_j},p)\cap C_j$. Then $q$ is a 2-point for $D$, and one of the following holds:
 \begin{enumerate}
 \item[1.] There exists $a\in S(i_j)$ such that $q=\Lambda(Y_{i_j},\nu_a^t)$ and $C_j=C(Y_{i_j},\nu_a^t)$ for some $t$, or
 \item[2.] $\sigma_D(q)=0$ and $q\not\in W(i_j)$.
 \end{enumerate}
 \end{Lemma}

\begin{proof} First suppose that $p\in W(i_0)$. Then there exists a  point $b\in S(i_0)$ such that $p\in W(Y_{i_0},b)$. Perform Lemma \ref{A} to construct an allowable sequence $Y_{i_1}\rightarrow Y_{i_0}$ such that if  $C_1$ is the strict transform of $C$ on $Y_{i_1}$, and $q\in C_1\cap \mbox{Preimage}(Y_{i_0},p)$ is contained in $W(i_1)$, then 
there exists 
$a\in S(i_1)$ such that $q=\Lambda(Y_{i_1},\nu_a^t)$ and $C_1=C(Y_{i_1},\nu_a^t)$ for some $a$.
 Let
$$
\lambda(i_1):=\max\{\sigma_D(q)\mid q\in \left(C_1\cap \mbox{Preimage}(Y_{i_1},p)\right)\setminus W(i_1)\}.
$$
We have that 
$$
\lambda(i_1)<\sigma_D(p).
$$

If $p\not\in W(i_0)$, then we let $Y_{i_1}=Y_{i_0}$, $S(i_1)=S(i_0)$ and $\lambda(i_1)=\sigma_D(p)$.

The rest of the proof is the same for both cases considered above ($p\in W(i_0)$ and $p\not\in W(i_0)$).

Now perform Lemma \ref{2-prep1} to construct a sequence of blow ups of 2-curves for $D$, $Y_{i_2}\rightarrow Y_{i_1}$,
such that if  $C_2$ is the strict transform of $C_1$ on $Y_{i_2}$, then all points of $\left(\mbox{Preimage}(Y_{i_2},p)\cap C_2\right)\setminus W(i_2)$ (which are necessarily 2-points for $D$) are 3-prepared. Let 
$$
R(i_2)=\{q\in \left(\mbox{Preimage}(Y_{i_2},p)\cap C_2\right) \setminus W(i_2)\mid q \mbox{ is a 2-point and }\sigma_D(q)>0\}.
$$
Write $R(i_2)=\{q_1,\ldots,q_m\}$. For each $q_i\in R(i_2)$, let $(U_{q_i},\overline D_{q_i},I_{q_i},\nu_{q_i}^1,\nu_{q_i}^2)$ be a local resolver
in $Y_{i_2}$.
Let $Y_{i_3}\rightarrow Y_{i_2}$ be the admissible sequence consisting of transformations of type 4, where $S(i_3)=S(i_2)\cup R(i_2)$.
Let $C_3=C_2$, the strict transform of $C$ on $Y_{i_3}$. If $q\in \left(\mbox{Preimage}(Y_{i_3},p)\cap C_3\right) \setminus W(i_3)$. then
$\sigma_D(q)=0$.
If $q\in \left(\mbox{Preimage}(Y_{i_3},p)\cap C_3\right)$ and $q\in R(i_2)=S(i_3)\setminus S(i_2)$, then we have 
$$
\sigma_D(q)\le \lambda(i_1).
$$
Now again perform Lemma \ref{A}, to construct $Y_{i_4}\rightarrow Y_{i_3}$ such that if  $C_4$ be the strict transform of $C_3$ on $Y_{i_4}$, and  $q\in \left(\mbox{Preimage}(Y_{i_4},p)\cap C_4\right)\cap W(i_4)$, then $q=\Lambda(Y_{i_4},a)$ and $C_4=C(Y_{i_4},\nu_a^t)$ for some $a\in S(i_4)$ and $t$.
 If $\left(\mbox{Preimage}(Y_{i_4},p)\cap C_4\right)\cap W(i_4)\ne\emptyset$, we have that
$$
\lambda(i_3):=\max\{\sigma_D(q)|q\in \left(\mbox{Preimage}(Y_{i_4},p)\cap C_4\right)\cap W(i_4)\}< \lambda(i_1).
$$
Iterate the above, performing  Lemma \ref{2-prep1} followed by a sequence of adimissible transformations of type 4, and then performing Lemma \ref{A}, to eventually obtain $Y_{i_j}\rightarrow Y_{i_0}$ such that if $C_{i_j}$ is the strict transform of $C$ on $Y_{i_j}$, then 
$\sigma_D(q)=0$ if $q\in  \left(\mbox{Preimage}(Y_{i_j},p)\cap C_{i_j}\right) \setminus W(i_j)$, 
and if $q\in  \left(\mbox{Preimage}(Y_{i_j},p)\cap C_{i_j}\right)\cap W(i_j)$, then  $q=\Lambda(Y_{i_j},b)$ and $C_{i_j}=C(Y_{i_j},\nu_b^t)$ for some
$b\in S(i_j)$ and $t$.

\end{proof} 

We now define an admissible  transformation of type 6.  Suppose that $Y_{i_0}\rightarrow X$ is an admissible sequence, and $C$ is a good curve on
$Y_{i_0}$ (Definition \ref{goodcurve}). 

First assume that $\sigma_D(\eta)=0$, where $\eta$ is the generic point of $C$.  By Lemmas \ref{LemmaA} - \ref{LemmaD}, there exists a sequence of transformations of type 1
$Y_{i_1}\rightarrow Y_{i_0}$ such that the strict transform $C_{1}$ of $C$ in $Y_{i_1}$ is such  that $\sigma_D(q)=0$
and the other assumptions of Lemma \ref{LemmaE} hold for all
$q\in C_{1}\setminus W(i_1)$. Let $Y_{i_2}\rightarrow Y_{i_1}$ be the blow up of $C$ which is an admissible blow up. We have that $\sigma_D(q)=0$ for all 
$q\in \mbox{Preimage}(Y_{i_2}, C_{1}\setminus W(i_1))$ by Lemma \ref{LemmaE}. We define the morphism $Y_{i_2}\rightarrow Y_{i_0}$ to be the
transformation  of type 6 associated to $C$.

Now assume that $\sigma_D(\eta)>0$, where $\eta$ is the generic point of $C$. 
Let $Z\rightarrow Y_{i_0}\setminus \left(W(i_0)\cup D_{Y_{i_0}}\right)$
be the canonical sequence of blow ups above $C\setminus W(i_0)$ defined in 1) of Definition \ref{canonical}.
$Z\rightarrow Y_{i_0}\setminus \left(W(i_0)\cup D_{Y_{i_0}}\right)$ has
a factorization 
$$
Z=Z_m\rightarrow Z_{m-1}\rightarrow \cdots \rightarrow Z_1\rightarrow Z_0=Y_{i_0}\setminus \left(W(i_0)\cup D_{Y_{i_0}}\right)
$$
where each $Z_{j+1}\rightarrow Z_j$ is the blow up of a curve $A_j$ which is a section over $C\setminus W(i_0)$, and is permissible for $D$
(thus $A_j$ is either a 2-curve, or consists entirely of 1-points).
We will inductively extend these morphisms (to an admissible sequence
$$
X_m\rightarrow V_{m-1}\rightarrow X_{m-1}\rightarrow \cdots X_3\rightarrow V_2\rightarrow X_2\rightarrow V_1\rightarrow X_1\rightarrow Y_{i_0},
$$
so that 
$$
\mbox{Preimage}(V_{j}, Y_{i_0}\setminus \left(W(i_0)\cup D_{Y_{i_0}}\right))=\mbox{Preimage}(X_j,Y_{i_0}\setminus \left(W(i_0)\cup D_{Y_{i_0}}\right))=Z_j
$$
for all $j$.

We define $X_1$ to be the blow up of $C$ (which is an admissible blow up). 
If $A_j$ is a 2-curve for $D$, then $V_j\rightarrow X_j$ will just be the identity map (with $S(V_j)=S(X_j)$).

if $A_j$ is not a 2-curve, then let $\gamma_0$ be the Zariski closure of $A_j$ in $X_j$. $\gamma_0\setminus Z_j$ is a set of 2-points and 3-points for $D$.
First define a sequence $T_1\rightarrow X_j$ of blow ups of 3-points for $D$, so that the Zariski closure $\gamma_1$ of $A_j$ in $T_1$ is such that
$\gamma_1\setminus A_j$ consists only
of 2-points. Now successively apply Lemma \ref{blowupproc} to the points of  $\gamma_1\setminus A_j$ to construct an admissible sequence
$T_2\rightarrow T_1$ consisting of transformations of types 2, 3 and 4, so that if $\gamma_2$ is the Zariski closure of $A_j$ in $T_2$, 
and $q\in \gamma_2\setminus A_j$, then either $\sigma_D(q)=0$ and $q\not\in W(T_2)$, or there exists $a\in S(T_2)$ such that $q=\Lambda(T_2,\nu_a^t)$ and $C_j=C(T_2,\nu_a^t)$.

A point in $\gamma_1\setminus A_j$ cannot be  contained in a 2-curve which is a section over $C$, since 
$\gamma_1\setminus A_j$ contains no 3-points, and the points of $C\cap W(i_0)$ are all 2-points  for $D$.
Thus $T_2\rightarrow T_1$ has the property that 
$\mbox{Preimage}(T_2,  Y_{i_0}\setminus \left(W(i_0)\cup D_{Y_{i_0}}\right))=Z_j$.

Let $\eta_j$ be the generic point of $A_j$. Then $\sigma_D(\eta_j)>0$ (by Theorem \ref{1-pointspec}). Thus all points of $q\in \gamma_2$ satisfy $\sigma_D(q)\ge \sigma_D(\eta_j)>0$. We then define $V_j$ to be $T_2$.
 
We now define $X_{j+1}\rightarrow V_j$ to be the  blow up of $\gamma_2$, which is an admissible blow up.

\begin{Lemma}\label{Lemmamt14} Suppose that (\ref{eqmt1}) is an admissible sequence consisting of admissible transformations of types 1 - 7. Then for
any transformation $Y_{i_{j+1}}\rightarrow Y_{i_j}$ in (\ref{eqmt1}), the conditions (\ref{eqmt7}) - (\ref{eqmt11}) hold. 
\end{Lemma}

The proof of Lemma \ref{Lemmamt14} follows from our construction of an admissible transformation of type 6, and Theorem \ref{1-pointspec}, Lemma \ref{LemmaE} and Lemma \ref{mt16}.

\begin{Proposition}\label{Propmt1} Suppose that $Y_{i_0}\rightarrow X$ is an admissible sequence. Let $r=\sigma(Y_i)>0$. Then there exists an admissible sequence $Y_{i_j}\rightarrow Y_{i_0}$ such that $\sigma(Y_{i_j})\le r$ and
$\sigma_D(p)<r$ for all $p\in Y_{i_j}\setminus W(i_j)$.
\end{Proposition}

\begin{proof}
First perform Lemma \ref{fin3prepA}, to obtain an admissible sequence $Y_{i_1}\rightarrow Y_{i_0}$ such that $\Gamma(Y_{i_1})=\mbox{Sing}_r(Y_{i_1}\setminus W(i_1))$ 
consists of  3-prepared 1-points. By Lemma \ref{SingLoc}, $\Gamma(Y_{i_1})$ is a disjoint union of nonsingular curves.

 Suppose that $C$ is the closure in $Y_{i_1}$ of a curve in $\Gamma(Y_{i_1})$. By Lemma \ref{blowupproc}, there exists an admissible sequence
 $Y_{i_2}\rightarrow Y_{i_2}$ consisting of transformations of types 2, 3 and 4 such that the strict transform $C_2$ of $C$ in $Y_{i_2}$ is a good curve.
  We may thus perform an admissible transformation
 of type 6, $Y_{i_3}\rightarrow Y_{i_2}$ to get that all points $q$ of $\mbox{Preimage}(Y_{i_3},C_2\setminus W(i_2))$ are 2-prepared for $D$ with $\sigma_D(q)\le r-1$ (by Theorem \ref{1-pointspec}). Further, $\sigma_D(q)\le r-1$ for $q\in \mbox{Preimage}(Y_{i_3},W(i_1))\setminus W(i_3)$. 
 We now apply Lemma \ref{blowupproc} followed by an admissible transformation of type 6 for the other curves of $\Gamma(Y_{i_1})$, to obtain the conclusions of the Proposition.
 \end{proof}

\begin{Proposition}\label{Propmt2} Suppose that $Y_{i_0}\rightarrow X$ is an admissible sequence, $r=\sigma(Y_{i_0})>0$ and $\sigma_D(p)<r$ if $p\in Y_{i_0}\setminus W(i_0)$. Then there exists an admissible sequence
$Y_{i_j}\rightarrow Y_{i_0}$ such that $\sigma(Y_{i_j})<r$.
\end{Proposition}

\begin{proof}
Let 
$$
T(i_0)=\{p\in S(i_0)\mid \sigma_D(p)=r\}.
$$
Suppose there exists $p\in T(i_0)$ and $t$ such that $I_p\mathcal O_{Y_{i_0},\eta}$ is not principal, where $\eta$ is the generic point of 
$C(Y_{i_0},\nu_p^t)$. First apply  Lemma \ref{fin3prepB} to $C(Y_{i_0},\nu_p^t)$ to construct an admissible sequence $Y_{i_1}\rightarrow Y_{i_0}$
so that all points $q$ of $C(Y_{i_2},\nu_p^t)\setminus W(i_2)$  are 3-prepared 1-points if $\sigma_D(\eta)>0$ and are 1-points or 2-points which satisfy $\sigma_D(q)=0$ if $\sigma_D(\eta)=0$.
Then successively apply Lemma \ref{blowupproc} to all 2-points  $q$ of $C(Y_{i_2},\nu_p^t)$ which have $\sigma_D(q)>0$, to construct an admissible sequence $Y_{i_2}\rightarrow Y_{i_1}$
such that $C(Y_{i_2},\nu_p^t)$ (which is the strict transform of $C(Y_{i_0},\nu_p^t)$) is a good curve. 
Let $Y_{i_3}\rightarrow Y_{i_2}$ be a transformation of type 6 applied to  $C(Y_{i_2},\nu_p^t)$. 
We continue to have $\sigma(Y_{i_3})<r$ and if $p\in S(i_3)\setminus S(i_0)$, then $\sigma_D(p)<r$ (by Lemma \ref{Lemmamt14}). Thus 
$$
T(i_2)=\{p\in S(i_2)\mid \sigma_D(p)=r\}=T(i_0).
$$
We may thus repeat the above construction for some $q\in T(i_2)$ and $t$ such that $I_q\mathcal O_{Y_{i_3},\zeta}$ is not principal, where $\zeta$ is the generic point of 
$C(Y_{i_3},\nu_q^t)$. After iterating this procedure a finite number of times, we will construct an admissible sequence $Y_{i_4}\rightarrow Y_{i_0}$
such that $\sigma(Y{i_4})\le r$, $\sigma(Y_{i_4}\setminus W(i_4))<r$, 
$$
T(i_4)=\{p\in S(i_4)\mid \sigma_D(p)=r\}=T(i_0),
$$
and for all $p\in T(i_4)$, and $t$, $I_p\mathcal O_{Y_{i_4},\eta}$ is  principal, where $\eta$ is the generic point of 
$C(Y_{i_4},\nu_p^t)$.

Now perform a sequence of blow ups of 2-curves for $D$ $Y_{i_5}\rightarrow Y_{i_4}$, so that
$W(Y_{i_5},p)\subset \mbox{Preimage}(Y_{i_5},p)$ for all $p\in T(i_5)=T(i_0)$.  
Finally, we may construct an admissible sequence $Y_{i_6}\rightarrow Y_{i_5}$ consisting of transformations of type 3,  so that $W(i_6)=\emptyset$
for all $p\in T(i_6)=T(i_0)$. We may then apply a transformation of type 7, $Y_{i_7}\rightarrow Y_{i_6}$, defined by
$Y_{i_7}=Y_{i_6}$ and $S(i_7)=S(i_6)\setminus T(i_0)$ to obtain that $\sigma(Y_{i_7})\le r-1$.
\end{proof}

\vskip .2truein

Now we prove Theorem \ref{maintheorem}, by starting with $Y_0=X$ and $S(0)=\emptyset$. After applying successively Propositions \ref{Propmt1} and then \ref{Propmt2} enough times, we construct an admissible sequence $Y_n\rightarrow X$ such that 
$\sigma(Y_n)=0$, so that $S(Y_n)=\emptyset$, and $\sigma_D(p)=0$ for $p\in Y_n$.

\section{Proof of Toroidalization}\label{Section6}

\begin{Theorem}\label{TheoremA} Suppose that $\mathfrak k$ is an algebraically closed field of characteristic zero, and $f:X\rightarrow S$ is a dominant morphism from a nonsingular 3-fold over $\mathfrak k$ to a nonsingular 
surface $S$ over $\mathfrak k$ and $D_S$ is a reduced SNC divisor on $S$ such that $D_X=f^{-1}(D_S)_{\mbox{red}}$ is a SNC divisor on $X$ which contains the locus
where $f$ is not smooth. Further suppose that $f$ is 1-prepared. Then there exists a sequence of blow ups of points and nonsingular curves
$\pi_2:X_1\rightarrow X$, which are contained in the preimage of $D_X$, such that the induced morphism $f_1:X_1\rightarrow S$ is  prepared with respect to $D_S$.
\end{Theorem}

\begin{proof} The proof is immediate from Lemma \ref{1-prep}, Proposition \ref{Step2} and Theorem \ref{maintheorem}.
\end{proof}

Theorem \ref{TheoremA} is a slight restatement of Theorem 17.3 of \cite{C3}. Theorem 17.3 \cite{C3} easily follows from Lemma \ref{1-prep} and Theorem \ref{TheoremA}.

\begin{Theorem}\label{TheoremB} Suppose that $\mathfrak k$ is an algebraically closed field of characteristic zero, and $f:X\rightarrow S$ is a dominant morphism from a nonsingular 3-fold over $\mathfrak k$ to a nonsingular 
surface $S$ over $\mathfrak k$ and $D_S$ is a reduced SNC divisor on $S$ such that $D_X=f^{-1}(D_S)_{\mbox{red}}$ is a SNC divisor on $X$ which contains the locus where $f$ is not smooth. Then there exists a sequence of blow ups of points and nonsingular curves
$\pi_2:X_1\rightarrow X$, which are contained in the preimage of $D_X$, and a sequence of blow ups of points
$\pi_1:S_1\rightarrow S$ which are in the preimage of $D_S$, such that the induced rational  map $f_1:X_1\rightarrow S_1$ is a morphism which is toroidal with respect to $D_{S_1}=\pi_1^{-1}(D_S)$.
\end{Theorem}

\begin{proof} The proof follows immediately from Theorem \ref{TheoremA}, and Theorems 18.19, 19.9 and 19.10 of \cite{C3}.
\end{proof}

Theorem \ref{TheoremB} is a slight restatement of Theorem 19.11 of \cite{C1}. Theorem 19.11 \cite{C3} easily follows from Theorem \ref{TheoremB}.

\end{document}